\documentclass[a4paper,reqno]{amsart}
\usepackage[width=170mm,height=250mm]{geometry}
\usepackage{url}
\usepackage{amssymb}
\usepackage{mathtools}
\usepackage{paralist}
\numberwithin{equation}{section}
\usepackage{etoolbox}
\patchcmd{\section}{\scshape}{\bfseries}{}{}
\makeatletter
\renewcommand{\@secnumfont}{\bfseries}
\makeatother
\usepackage[usenames,dvipsnames,x11names]{xcolor}

\usepackage{autonum}
\usepackage[notref,notcite,color,final]{showkeys}      
\definecolor{refkey}{named}{blue}
\definecolor{labelkey}{named}{blue}
\usepackage[format=hang,font=small,labelfont={rm,sc},textfont=sl,width=0.9\textwidth]{caption}
\usepackage{subcaption}
\usepackage[bottom]{footmisc}
\usepackage{graphicx}
\graphicspath{{./figures/},{./figures/propagation/D4/},{./figures/propagation/D6/},{./figures/propagation/Z3/}}
\makeatletter
\def\input@path{{./files/}}
\makeatother
\newtheoremstyle{mythm}{2ex}{2ex}{\itshape}{}{\normalfont\normalsize}{}{.5em}{\textsc{\thmname{#1}}\thmnumber{ #2}\thmnote{ (#3)}}
\newtheoremstyle{bfnote}%
{}{}%
{\itshape}{}%
{\bfseries}{.}%
{ }%
{\thmname{#1}\thmnumber{ #2}\thmnote{ (#3)}}
\theoremstyle{bfnote}
\newtheorem{lemma}{Lemma}
\newtheorem{prop}{Proposition}

\newtheorem{remark}{Remark}
\newcommand{\thmtitre}{}
\newtheorem*{boxthm}{\thmtitre}
\usepackage[backgroundcolor=blue!10]{todonotes}
\usepackage{lipsum}
\usepackage{upgreek}
\usepackage{bbold}

\newcommand{\Newcommand}[2]{\providecommand{#1}{}\renewcommand{#1}{#2}}

\makeatletter
\newcommand{\proofstep}[1]{%
  \par
  \addvspace{\smallskipamount}
  \textit{#1\@addpunct{.}}\enspace\ignorespaces
}
\newlength{\parindentsave}
\newcommand{\Nproofstep}[1]{%
\setlength{\parindent}{0pt}
  \par
  \addvspace{\smallskipamount}
  \textit{\noindent #1\@addpunct{.}}\enspace\ignorespaces
\setlength{\parindent}{\parindentsave}
}
\makeatother

\newlength{\kaka}
\newlength{\LSpace}

\providecommand{\ahref}[2]{}

\newcommand{\suite}[1][0ex]{\notag \\[#1] & \mbox{}\hspace{15pt}}

\newcommand{\Div}{\mbox{div}\,}

\newcommand{\bfna}{\boldsymbol{\nabla}}

\newcommand{\shtimes}{\hspace*{-0.1em}\times\hspace*{-0.1em}}

\newcommand{\shm}{\hspace*{-0.1em}-\hspace*{-0.1em}}
\newcommand{\shp}{\hspace*{-0.1em}+\hspace*{-0.1em}}
\newcommand{\sheq}{\hspace*{-0.1em}=\hspace*{-0.1em}}
\newcommand{\shdeq}{\hspace*{-0.1em}:=\hspace*{-0.1em}}

\newcommand{\shneq}{\hspace*{-0.1em}\not=\hspace*{-0.1em}}

\newcommand{\shg}{\hspace*{-0.1em}>\hspace*{-0.1em}}
\newcommand{\shleq}{\hspace*{-0.1em}\leq\hspace*{-0.1em}}
\newcommand{\shgeq}{\hspace*{-0.1em}\geq\hspace*{-0.1em}}

\newcommand{\shgg}{\hspace*{-0.1em}\gg\hspace*{-0.1em}}
\newcommand{\shin}{\hspace*{-0.1em}\in\hspace*{-0.1em}}

\newcommand{\del}[1][]{\partial_{#1}}

\newcommand{\tens}{\hspace*{-1pt}\otimes\hspace*{-1pt}}

\newcommand{\pinv}[1]{\tfrac{1}{#1}}
\newcommand{\inv}[1]{\dfrac{1}{#1}}

\newcommand{\lbra}{\big\langle\hspace*{0.1em}}
\newcommand{\rbra}{\hspace*{0.1em}\big\rangle}

\newcommand{\lsqb}{\big[\hspace*{0.1em}}
\newcommand{\rsqb}{\hspace*{0.1em}\big]}
\newcommand{\lcb}{\big\{\hspace*{0.1em}}
\newcommand{\rcb}{\hspace*{0.1em}\big\}}
\newcommand{\Lcb}{\Big\{\hspace*{0.1em}}
\newcommand{\Rcb}{\hspace*{0.1em}\Big\}}

\newcommand{\lpar}{\big(\hspace*{0.1em}}
\newcommand{\rpar}{\hspace*{0.1em}\big)}

\newcommand{\labs}{\big|\hspace*{0.1em}}
\newcommand{\rabs}{\hspace*{0.1em}\big|}

\newcommand{\Lpar}{\Big(\,}
\newcommand{\Rpar}{\,\Big)}

\newcommand{\bfze}{\mathbf{0}}

\newcommand{\OO}{\Omega}

\Newcommand{\Sp}{S_{p}}

\newcommand{\Oe}{\OO_{\varepsilon}}

\newcommand{\dtau}{\,\text{d}\tau}

\newcommand{\dt}{\,\text{d}t}

\newcommand{\dV}{\;\text{d}V}

\newcommand{\dxi}{\,\text{d}\xi}

\newcommand{\dx}{\,\text{d}x}
\newcommand{\dy}{\,\text{d}y}

\newcommand{\iT}{\int_{0}^{T}}

 \newlength{\meno}
 \newlength{\integrale}
 \newlength{\trattino}
 \newlength{\integralegrande}
 \newlength{\uguale}
 \newlength{\ugualegrande}

\newcommand{\Acal}{\mathcal{A}}

\newcommand{\Dcal}{\mathcal{D}}
\newcommand{\Ecal}{\mathcal{E}}
\newcommand{\Fcal}{\mathcal{F}}

\newcommand{\Hcal}{\mathcal{H}}
\newcommand{\Ical}{\mathcal{I}}

\newcommand{\Kcal}{\mathcal{K}}
\newcommand{\Lcal}{\mathcal{L}}

\newcommand{\Ocal}{\mathcal{O}}

\newcommand{\Qcal}{\mathcal{Q}}
\newcommand{\Rcal}{\mathcal{R}}
\newcommand{\Scal}{\mathcal{S}}

\newcommand{\Xcal}{\mathcal{X}}
\newcommand{\Ycal}{\mathcal{Y}}
\newcommand{\Zcal}{\mathcal{Z}}

\Newcommand{\PS}{\mbox{\boldmath $\mathcal{P}$}}

\renewcommand{\SS}{\ensuremath{\mbox{\boldmath $\mathcal{S}$}}}

\Newcommand{\NC} {{\rm NC}}

\Newcommand{\NF}{\rmN\Fsub}

\Newcommand{\NN}{\rmN\Nsub}

\newcommand{\dip} {\! :\!}
\newcommand{\sip} {\! \cdot\!}

\newcommand{\qip} {\!::\!}
\newcommand{\dotp}{\raisebox{1pt}{\hspace*{1pt}\scalebox{0.45}{$\bullet$}}\hspace*{1pt}}
\newcommand{\Cdot}{\raisebox{1pt}{\hspace*{1pt}\scalebox{0.6}{$\bullet$}}\hspace*{1pt}}

\newcommand{\tdemi} {\tfrac{1}{2}}

\newcommand{\tquart} {\tfrac{1}{4}}

\Newcommand{\br} {b_{\incsize}}

\newcommand{\incsize}{\eps}

\newcommand{\eps}{\varepsilon}

\newcommand{\oo}{\omega}

\newcommand{\esup}{^{\text{\scriptsize e}}}

\newcommand{\gsup}{^{\text{\scriptsize g}}}

\newcommand{\Hsup}{^{\text{\tiny H}}}

\newcommand{\ssup}{^{\text{s}}}
\newcommand{\Tsup}{^{\text{\tiny T}}}

\newcommand{\Esub}{_{\text{\tiny E}}}
\newcommand{\Fsub}{_{\text{\tiny F}}}

\newcommand{\isub}{_{\text{\scriptsize i}}}

\newcommand{\Nsub}{_{\text{\tiny N}}}

\newcommand{\Osub}{_{\text{\tiny O}}}

\newcommand{\ssub}{_{\text{s}}}

\newcommand{\msub}{_{\text{\scriptsize m}}}

\newcommand{\Max}{^{\text{max}}}

\Newcommand{\Ref}{^{\text{ref}}}

\Newcommand{\skew}{_{\text{skew}}}

\Newcommand{\Tr}{\text{Tr}}

\renewcommand{\Re}{\text{Re}}

\newcommand{\rmb} {\mathrm{b}}
\newcommand{\rmc} {\mathrm{c}}
\newcommand{\rmd} {\mathrm{d}}
\newcommand{\rme} {\mathrm{e}}

\newcommand{\rmi} {\mathrm{i}}
\newcommand{\rmk} {\mathrm{k}}

\newcommand{\rmx} {\mathrm{x}}

\newcommand{\rmA} {\mathrm{A}}
\newcommand{\rmB} {\mathrm{B}}
\newcommand{\rmC} {\mathrm{C}}
\newcommand{\rmD} {\mathrm{D}}

\newcommand{\rmI} {\mathrm{I}}

\newcommand{\rmK} {\mathrm{K}}

\newcommand{\rmM} {\mathrm{M}}
\newcommand{\rmN} {\mathrm{N}}
\newcommand{\rmO} {\mathrm{O}}

\newcommand{\rmV} {\mathrm{V}}

\newcommand{\rmX} {\mathrm{X}}

\newcommand{\rmZ} {\mathrm{Z}}

\newcommand{\sfB} {\mathsf{B}}
\newcommand{\sfC} {\mathsf{C}}

\newcommand{\sfQ} {\mathsf{Q}}
\newcommand{\sfR} {\mathsf{R}}

\newcommand{\bfuT} {\widetilde{\boldsymbol{u}}{}}

\newcommand{\bfvH} {\widehat{\boldsymbol{v}}{}}

\newcommand{\brb} {\mathbf{b}}
\newcommand{\brc} {\mathbf{c}}

\newcommand{\brf} {\mathbf{f}}
\newcommand{\brg} {\mathbf{g}}
\newcommand{\brh} {\mathbf{h}}

\newcommand{\brp} {\mathbf{p}}
\newcommand{\brq} {\mathbf{q}}

\newcommand{\brs} {\mathbf{s}}

\newcommand{\bru} {\mathbf{u}}
\newcommand{\brv} {\mathbf{v}}
\newcommand{\brw} {\mathbf{w}}
\newcommand{\brx} {\mathbf{x}}

\newcommand{\brz} {\mathbf{z}}

\newcommand{\brA} {\mathbf{A}}
\newcommand{\brB} {\mathbf{B}}
\newcommand{\brC} {\mathbf{C}}
\newcommand{\brD} {\mathbf{D}}

\newcommand{\brF} {\mathbf{F}}
\newcommand{\brG} {\mathbf{G}}
\newcommand{\brH} {\mathbf{H}}
\newcommand{\brI} {\mathbf{I}}
\newcommand{\brJ} {\mathbf{J}}
\newcommand{\brK} {\mathbf{K}}
\newcommand{\brL} {\mathbf{L}}
\newcommand{\brM} {\mathbf{M}}

\newcommand{\brO} {\mathbf{O}}

\newcommand{\brQ} {\mathbf{Q}}
\newcommand{\brR} {\mathbf{R}}

\newcommand{\brT} {\mathbf{T}}
\newcommand{\brU} {\mathbf{U}}
\newcommand{\brV} {\mathbf{V}}
\newcommand{\brW} {\mathbf{W}}
\newcommand{\brX} {\mathbf{X}}

\newcommand{\bsfF} {\boldsymbol{\mathsf{F}}}

\newcommand{\bsfU} {\boldsymbol{\mathsf{U}}}

\newcommand{\bfe} {\boldsymbol{e}}

\newcommand{\bfg} {\boldsymbol{g}}
\newcommand{\bfh} {\boldsymbol{h}}

\newcommand{\bfn} {\boldsymbol{n}}

\newcommand{\bfv} {\boldsymbol{v}}
\newcommand{\bfw} {\boldsymbol{w}}
\newcommand{\bfx} {\boldsymbol{x}}
\newcommand{\bfy} {\boldsymbol{y}}

\newcommand{\bfC} {\boldsymbol{C}}

\newcommand{\bfH} {\boldsymbol{H}}

\newcommand{\bfL} {\boldsymbol{L}}

\newcommand{\bfP} {\boldsymbol{P}}

\newcommand{\bfV} {\boldsymbol{V}}

\newcommand{\bruh} {\hat{\mathbf{u}}{}}

\newcommand{\bruH} {\widehat{\mathbf{u}}{}}

\newcommand{\bruT} {\widetilde{\mathbf{u}}{}}

\newcommand{\brwT} {\widetilde{\mathbf{w}}{}}

\newcommand{\FHat}{\widehat{F}}

\newcommand{\brzB} {\overline{\mathbf{z}}{}}

\newcommand{\Lbar}{\Lbar{}{}}

\newcommand{\Nbar}{\Nbar{}{}}
\newcommand{\Obar}{\Obar{}{}}

\newcommand{\Rbar}{\Rbar{}{}}

\newcommand{\Tbar}{\Tbar{}{}}

\Newcommand{\hbar}{\bar{h}{}}

\Newcommand{\Bbb} {\mathbb{B}}
\newcommand{\Cbb} {\mathbb{C}}

\newcommand{\Rbb} {\mathbb{R}}

\newcommand{\bfepsb}{\bar{\bfeps}}

\newcommand{\bfetaB}{\overline{\bfeta}}

\newcommand{\bfchih}{\hat{\bfchi}{}}

\newcommand{\bfDel}{\boldsymbol{\Delta}}

\newcommand{\bfpi}{\boldsymbol{\pi}}

\newcommand{\bfchi}{\boldsymbol{\chi}}

\newcommand{\bfeps}{\boldsymbol{\eps}}
\Newcommand{\bfeta}{\boldsymbol{\eta}}
\newcommand{\bfga}{\boldsymbol{\gamma}}

\newcommand{\bfka}{\boldsymbol{\kappa}}

\newcommand{\bfphc}{\boldsymbol{\phi}}

\newcommand{\bfrhv}{\boldsymbol{\varrho}}
\newcommand{\bfsig}{\boldsymbol{\sigma}}
\newcommand{\bftau}{\boldsymbol{\tau}}

\newcommand{\bfxi}{\boldsymbol{\xi}}
\newcommand{\bfzet}{\boldsymbol{\zeta}}

\Newcommand{\therefore}{\hspace*{0pt}\smash{\raisebox{-0.7ex}{\scalebox{1.4}{$\begin{gathered}\cdot\\[-3.1ex]\cdot\hspace*{-0.6pt}\cdot\end{gathered}$}}}\hspace*{1pt}}
\Newcommand{\dx}{\,\text{dx}}
\newcommand{\ie}{\emph{i.e.}~}
\newcommand{\eg}{\emph{e.g.}~}

\renewcommand{\eps}{\ell}
\renewcommand{\bfeps}{\boldsymbol{\varepsilon}} 

\newcommand{\Lii}{\macrolength^2}

\newcommand{\thZ} {\mathring{\vartheta}}
\newcommand{\eZ} {\mathring{e}}
\newcommand{\muZ} {\mathring{\mu}{}}
\newcommand{\Ath} {\rmD^2}
\newcommand{\brAth} {\brD^2}
\newcommand{\brAZii} {\mathring{\brA}\!^2}

\Newcommand{\Oe} {\rmO^{\epsilon}}
\newcommand{\iY}{\int_{\Ycal}}
\newcommand{\iRd}{\int_{\Rbb^d}}

\newcommand{\CZ} {\rmC^0}
\newcommand{\bfCz} {\brC^0}
\newcommand{\brCz} {\brC^0}
\newcommand{\Ci} {\rmC^1}

\newcommand{\brCi} {\brC^1}
\newcommand{\rmcii} {\rmc^2}

\newcommand{\Cii} {\rmC^2}
\newcommand{\Bii} {\rmB^2}
\newcommand{\cii} {\upbeta^2}
\newcommand{\bfcii} {\brc^2}
\newcommand{\brcii} {\boldsymbol{\upbeta}^{\hspace*{-1pt}2}}
\newcommand{\bfCii} {\brC^2}
\newcommand{\brCii} {\brC^2}
\newcommand{\brAii} {\brA\!^2}
\newcommand{\brJii} {\brJ^2}
\newcommand{\bsh} {\rmB^1}
\newcommand{\bfbsh} {\brB^1}
\newcommand{\bii} {\rmb^2}
\newcommand{\bfbii} {\brb^2}
\newcommand{\bfBii} {\brB^2}
\newcommand{\rhoz} {\rhoSG^0}

\newcommand{\brMsh}{\brM^\sharp}
\newcommand{\brKsh}{\brK^\sharp}

\newcommand{\Hper} {H^1_{\text{per}}}
\newcommand{\bHzm} {\bfH^1_{\#}}
\newcommand{\bHper} {\bfH^1_{\text{per}}}

\newcommand{\bSz}{\SS^0{}}
\newcommand{\bSi}{\SS^1{}}
\newcommand{\bSii}{\SS^2{}}

\newcommand{\brFth} {\brF_{\vartheta}}
\newcommand{\rmFth} {\Fcal_{\vartheta}}
\newcommand{\brFstrth} {\brFstr_{\vartheta}}

\newcommand{\rmQZth} {\rmQZ_{\vartheta}}
\newcommand{\rmQZ} {\mathring{\Qcal}}

\newcommand{\rmQth} {\Qcal_{\vartheta}}
\newcommand{\rmRth} {\Rcal_{\vartheta}}

\newcommand{\sfQZ} {\mathring{\sfQ}{}}
\newcommand{\sfQZth} {\sfQZ_{\vartheta}}
\newcommand{\sfQth} {\sfQ_{\vartheta}}
\newcommand{\sfRth} {\sfR_{\vartheta}}

\newcommand{\stens}{\tens\ssup}
\newcommand{\bfnax}{\bfna_{\!\mathrm{x}}}
\newcommand{\bfnay}{\bfna_{\!y}}
\renewcommand{\Div}{\mathrm{div}\,} 
\newcommand{\Divx}{\Div_{\!\mathrm{x}}}
\newcommand{\Divy}{\Div_{\!y}}
\newcommand{\bfupph}{\boldsymbol{\upvarphi}}

\newcommand{\macrolength}{L} 

\newcommand{\rhoSG}{\uprho}
\newcommand{\oostr}{\breve{\oo}{}}
\newcommand{\brFstr}{\breve{\brF{}}}

\newcommand{\bfnaX}{\bfna} 
\Newcommand{\bfnax}{\bfna_\mathrm{x}} 
\Newcommand{\bfnay}{\bfna_y} 

\newcommand{\bfepsx}{\bfeps_\mathrm{x}} 
\newcommand{\bfepsy}{\bfeps_y} 

\newcommand{\bfetax}{\bfeta_\mathrm{x}} 

\newcommand{\bfkax}{\bfka_\mathrm{x}} 

\newcommand{\bfDelx}{\bfDel_\mathrm{x}} 

\newcommand{\brze}{\mathbf{0}}

\newcommand{\SGvt}{SG($\vartheta$) }
\newcommand{\SGz}{SG($0$) }

\begin{document}
\title[Well-posed homogenized strain-gradient models for linear elastodynamics]{Well-posed homogenized strain-gradient models for linear elastodynamics and elastostatics in arbitrary periodic media}
\date{Draft, \today}
\author{R\'emi Cornaggia}
\address{Institut Jean Le Rond $\partial $'Alembert, Sorbonne universit\'e, CNRS UMR 7190, 75252 Paris, France}
\email{remi.cornaggia@sorbonne-universite.fr}
\author{Marc Bonnet}
\address{POEMS (CNRS-INRIA-ENSTA), Dept. of Applied Mathematics, ENSTA Paris, Palaiseau, France}
\email{mbonnet@ensta.fr}
\author{Giuseppe Rosi}
\address{Laboratoire MSME, Univ Paris Est Creteil, CNRS UMR 8208, F-94010 Creteil, France}
\email{giuseppe.rosi@u-pec.fr}
\thanks{See last page for author affiliation information}
\author{Saad El Ouafa}
\address{Laboratoire MSME, Univ Gustave Eiffel, CNRS UMR 8208, F-77474 Marne-la-Vall\'ee, France}
\email{saad.elouafa@u-pem.fr}
\author{Nicolas Auffray}
\address{Institut Jean Le Rond $\partial $'Alembert, Sorbonne universit\'e, CNRS UMR 7190, 75252 Paris, France}
\email{nicolas.auffray@sorbonne-universite.fr}

\maketitle

\begin{abstract}
This work develops well-posed homogenized strain-gradient models for linear elastostatics and elastodynamics in periodic media, with a primary focus on elastic wave propagation. Using the classical two-scale asymptotic expansion method, we carry out second-order periodic homogenization for media in $\Rbb^d$ ($d = 2, 3$), with no restriction on the periodicity cell geometry or material distribution. Reciprocity identities applied to suitably chosen pairs of cell solutions provide alternative expressions for the effective stiffness and inertia tensors arising at the leading, first and second orders, substantially reducing the number of cell problems that must actually be solved. Since direct two-scale homogenization beyond leading order generically yields ill-posed effective operators, a Boussinesq-trick procedure is introduced, involving a tunable scalar weight, to recast the resulting fourth-order partial differential equation as a valid strain-gradient elasticity (SGE) model possessing the requisite symmetry, sign-definiteness and coercivity properties. These properties are then used, via the Hille-Yosida theorem, to establish the well-posedness of the corresponding transient initial-value and forced-response problems. Several practically relevant special cases are examined, including centrosymmetric cells, homogeneous mass density and homogeneous elasticity, each yielding simplified model structures. Numerical illustrations on three two-dimensional periodicity cells (square, hexagonal and a non-centrosymmetric chiral lattice) compare the resulting dispersion relations against reference Floquet-Bloch computations and assess transient wave propagation, demonstrating the model's capacity to capture anisotropic and dispersive effects beyond classical elasticity while preserving mathematical well-posedness.
\end{abstract}

\small
\tableofcontents
\normalsize

\section{Introduction}

All materials intrinsically possess structures at multiple length scales, from atomic arrangements to macroscopic geometries. In classical materials, these scales are typically well separated: microscale heterogeneities average out, rendering the material macroscopically homogeneous under typical loading conditions and sample geometries. In contrast, architectured materials are built to reduce this separation of scales, enabling interactions across them that can produce exotic emergent properties, such as size effects, tunable anisotropy \cite{RA16}, wave-beaming \cite{rosi:19} or elastic-wave polarisation \cite{rosi:24}. Nature frequently exploits this multiscale coupling as an optimization strategy, designing hierarchical architectures that enhance mechanical performance, damage tolerance, or multifunctionality beyond what homogeneous materials can achieve \cite{Fra07}. If these effects are present in both static and dynamic behavior, their consequences are significantly more pronounced in the latter case. When the mesostructure is non-centrosymmetric, new mechanical couplings arise — such as tension–bending coupling in the static regime \cite{poncelet:18} and the emergence of circularly polarized elastic waves \cite{toupin:62,rosi:24}, among others.
These couplings have been studied in physics \cite{Por68,divincenzo1986dispersive,maranganti:07}, but remain relatively unexplored within the mechanics community. When present, these effects play a predominant role in the effective response of the architectured material.
The scientific field that primarily exploits these properties is traditionally known as the field of elastic metamaterials \cite{craster:23}.

However, as the unique properties of architectured materials arise from their complex inner geometries, accurately simulating large specimens becomes particularly challenging. Capturing fine-scale structural features requires significant mesh refinement in numerical models, leading to high computational costs and making direct simulations of large-scale components overwhelming. One possible strategy to mitigate this challenge is to replace the explicitly heterogeneous architectured material with an effective homogeneous medium that somehow accounts for the microscopic behavior while significantly reducing computational costs \cite{geers:10}.
An effective homogeneous medium can  be interpreted as a filter that preserves only the dominant features of the mesostructure’s behavior. By distinguishing essential information from non-essential details, such an approach helps elucidate the relationship between the geometry of the mesostructure and the resulting macroscopic response. Consequently, it provides valuable insight into the underlying physics of the problem and serves as a powerful tool for guiding the optimal design of mesostructures\footnote{This idea is, in essence, already present in Pierre Curie’s seminal works, which established a connection between crystal symmetries and the presence or absence of piezoelectric coupling \cite{curie:1894}}.
By leveraging this concept, the use of an effective homogeneous continuum paves the way for an efficient computational framework for the optimal design of mesostructures. Through topological optimization techniques, it becomes possible to de-homogenize an effective behavior in order to identify the microstructure(s) responsible for it \cite{amstutz:10,corn:bell:20,mou:desmorat:23}. 

When the characteristic scales are weakly separated, the effective behavior is strongly influenced by the mesostructure organization. Consequently, the equivalent homogeneous medium to be considered must have an enriched kinematics in order to describe mesostructure effects at the continuous level. Based on phenomenological considerations, several generalized continuum models have been proposed in the literature to capture the specific effects arising from microstructural organization. These models can be broadly categorized into two main families: \textit{higher-order models} and \emph{higher-grade continua}. 
The higher-order models introduce additional degrees of freedom; they are commonly referred to as micromorphic continua and are characterized by the nature of the supplementary kinematical variables \cite{germain:73, eringen:12}. The classical Cosserat model \cite{cosserat1909} belongs to this category, as it accounts for independent material rotations. Such media are particularly well suited for the effective description of band-gap phenomena. These models are mentioned here for completeness and will not be used in this work, as we elect to adopt the framework of higher-grade continua. Models in this second family retain the displacement field as the fundamental degree of freedom while incorporating its higher-order gradients into the constitutive relations. It is worth noting that these models can be obtained from higher-order models by imposing suitable kinematical constraints. Higher-grade continua include the strain-gradient models, first introduced by Mindlin \cite{mindlin:64} and extensively investigated since then.   
Such formulations typically introduce internal length scales through higher-order constitutive tensors—namely, higher-order stiffness and inertia tensors (hereafter referred to as effective tensors). It is worth noting, however, that while the properties of higher-order stiffness tensors have been investigated \cite{auffray:15,auffray:19}, those of higher-order inertia tensors have received comparatively little attention. 

The main practical issue with either generalized-continuum modeling approach is the determination of the parameters of the resulting effective model. Two main classes of approaches can be distinguished, depending on whether one proceeds by:
\begin{compactitem}
    \item \textbf{identification}: the coefficients of the apparent model are inferred through an inverse approach, relying on physically relevant observations of the full behavior. This constitutes a phenomenological approach;
    \item \textbf{homogenization}: the coefficients of the effective model are derived from the description of the mechanical behavior of the elementary cell at the mesoscopic scale. This constitutes a multiscale approach.
\end{compactitem}
In the above classification, distinct "apparent" and "effective" qualifiers are deliberately used for the resulting media, to emphasize the differences between the two approaches and the nature of the obtained model.

In the case of wave propagation in anisotropic architectured media, some of our previous work focused on modeling the phenomenon using apparent strain-gradient continua \cite{RA16,rosi:18,rosi:19}. This approach has qualitatively highlighted the relevance of such a framework. While it provides a quantitatively improved description of macroscopic wave propagation, the determination of the model coefficients is typically performed heuristically—by fitting to dispersion diagrams—rather than through a rigorous homogenization procedure. In some cases, the identification of the second-order operator remains partial, leading to an incomplete gradient model \cite{rosi:24b}.
Phenomenological approaches based on parameter identification nevertheless offer the major advantages of simplicity in their implementation and practical efficiency. The resulting corrections significantly enhance the effective description compared to classical homogenized models \cite{rosi:18}. On the other hand, these approaches are essentially ad hoc, and no clear relationship exists between the geometry of the mesostructure and the values of the identified effective tensors. Moreover, as already mentioned, the identification of higher-order operators is often very limited. From a fundamental standpoint, such approaches are therefore not fully satisfactory, and a more comprehensive, systematic, and rigorous framework is desirable.

Mathematical approaches based on homogenization provide a rigorous framework to establish the transition between scales. They are well established and theoretically sound in the case of classical leading-order elastic homogenization. By contrast, their extension to generalized continua remains a partially open problem and an active area of research. Indeed, as will be discussed in more detail below, a straightforward extension of these methods often leads to effective operators that result in ill-posed boundary- or initial-value problems.



The vocable "homogenization" encompasses a wide range of distinct mathematical approaches, including formal asymptotic expansions, variational–asymptotic methods, and $\Gamma$-convergence; this work focuses on the first type of approach.
The ratio $\epsilon$ between the periodicity size and a larger macroscopic scale of interest (typically, the characteristic variation length of wavelengths in dynamics or external loads in statics) is used as a small asymptotic parameter. The most classical approach then relies on two-scale formal expansions of the fields quantities (displacement, strain, stress) and balance equations in terms of $\epsilon$, see
\cite{blp:78,SanchezPalenciaBook1987,BakhvalovBook1989,ciora:donato:99}. These expansions lead to a
decomposition of the fields at each order in $\epsilon$ into sums of (i) slowly varying ``macroscopic'' components and (ii)
oscillating \emph{correctors} written in terms of \emph{cell functions} solving elementary problems on a representative
cell of the microstructure. Moreover, the procedure provides \emph{effective equations} satisfied by the macroscopic
fields, involving effective tensors given explicitly in terms of the cell functions.

Pushing the expansion to higher orders in $\epsilon$ enable to reach ``weak'' scale separation regimes and to describe relevant
effects that appear at these regimes \eg non-negligible second-order ``gradient'' effects in statics (notably in the
case of a degenerated leading-order mode  \cite{durand:22}) or mode
conversion, dispersion and Rayleigh effects for elastic waves \cite{boutin:96} (at first, second and third order
respectively). Some authors even push the expansion to (very) high orders to include band-gaps and optic branches of the
dispersion diagram in the effective model \cite{hu:oskay:19}.
In some cases, this formal process can be rigorously justified by providing \emph{a posteriori} asymptotic
error estimates. This is classical for expansions truncated at the leading-order \cite{ciora:donato:99}, but much
more involved for higher-order models \cite{Allaire2022,allaire:24}. In these cases, the effective balance equation must also be supplemented by non-trivial
boundary correctors to address scattering or boundary-value problems, see \cite{cakoni:19} for antiplane shear waves or \cite{lin:meng:19} for elastodynamics.

Alternatives to this classical two-scale expansion approach include the variational-asymptotic approach introduced by
\cite{smysh:00}. The formal asymptotic expansions of fields are introduced in the energy of the system \ie the
space of kinematically admissible displacements is reduced to the subspace of two-scale expansions up to a given order in the variational formulation. It has then be applied to elastodynamics \cite{bacigalupo:14} and to
the design of material exhibiting strong strain-gradient effects \cite{durand:22,calisti:23}. Recently, a general
energy approach was proposed in \cite{Audoly:23,thbaut:25}, that relies on the variational formulation of static,
microstructured equilibrium problems from the very beginning of the asymptotic process.

All these approaches provide a leading-order homogenized model that features a uniquely defined effective elasticity
tensor (``the'' effective elasticity tensor) satisfying all the expected properties \ie minor and major symmetries and
positivity. By contrast, when pushing formal expansions to higher-order, one obtains a \emph{family} of models with the
same asymptotic accuracy (rather than a unique model), that \emph{may include ill-posed members}; in particular, higher-order effective tensors may have the wrong sign, see \eg the discussion of \cite{durand:22} for the second-order strain-gradient tensor in elastostatics. This issue is addressed in depth for spatially one-dimensional settings, see \cite{waut:guz:14,corn:lomb:23}, and for scalar waves (modeling \eg acoustics or antiplane 2D elasticity), see \cite{Allaire2016,abdulle:20,Allaire2022}. To choose a well-posed member
of such a family, \emph{reciprocity relations} between cell problems at various orders enable to exhibit relations
between effective tensors, and so-called \emph{Boussinesq tricks} are used to build well-posed models
\cite{Allaire2016,abdulle:20}. Recent alternatives were also proposed by \cite{Allaire2022,allaire:24} who (i) systematically remove all higher-order time derivatives from effective models and (ii) filter the source term to remove potentially problematic higher frequencies.

For full (2D or 3D) elasticity, reciprocity relationships are derived for the first-order elasticity tensor in
\cite{boutin:93,lin:meng:19}, and for the second-order in \cite{fish:02}. To the best of our knowledge they are not used to address the
aforementioned issues, and no similar relationship is available for the inertial tensors. Concerning the ``wrong'' sign of the full strain-gradient tensor in elastostatic, original remedies were recently proposed: \cite{durand:22} decomposes the second-order elasticity tensor onto its eigenspaces and retains only the contributions associated with eigenvalues with the ``good'' sign, while \cite{thbaut:25} presents a systematic energy-based, Cholesky-inspired procedure to obtain a positive homogenized energy defining a well-posed variational problem. However, elastodynamics is much less explored, and available analyses are often restricted to \emph{centrosymmetric} periodicity cells \cite{durand:22,calisti:23}, which considerably simplifies the homogenized models by canceling all odd-order contributions in the effective equation \cite{bacigalupo:14,rosi:19}.
For non-centrosymmetric cells, however, these odd-order terms are a distinctive feature of 2D or 3D elasticity
\cite{lin:meng:19,allaire:24} compared to \eg acoustics or antiplane 2D elasticity where they vanish for any microstructure \cite{Allaire2016,Allaire2022}, and they enable the description of the unusual mechanical coupling quoted above.

In this context, the present work revisits the ``classical'' homogenization methodology for both elastostatics and elastodynamics, with a primary focus on the latter case. Following \cite{gambin:89,boutin:93,abdulle:20}, we construct a physically-valid and mathematically well-behaved strain-gradient effective model that is completely determined by, and computable from, the microstructure. The  intended contributions of this work are as follows:
\begin{compactenum}[\hspace*{0.5em} 1.]
\item We carry out the two-scale periodic homogenization methodology for both elastostatics and elastodynamics beyond the leading order, for periodic elastic media in $\Rbb^d$ ($d=2,3$). To maximize generality, we assume only geometric and material periodicity in space (with a spatially-bounded periodicity cell), with the minimal conditions on the material parameters that ensure physical consistency and well-defined variational formulations on the periodicity cell. We then derive by means of a Boussinesq-trick approach a family of enriched fourth-order effective partial differential equations (PDEs) involving a tunable scalar weight $\vartheta$, and show that suitable values of $\vartheta$ exist for which the obtained PDE constitutes a valid strain gradient elasticity (SGE) model having all the requisite symmetry, sign and coercivity properties.
\item The above established properties in turn (i) produce dispersion equations with matrices having the requisite sign and symmetry properties, and (ii) allow to prove, by verifying that the Hille-Yosida theorem applies, the well-posedness of transient initial-value or forced-response elastodynamic problems governed by the second-order in time effective SGE-type PDE with suitable $\vartheta$.
\item The foregoing results rest to a significant extent on alternative expressions of the effective stiffness and inertial tensors arising at the leading, first and second orders, obtained through systematic application of a reciprocity identity to a number of suitably chosen pairs of cell solutions. From a theoretical standpoint, those alternative expressions make the symmetry or skew-symmetry properties of the effective tensors explicit and help in assessing the sign properties of the various terms in the effective PDE. They moreover bring significant practical advantage by allowing a substantial reduction of the number of cell problems whose solution is actually needed for setting up the complete strain-gradient PDE.
\item The proof method for the requisite coercivity and sign properties provides explicit sufficient conditions on the weight $\vartheta$ for the homogenized SGE model to be associated with positive strain and kinetic energy densities, which are easily implementable (by solving low-dimensional eigenvalue problems involving raw effective tensors), making the practical calibration of $\vartheta$ easy and effective;
\item Among its outcomes, our approach sheds a new light on the dependence of the strain-gradient tensors to the elastic and inertial properties of the microstructure for elastodynamics, in particular on those arising at the first order for non-centrosymmetric cells.
\item The main aspects of the resulting homogenized elastodynamic SGE model are implemented and demonstrated on numerical examples under 2D plane-strain conditions.
\end{compactenum}

While, to repeat, this work is primarily aimed at modeling elastic waves in periodic media, we begin by fully addressing the derivation of a well-posed SGE elastostatic homogenized model, and then develop the elastodynamic version of this analysis and the resulting SGE model. We chose this two-stage development for several reasons. Firstly, the elastostatic high-order homogenization and its interpretation as a valid SGE model constitutes on its own a useful contribution to the literature. For instance, we propose a procedure to address the issue of positive definiteness of higher-order operators, raised in numerous other works \cite{maranganti:07,durand:22,thbaut:25}. Secondly, the elastostatic ingredients of the homogenization process are in any case involved in the elastodynamic analysis, and the chosen exposition is easier to follow without being unduly longer than directly treating the elastodynamic case. Thirdly, comparisons between the elastodynamic and elastostatic contexts are made easier.

Numerical experiments in 2D plane-strain elasticity are included to illustrate some of the key features of the methodology, in particular the (expected) second-order asymptotic accuracy of the
effective dispersion relations (through comparisons with reference Floquet-Bloch computations). A more-thorough numerical study,
fully exploiting the results of this work on 3D periodicity cells and the numerous architectures classified in \cite{elouafa}, is deferred to an upcoming separate paper.

This article is organized as follows. Section~\ref{sec:prelim} is devoted to specifying the class of microstructures addressed, summarizing the classical two-scale homogenization approach upon which this work is built and collecting available mathematical material for PDEs on periodicity cells. Then, Section~\ref{sec:homog:zf} addresses the second-order homogenization for elastostatics, the derivation of reciprocity-based expressions of the resulting homogenized stiffness tensors, and the subsequent derivation of a $\vartheta$-dependent family of effective equilibrium PDEs. Building on this foundation, Section~\ref{sec:homog} carries out elastodynamic second-order homogenization, yielding reciprocity-based expressions of the homogenized inertial tensors and a $\vartheta$-dependent family of effective PDEs. In Section~\ref{sec:homo:SGE}, we review the target strain-gradient elasticity (SGE) framework and  show that the foregoing effective PDEs define valid SGE models for both static and dynamic cases, its stiffness and inertial components being in particular proved to have the requisite sign definiteness and symmetry properties for suitably chosen $\vartheta$; the latter properties are in turn used to establish the well-posedness of transient problems in free space governed by the homogenized SGE model. Section~\ref{sec:special} next outlines the resulting practical homogenization procedure, discusses several particular situations as well as a fourth-order-in-time version of the effective PDE and presents the generalized Christoffel equation arising from the homogenized SGE model. Then, numerical experiments on dispersion analysis and transient propagation simulation under two-dimensional plane-strain conditions are presented in Section~\ref{sec:numerics}. Section \ref{sec:proofs} provides the proofs of several key results. Section~\ref{sec:conclusion_perspectives} finally discusses possible directions for future work.



\section{Problem setting and preliminaries}
\label{sec:prelim}

\subsection{Periodic propagation media}
\label{sec:periodic}

We consider $d$-dimensional unbounded periodic elastic media $\OO$, characterized by
\begin{enumerate}[\hspace{0.5em}1.]
    \item a chosen open periodicity cell $\Ycal_\ell = \ell \Ycal$, where $\ell$ is a characteristic length and $\Ycal$ a reference cell in non-dimensional coordinates (assumed for definiteness to be centered at the coordinate origin),
    \item a $d$-uple of non-dimensional basis vectors $\brv_1,\ldots,\brv_d$,
    \item given spatially-varying material parameters in $\Ycal$,
\end{enumerate}
and constituted by the union of all translated copies of the periodicity cell:
\begin{equation}
  \OO = \overline{\cup_{\bfV\in\Zcal} \ell(\Ycal\shp\bfV)}, \qquad
  \brV\shin\Zcal:=\lcb n_1\brv_1\shp\ldots\shp n_d\brv_d,\  (n_1,\ldots,n_d)\shin\mathbb{Z}^d \rcb.
\end{equation}
The countable set $\Zcal$ thus collects all lattice vectors generated by the basis $(\brv_1,\ldots,\brv_d)$, which themselves are chosen such that the translated copies $\Ycal+\brV,\ \brV\shin\Zcal$ of $\Ycal$ are disjoint while their closures cover the entire space $\Rbb^d$.
The constitutive material of $\Ycal$ is assumed to be linearly elastic and modeled as a Cauchy-type continuum, and the material parameters are therefore the mass density $\rho$ and elasticity tensor $\bfC$. The resulting periodic medium $\OO$ is full (i.e. $\OO=\Rbb^d$) if the whole cell $\Ycal$ is filled by material. Perforated periodic media $\OO\ssub:=\overline{\cup_{\bfV\in\Zcal} \ell(\Ycal\ssub\shp\bfV)}$, where the material occupies only a subset $\Ycal\ssub$ of $\Ycal$ and the hole boundaries $\del\OO\ssub$ are traction-free, also fall within the scope of this work; for such media, the non-void part $\Ycal\ssub$ of $\Ycal$ must be such that $\OO\ssub$ is connected. We illustrate the foregoing definitions with the example of a 2D periodic medium whose periodicity cell $\Ycal$ is the square $[-0.5, 0.5]^2$, as depicted in Figure~\ref{fig:wave:unit:cell}.\enlargethispage*{1ex}

\begin{figure}[b]
\centering
	\includegraphics[width=0.4\textwidth]{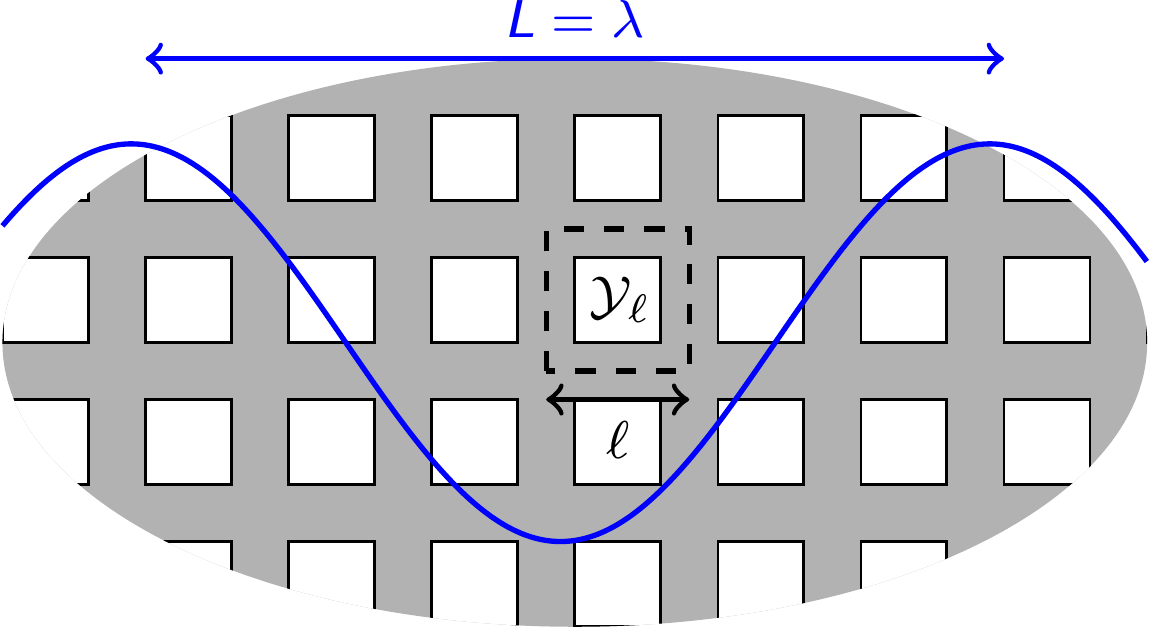}\hspace{0.1\textwidth}
    \includegraphics[width=0.22\textwidth]{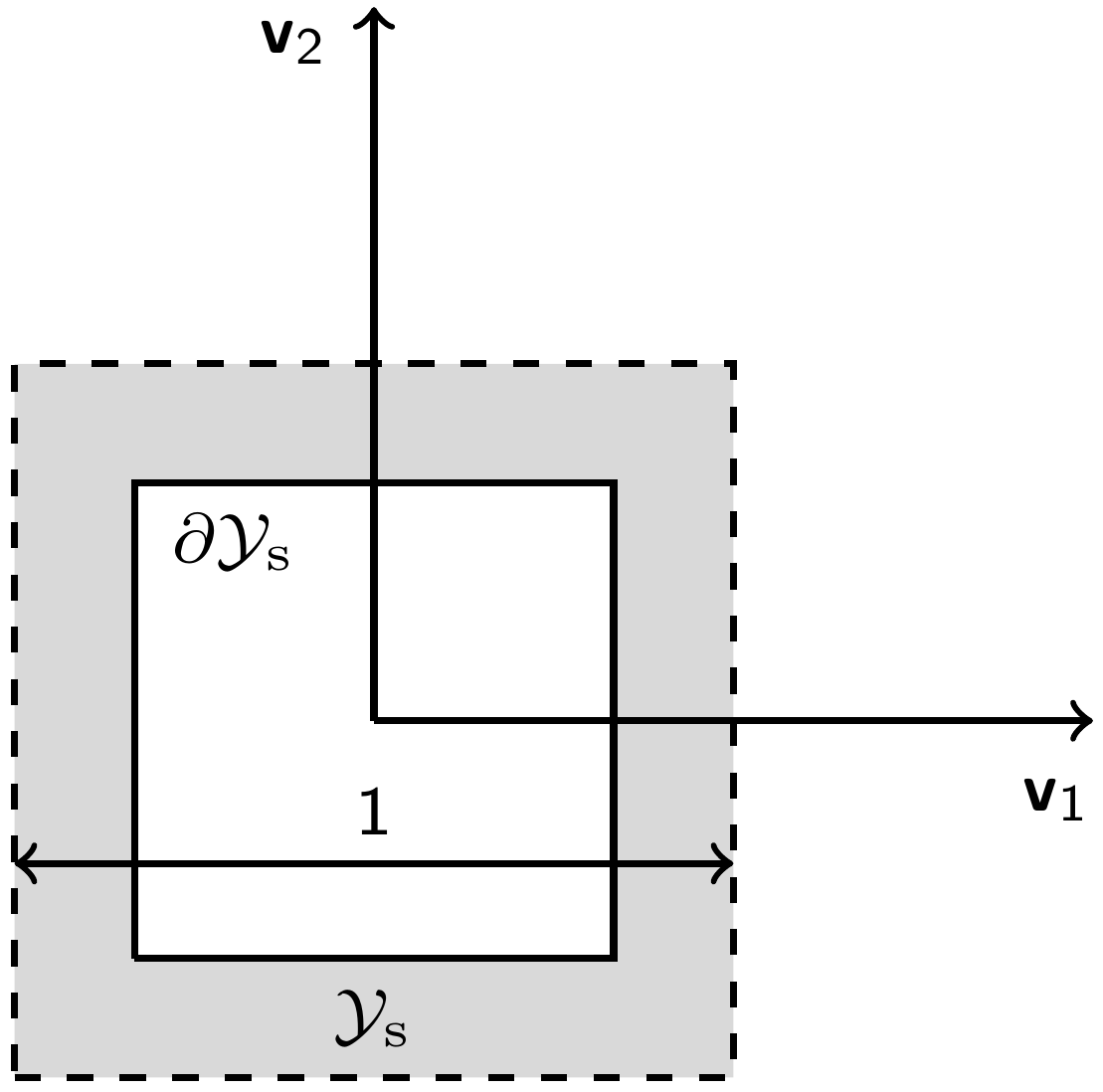}
	\caption{Sketch of an unbounded heterogeneous periodic medium and its periodicity cell $\Ycal$. Left: unbounded (perforated) domain $\Omega$, made of replica of the cell $\Ycal_\ell = \ell\Ycal$ and supporting the propagation of waves with large wavelength $\lambda$. Right: unit cell $\Ycal$ and associated lattice vectors $(\brv_1,\brv_2)$}
    \label{fig:wave:unit:cell}
\end{figure}

The position $\brX$ of a material point of the periodic medium can be expressed as
\begin{equation}
  \brX = \ell(\brV \shp \bfy), \label{slow:fast}
\end{equation}
where $\brV\in\Zcal$ is the lattice vector defining the center of the cell containing $\brX$ and $\bfy \in \Ycal$ is the non-dimensional position relative to that cell center. The material parameters $(\rho,\bfC)$ in $\Ycal$ are only assumed to satisfy the minimal requirements usually made for the variational formulation of elastodynamic problems, namely to be (respectively scalar and tensor-valued) $L^{\infty}(\Ycal)$ functions (i.e. to be essentially bounded in $\Ycal$) that in addition verify for some $\rho\msub,\kappa\msub\shg0$ the sign and ellipticity conditions
\begin{equation}
  \rho(\bfy) \geq \rho\msub, \qquad \bfeps\dip\bfC(\bfy)\dip\bfeps \geq \kappa\msub|\bfeps|^2 \quad \forall \bfeps\in S^{2}(\Rbb^d),\ \bfy\shin\Ycal. \label{sign:rho:C}
\end{equation}
where $S^{2}(\mathbb{V})$ indicates the space generated by the symmetrized tensor product of $\mathbb{V}$ with itself.

The medium $\OO$ is then made materially $\Ycal_{\ell}$-periodic by defining its elasticity tensor $\bfC_\ell$ and mass density $\rho_\ell$ by translating copies of $\rho$ and $\bfC$:
\begin{equation}
  \bfC_\ell(\brX) = \bfC_\ell\lpar \ell(\brV \shp \bfy) \rpar = \bfC(\bfy), \qquad
   \rho_\ell(\brX) = \rho_\ell\lpar \ell(\brV \shp \bfy) \rpar = \rho(\bfy) \qquad (\brV\shin\Zcal,\,\bfy\shin\Ycal) \label{coeff:xy}
\end{equation}


The equations governing the transient elastodynamic displacement field $\bru$ at time $t$ read
\begin{equation}
\begin{aligned}
    &\text{(a) \ }& \bfeps(\brX,t)
    &:= \pinv{2}\lsqb \bfnaX\bru(\brX,t) + \bfnaX\bru(\brX,t)\Tsup \rsqb =: \bfnaX\ssup\bru(\brX,t), \\
	&\text{(b) \ }& \bfsig(\brX,t) &= \bfC_\ell(\brX) \dip \bfeps(\brX,t), \\
	&\text{(c) \ }& \Div\bfsig(\brX,t) &+ \brf(\brX,t) = \rho_\ell(\brX) \bru''(\brX,t)
\end{aligned} \qquad\qquad \brX\shin\Rbb^d,\, t\shgeq0, \label{ch5}
\end{equation}
in which $\brf$ is a prescribed body force density.  In~\eqref{ch5} and hereafter, the prime symbol is used to denote time derivatives, and the gradient operator $\bfnaX$ conventionally adds an extra tensorial order ``to the right'' (so that e.g. the derivative of some tensor $\brT$ along a direction $\brb$ is given by $(\bfna\brT)\sip\brb$). Likewise, iterated gradients $\bfnaX^k$ add $k$ extra tensorial orders ``to the right'' and the divergence operator $\Div(\dotp)$ also acts ``to the right'' of tensors (so that $\Div\brT:=(\bfnaX\brT)\dip\brI$ in the present Cartesian orthonormal coordinate setting). On incorporating the constitutive equation~(\ref{ch5}b) and the compatibility equation~(\ref{ch5}a), the linear momentum balance equation~(\ref{ch5}c) becomes
\begin{equation}
  - \Div \lcb\bfC_\ell(\brX)\dip\bfnaX\ssup\bru(\brX,t) \rcb + \rho_\ell(\brX) \bru''(\brX,t) = \brf(\brX,t). \label{motion}
\end{equation}
Accordingly, applying to~\eqref{motion} the Fourier transform with respect to time yields the motion field equation in the frequency domain:
\begin{equation}
  - \Div \lcb\bfC_\ell(\brX) \dip \bfna\ssup\bruH(\brX,\oo) \rcb - \oo^2 \rho_\ell(\brX) \bruH(\brX,\oo) = \widehat{\brf}(\brX,\oo). \label{equation of motion}
\end{equation}
In the sequel, we drop hat symbols and use the same notation, e.g. $\bru$, for (real-valued) time-dependent fields and their (complex-valued) frequency-domain counterparts.

The corresponding elastostatic equations are then obtained by removing the time dependency in the time-domain equations, i.e. setting $\oo = 0$ in the time-harmonic equations (see Sec.~\ref{sec:mean}).\enlargethispage*{3ex}


\subsection{Two-scale formal asymptotic expansions}
\label{sec:twoscale}

We focus on problems featuring a characteristic length $\macrolength$ assumed to be significantly larger than the linear size $\ell$ of the periodicity cell; see Remark~\ref{rem:Lspec} about the specification of $L$. Scale separation then occurs as the ratio of those characteristic lengths is small, i.e.
\begin{equation}
  \epsilon := \ell / \macrolength \ll 1. \label{scale:sep}
\end{equation}
The scale separation assumption~\eqref{scale:sep} makes it natural, and customary, to apply periodic homogenization and express all fields in the form of two-scale asymptotic expansions in powers of the scale parameter $\epsilon$. For instance, the displacement field is sought in the form
\begin{equation}
\bru_{\epsilon}(\brX) = \sum_{n\geq0} \epsilon^n \bruT^n(\brx,\bfy) = \bruT^0(\brx,\bfy) + \epsilon\bruT^1(\brx,\bfy) +  \epsilon^2\bruT^2(\brx,\bfy) + \epsilon^{3}\ldots, \qquad \brx:=\macrolength^{-1}\brX\in\Rbb^d,\ \bfy\in\Ycal
\label{expansion}
\end{equation}
where the above-defined \emph{slow} variable $\brx$ is the material position measured in macroscopic length scale units, the \emph{fast} variable $\bfy$ (defined as in~\eqref{slow:fast}) is the local position in a cell and is measured in units of linear cell size, and each coefficient $\bruT^n(\brx,\bfy)$ is a function assumed to be $\Ycal$-periodic with respect to the fast variable $\bfy$.

In this study, we focus on the truncated ansatz
\begin{equation}
  \bru_{\epsilon}(\brX)
 = \bruT^0(\brx,\bfy) + \epsilon\bruT^1(\brx,\bfy) +  \epsilon^2\bruT^2(\brx,\bfy) + \epsilon^{3}\bruT^3(\brx,\bfy) + o(\epsilon^3).
 \label{expansion:4term}
\end{equation}
Introducing the mean displacement fields $\brU^n(\brx)$ at each order, defined by
\begin{equation}
   \brU^n(\brx) := \lbra \bruT^n(\brx,\Cdot) \rbra_{\Ycal} \label{meandisp:def}
\end{equation}
with $\lbra \dotp \rbra_{\Ycal}$ denoting the cell averaging operator for (scalar or tensor) functions of the fast coordinate defined by
\begin{equation}
    \lbra f \rbra_{\Ycal} := \inv{|\Ycal|}\iY f(\bfy) \dy, \label{mean:def}
\end{equation}
our main goals are then to formulate effective partial differential equations (PDEs) obeyed (up to a $o(\epsilon^2)$ residual error) by the second-order approximation
\begin{equation}\label{eq:Ueps:def}
	 \brU^{\epsilon}(\brx) := \brU^0(\brx) + \epsilon \brU^1(\brx)  + \epsilon^2 \brU^2(\brx)
\end{equation}
of the macroscopic displacement field, to determine their properties and to study their implementation. The body force density in~\eqref{equation of motion} is assumed in this work to depend only of the slow coordinate, i.e. to have the form
\begin{equation}
  \brf(\brX,t) = \brF(\brx,t). \label{F:slow}
\end{equation}
\begin{remark}\label{rem:Lspec} 
The macroscopic length $\macrolength$ introduced in~\eqref{scale:sep} may in particular be (i) the characteristic fluctuation length of a macroscopic load (such as the size of a sample subjected to smooth boundary loads), for elastostatics, or (ii) a characteristic wavelength of waves propagating in the medium, for elastodynamics. The asymptotic regime $\epsilon \ll 1$ may then reflect at least two physical viewpoints: (i) fixed macroscopic length scale $\macrolength$ and much smaller cell scale $\ell$ (\eg microstructures of vanishingly small characteristic size filling the same fixed macroscopic domain) or (ii) fixed cell scale $\ell$ (\ie given microstructured material) and much larger macroscopic length scale $\macrolength$ (typically wave propagation at frequencies low enough to induce wavelengths $\macrolength\shgg\ell$). We mainly follow viewpoint (ii) in this work, which is however readily adaptable to viewpoint (i).

The macroscopic length $\macrolength$ may (and in this work will) remain largely unspecified while deriving an approximate model valid for $\epsilon$ small. The process to follow, performed for an unbounded medium, uses $L$ to express scale separation and to write expansions in terms of the dimensionless parameter $\epsilon$; the resulting approximate models will be found to only involve $\ell$. As argued in \cite{boutin:93,boutin:96}, $\macrolength$ may then be \emph{a posteriori} given as the characteristic fluctuation length of the macroscopic field $\brU$ (not known a priori), \ie $\macrolength \sheq \Ocal(\min (\|\brU \|/ \|\bfna \brU \|))$, and ranges of $\epsilon \sheq \ell/L$ making the approximate model physically valid discussed on that basis. For harmonic waves, the fluctuation length is $\macrolength \sheq \Ocal(\lambda)$ with $\lambda$ a macroscopic wavelength, so that $\epsilon \sheq \ell/\lambda$.
\end{remark}

\subsection{Outline of homogenization procedure}
\label{sec:homproc}

Expressing $\bfy$ from~\eqref{slow:fast},  any field quantity $\brw(\brX)$ treated as a function $\brwT(\brx,\bfy)$ of the dimensionless slow and fast coordinates verifies $\brw(\brX) = \brwT\lpar \macrolength^{-1}\brX,\, \epsilon^{-1}\macrolength^{-1}\brX-\brV \rpar$, so that its spatial differentiation w.r.t. $\brX$, a generic operation involved in the governing field equations, becomes
\begin{equation}\label{grad:twoscale}
  \bfnaX \brw(\brX)  = \macrolength^{-1} \lpar \bfnax + \epsilon^{-1}\bfnay \rpar \,\brwT(\brx,\bfy).
\end{equation}
On substituting the representation~\eqref{coeff:xy} of the medium parameters together with expansion~\eqref{expansion}, using the differentiation rule \eqref{grad:twoscale},
the equation of motion~\eqref{equation of motion} becomes a dimensionless power series in $\epsilon$. Setting each coefficient of that expansion to zero yields the following cascade of PDE:
\begin{equation}
	\begin{aligned}
	&\Ocal(\epsilon^{-2}): \qquad&   \Lcal_{yy}\{ \bruT^0 \}  &= 0  \vspace{0.2cm} \\
	&\Ocal(\epsilon^{-1}): \qquad& \Lcal_{yy}\{ \bruT^1 \} + \Lcal_{xy}\{ \bruT^0 \} &= 0   \vspace{0.2cm}\\
	&\Ocal(\epsilon^{0 }): \qquad& \Lcal_{yy}\{ \bruT^2 \} + \Lcal_{xy}\{ \bruT^1 \} + \Lcal_{xx}\{ \bruT^0 \} &= \oostr^2 \rho\bruT^0 +  \brFstr    \vspace{0.2cm}\\
	&\Ocal(\epsilon^{1 }): \qquad& \Lcal_{yy}\{ \bruT^3 \} + \Lcal_{xy}\{ \bruT^2 \} + \Lcal_{xx}\{ \bruT^1 \} &= \oostr^2 \rho\bruT^1  \vspace{0.2cm} \\
	&\Ocal(\epsilon^{n }): \qquad& \Lcal_{yy}\{ \bruT^{n+2} \} + \Lcal_{xy}\{ \bruT^{n+1} \} + \Lcal_{xx}\{ \bruT^{n} \} &= \oostr^2 \rho  \bruT^n \quad,\quad ~n\geq 2,
	\end{aligned} \label{cascade epsilon}
\end{equation}
where we use the following operators:
\begin{equation}
\begin{aligned}
  \Lcal_{yy} (\Cdot)
 &:= -\Divy \lcb \bfC\dip\bfepsy(\Cdot) \rcb, \\
  \Lcal_{xy} (\Cdot)
 &:= -\Divx \lcb \bfC\dip\bfepsy(\Cdot) \rcb - \Divy \lcb \bfC\dip\bfepsx(\Cdot) \rcb, \\
  \Lcal_{xx} (\Cdot)
 &:= -\Divx \lcb \bfC\dip\bfepsx(\Cdot) \rcb
\end{aligned}
\end{equation}
in terms of the "fast" and "slow" strain operators \cite{gambin:89}
\begin{equation} \label{diffops:def}
  \bfepsy(\Cdot) := \bfnay\ssup (\Cdot), \quad  \bfepsx(\Cdot) := \bfnax\ssup (\Cdot),
\end{equation}
while the modified versions $\oostr$ and $\brFstr$ of the angular frequency and body force density, which appear as a result of using nondimensional coordinates scaled by $L$ in the field equation of motion (see Remark~\ref{rem:stretch}), are given by
\begin{equation}
  \oostr := L\oo, \qquad \brFstr := \Lii\brF.  \label{oo:F:stretch}
\end{equation}
For future reference, we also introduce the first and second-order ``slow'' strain gradients
\begin{equation} \label{diffops:def:eta:kappa}
  \bfetax(\Cdot) := \bfnax\bfepsx (\Cdot), \qquad \bfkax(\Cdot) := \bfnax\bfetax(\Cdot).
\end{equation}
The index symmetries of those tensors are determined by the structure of the tensor spaces they inhabit\footnote{For complete rigor, 
$\bfkax$ belongs to a subspace of 
$S^{2}(\Cbb^d)\otimes S^{2}(\Cbb^d)$
that is isomorphic to $\Cbb^d\otimes S^{3}(\Cbb^d)$. In the case $d=3$ the dimension of the former is 36, whereas that of the latter is 30.
The distinction between these two subspaces arises from whether the integrability conditions of  $\bfepsx$ are automatically enforced.
Although this remark is not essential to the subsequent developments, it is stated here for the sake of completeness and precision.}, i.e.
\begin{equation}
   \bfepsx\in S^{2}(\Cbb^d),\quad \bfetax\in S^{2}(\Cbb^d)\tens\Cbb^d,\quad \bfkax\in S^{2}(\Cbb^d)\tens S^{2}(\Cbb^d).
\end{equation}

In each equation of order $\Ocal(\epsilon^{n-2})$, the main unknown is the displacement field $\bruT^n$, the fields $\bruT^{n-1}$ and $\bruT^{n-2}$ also featured being determined by lower-order equations. Sequentially solving the above system therefore allows to determine the coefficients of the two-scale expansion \eqref{expansion}, and in particular to define the cell solutions (also known as correctors) featured in~\eqref{sep:var}, as explained in the next sections. Specifically, determining the truncated expansion~\eqref{expansion:4term} will entail solving the first four equations of~\eqref{cascade epsilon} and using the solvability condition for the fifth.
The coefficients $\bruT^n$ will as a result be found to have the separated-variable form
\begin{equation}
\begin{aligned}
  \bruT^0(\brx,\bfy) &= \brU^0(\brx) \\
  \bruT^1(\brx,\bfy) &= \brU^1(\brx) + \bfchi^1(\bfy)\dip\bfepsx\lpar\brU^0(\brx)\rpar \\
  \bruT^2(\brx,\bfy) &= \brU^2(\brx) + \bfchi^1(\bfy)\dip\bfepsx\lpar\brU^1(\brx)\rpar + \bfchi^2(\bfy)\therefore \bfetax\lpar\brU^0(\brx)\rpar + \oostr^2\rhoz\bfzet^2(\bfy)\sip\brU^0(\brx) \\
  \bruT^3(\brx,\bfy) &= \brU^3(\brx) + \bfchi^1(\bfy)\dip\bfepsx\lpar\brU^2(\brx)\rpar
  + \bfchi^2(\bfy)\therefore \bfetax\lpar\brU^1(\brx)\rpar
  + \bfchi^3(\bfy)\qip\bfkax\lpar\brU^0(\brx)\rpar \\
  & \quad + \oostr^2\rhoz \left[\bfzet^2(\bfy)\sip\brU^1(\brx) +  \bfzet^3(\bfy)\dip\bfnax\brU^0(\brx)\right]
\end{aligned} \label{sep:var}
\end{equation}
where $\rhoz = \lbra \rho \rbra_\Ycal$ is the effective mass density (as justified later); the three \emph{cell functions} $\bfchi^1, \bfchi^2, \bfchi^3$ and the two \emph{inertial cell functions} $\bfzet^2$, $\bfzet^3$, periodic and zero-mean over $\Ycal$, solve \emph{cell problems} (to be formulated); and the mean displacements $\brU^n$ verify macroscopic balance equations (also to be formulated).\enlargethispage*{5ex}

\begin{remark}
In anticipation of one of our main results, we will show that only two cell functions $\bfchi^1, \bfchi^2$ and one inertial function $\bfzet^2$ are ultimately needed to build an effective SGE governing model for the second-order approximation $\brU^{\epsilon}$ defined by~ \eqref{eq:Ueps:def}. Accordingly, $\bruT^3(\brx,\bfy)$ will not be explicitly involved in that final model, as suggested by definition~\eqref{eq:Ueps:def}, although the equations satisfied by $\bruT^3(\brx,\bfy)$, $\bfchi^3$ and $\bfzet^3$ are needed to establish it. 
\end{remark}

\begin{remark}
The chosen zero-mean requirement for the cell solutions is not mandatory, albeit classical in the literature; a notable exception is~\cite{abdulle:20} where additional tunable averages are introduced to build families of well-posed models. In this study, we retain the zero-mean convention as it will be useful to highlight simplifications arising in special cases.
\end{remark}

\begin{remark}\label{rem:stretch}
The homogenization approach involves, of necessity, a small dimensionless parameter $\epsilon$, which makes it natural to use the PDEs~\eqref{cascade epsilon} acting on dimensionless coordinates as a basis for the homogenization process. The factor $\Lii$ on $\oo^2$ and $\brF$ induced by coordinate stretching is then absorbed for convenience in their stretched counterparts defined by~\eqref{oo:F:stretch} (which carry modified units) in the interest of limiting notational clutter in the forthcoming derivations of Sections~\ref{sec:homog:zf} and~\ref{sec:homog}, planning to revert to the original $\oo,\brF$ when interpreting the effective PDEs as SGE models in Section~\ref{sec:homo:SGE}. A valid alternative would consist in writing a fully dimensionless version of the equation of motion (using e.g. $L,\kappa\msub,\rho\msub$ as characteristic quantities), at the notational expense of carrying scaling material factors $\kappa\msub,\rho\msub$ when reverting to dimensional effective models.
%
%
%
\end{remark}

Combining the macroscopic balance equations satisfied by the fields $\brU^n$, we will derive effective PDEs obeyed (up to a $\Ocal(\epsilon^3)$ residual) by the second-order approximation~\eqref{eq:Ueps:def}, that will be shown (i) to have the form:
\begin{equation}
  \sfQ\brU^{\epsilon} + \sfR\brU^{\epsilon}{}'' = \brF \label{eq:PDE:general}
\end{equation}
(or its time-harmonic counterpart), (ii) to satisfy well-posedness requirements, and (iii) to be interpretable as PDEs modeling continua endowed with strain gradient elasticity (SGE) properties, whose format is also~\eqref{eq:PDE:general}. The elastic and inertial partial differential operators (PDOs) $\sfQ,\sfR$ acting on the relevant spatial coordinates will be found to be respectively of fourth and second order and to have, for a generic vector field $\brU$, the generic form
\begin{equation}
\begin{aligned}
  \sfQ\brU &= \sfQ[\brQ^0,\brQ^1,\brQ^2]\brU \hspace{-1em}
 && :=- \Div\lcb \brQ^0\dip\bfeps(\brU) + \brQ^1 \therefore \bfeta(\brU) - \Div\lpar \brQ^2 \therefore \bfeta(\brU) \rpar \rcb, \\
  \sfR\brU &= \sfR[\varrho,\brR^1,\brR^2]\brU 
 &&:= \varrho\brU + \brR^1\dip\bfna\brU - \Div(\brR^2\dip\bfna\brU)
\end{aligned}. \label{PDO:generic}
\end{equation}
The scalar parameter $\varrho>0$ and the tensor parameters $\brQ^0,\brQ^1,\brQ^2$ and $\brR^1,\brR^2$ are to be specified, and their properties investigated, as part of the forthcoming developments, and the signs in~\eqref{PDO:generic} are chosen so that $\sfQ,\sfR$ are positive when those parameters have the desired sign properties.
The generic PDOs~\eqref{PDO:generic} verify for any scalar and tensor parameters the coordinate-scaling identities
\begin{equation}
  \sfQ_{\rmx}[\brQ^0,\epsilon\brQ^1,\epsilon^2\brQ^2] = L^2\sfQ_{\rmX}[\brQ^0,\ell\brQ^1,\ell^2\brQ^2], \qquad
  \sfR_{\rmx}[\varrho,\epsilon\brR^1,\epsilon^2\brR^2] = \sfR_{\rmX}[\varrho,\ell\brR^1,\ell^2\brR^2] \label{PDO:scaling}
\end{equation}
(where the subscripts $\rmx$ and $\rmX$ specify differentiation with respect to dimensionless and dimensional slow coordinates, respectively) which will play an major role for interpreting as SGE models the effective PDEs yielded by homogenization.


\subsection{Mathematical framework for PDEs on the periodicity cell}

In preparation for this solution approach, we collect well-known generic concepts and results for PDEs on the periodicity cell for later reference, and in particular introduce the Sobolev space in which the coefficients of expansion~\eqref{expansion:4term} will be sought. We first recall that the $L^2(\Xcal)$ and $H^1(\Xcal)$ norms of some (possibly vector- or tensor-valued) function $\bfw$ on some open domain $\Xcal\subseteq\Rbb^d$ are respectively defined as
\begin{equation}
  \|\bfw\|^2_{L^2(\Xcal)} = \int_{\Xcal} |\bfw|^2 \dy, \qquad
  \|\bfw\|^2_{H^1(\Xcal)} = \|\bfw\|^2_{L^2(\Xcal)} + \|\bfnay\bfw\|^2_{L^2(\Xcal)},
\end{equation}
where $\bfw\shin L^2(\Xcal)$ needs only to have a weak derivative $\bfnay\bfw\shin L^2(\Xcal)$ (i.e. $\bfw$ is not necessarily differentiable, or even continuous, in $\Xcal$). In the sequel, boldface notation will indicate spaces of vector-valued functions, e.g. $\bfL^2(\Ycal):=L^2(\Ycal;\Rbb^d)$ or $\bfH^1(\Rbb^d):=H^1(\Rbb^d;\Rbb^d)$. 
Then, $\bHper(\Ycal)$ denotes the closure for the $H^1(\Ycal)$ norm of the space $\bfC^{\infty}_{\text{per}}(\Ycal)$ of all vector-valued $C^{\infty}(\Rbb^d)$ functions that are $\Ycal$-periodic (with $\Hper(\Ycal;\Rbb^d\tens\Rbb^d)$ defined similarly), and $\bHzm(\Ycal):= \lcb \bfw\in \bHper(\Ycal),\ \langle \bfw \rangle_{\Ycal} = \bfze \rcb$ is the subspace of all zero-mean functions in $\bHper(\Ycal)$.
%

The solution process will use the fact that each $\Ocal(\epsilon^{n-2})$ equation ($n\shgeq0$) in the cascade~\eqref{cascade epsilon} has the form
\begin{equation}
    \text{Find }\bruT^n \text{ such that} \qquad \bruT^n(\brx,\Cdot) \in\bHper(\Ycal), \quad
  \Lcal_{yy}\{ \bruT^n \} = \brg^n + \Divy\brh^n, \label{PDE:generic}
\end{equation}
where the right-hand sides $\brg^n(\brx,\Cdot)$ and $\brh^n(\brx,\Cdot)$ of~\eqref{PDE:generic} depend on $\bruT^{n-1},\bruT^{n-2}$. Equations \eqref{PDE:generic} can thus be solved sequentially, each $\bruT^n$ being sought as a $\Ycal$-periodic function of the fast variable $\bfy$ that is parameterized (through the right-hand side) by the slow variable $\brx$. The unique solvability result~\cite{blp:78} for the generic cell problem~\eqref{PDE:generic}, and the associated variational problem on $\Ycal$, are given in the following lemma:
\begin{lemma}[PDE problems on $\Ycal$: solvability, variational formulation]\label{lem:fredholm}
The problem
\begin{equation}
	\Lcal_{yy} \{\bfw\} = \bfg + \Divy\bfh \quad\text{in }\Ycal, \qquad \bfw \ \Ycal\text{-periodic}, \label{cell:generic}
\end{equation}
with given $\bfg\in \bfL^2(\Ycal)$ and $\bfh\in \Hper(\Ycal,\Rbb^d\tens\Rbb^d)$, has a unique solution $\bfw\in \bHzm(\Ycal)$ if and only if the right hand side satisfies the compatibility condition
\begin{equation}
  \lbra \bfg + \Divy\bfh \rbra_{\Ycal} = 
    \lbra \bfg \rbra_{\Ycal} = \bfze \label{g:zeromean}
\end{equation}
(since any $\bfh\in \Hper(\Ycal,\Rbb^d\tens\Rbb^d)$ verifies $\lbra \Divy\bfh \rbra_{\Ycal}=\bfze$), by virtue of the Fredholm alternative. In that case, the zero-mean solution of~\eqref{cell:generic} subject to~\eqref{g:zeromean} is also the unique solution of the well-posed variational problem
\begin{equation}
  \text{Find }\bfw\in\bHzm(\Ycal), \qquad A(\bfw,\bfv) = F(\bfv) \quad \text{for all }\bfv\in\bHzm(\Ycal), \label{cell:generic:weak}
\end{equation}
with the (continuous and symmetric) bilinear form $A$ and the (continuous) linear functional $F$ defined by
\begin{equation}
  A(\bfw,\bfv) := \iY \bfepsy(\bfw)\dip\bfC\dip\bfepsy(\bfv) \dy, \qquad F(\bfv) := \iY \lpar \bfg\cdot\bfv - \bfh\dip\bfnay\bfv \rpar \dy.
  \label{F:generic}
\end{equation}
All solutions $\bfw\in\bHper(\Ycal)$ of~\eqref{cell:generic} subject to~\eqref{g:zeromean} are then of the form $\bfw=\bfw_{\sharp}+\boldsymbol{c}$, where $\bfw_{\sharp}$ is the unique solution of~\eqref{cell:generic:weak} and $\boldsymbol{c} = \lbra\bfw\rbra_\Ycal\in\Rbb^d$ is arbitrary.
\end{lemma}
As usual, the symmetry of the bilinear form $A$ gives rise to a reciprocity theorem, which plays a major role in this study. We hence state it here in a generic form suitable for this work:
\begin{lemma}[Reciprocity identity]\label{lem:recip}
Let $\bfw_1,\bfw_2\in \bHzm(\Ycal)$ be zero-mean solutions of variational cell problems of the form~\eqref{cell:generic:weak}, with right-hand sides $F_1,F_2$ of the form~\eqref{F:generic} associated with respective data $(\bfg_1,\bfh_1),\,(\bfg_2,\bfh_2)$. Then, $\bfw_1,\bfw_2$ satisfy the reciprocity identity
\begin{equation}
  F_1(\bfw_2) - F_2(\bfw_1) = 0. \label{cell:generic:recip}
\end{equation}
\end{lemma}

\section{Elastostatic second-order homogenization}
\label{sec:homog:zf}

In this section, we carry out the formal homogenization process outlined in Section~\ref{sec:homproc} on the elastostatic case. We effect in Section~\ref{sec:cascade:zf} the sequential PDE solution process for the time-independent case, thereby establishing expressions~\eqref{sep:var} with $\oostr\sheq0$, macroscopic balance equations obeyed by $\brU^0$, $\brU^1$ and $\brU^2$ and effective tensors featured in them. We then derive reciprocity identities for the latter (Sec.~\ref{sec:recip:zf}) and obtain a one-parameter family of $\Ocal(\epsilon^2)$ effective PDEs satisfied by the second-order approximation~\eqref{eq:Ueps:def} of the mean field (Sec.~\ref{sec:mean:zf}). 

\subsection{Homogenization procedure}
\label{sec:cascade:zf}

The elastostatic homogenization procedure rests on sequentially solving equation~\eqref{PDE:generic} for $n=0,1,2\ldots$ and with $\oostr=0$. The generic form of the elastostatic $\Ocal(\epsilon^{n-2})$ problem is
\begin{equation}
    \text{Find }\bruT^n \text{ such that} \qquad \bruT^n(\brx,\Cdot) \in\bHper(\Ycal), \quad
  \Lcal_{yy}\{ \bruT^n \} = \brg^n + \Divy\brh^n, \label{pde:generic}
\end{equation}
the components $\brg^n$ and $\brh^n$ of the right-hand side being given by
\begin{equation}
  \brg^n = \Divx \lcb \bfC\dip\lsqb \bfepsy(\bruT^{n-1}) + \bfepsx(\bruT^{n-2}) \rsqb \rcb + \delta_{n2}\brFstr, \quad
  \brh^n = \bfC\dip\bfepsx(\bruT^{n-1})  \label{zf:rhs:generic}
\end{equation}
(with the convention $\bruT^{-1}=\bruT^{-2}=0$). The solution procedure at each order then consists in (i) enforcing the solvability condition~\eqref{g:zeromean}, i.e. $\lbra \brg^n(\brx,\Cdot) \rbra_{\Ycal}=0$, which produces the $\Ocal(\epsilon^{n-2})$ macroscopic balance equation, (ii) recasting $\brg^n$ as $\brg^n-\lbra\brg^n\rbra_{\Ycal}$ in~\eqref{zf:rhs:generic} and (iii) seeking the general solution $\bruT^n$ of problem~\eqref{pde:generic}-\eqref{zf:rhs:generic}. Once $\bruT^n$ in hand, the sequential process can move to the next order. Step (ii) allows to rewrite the right-hand side of problems~\eqref{pde:generic}-\eqref{zf:rhs:generic} in a form that automatically satisfies the solvability condition~\eqref{g:zeromean} while confining the explicit appearance of the macroscopic-scale body force $\brFstr$ to the macroscopic balance equations.\enlargethispage*{3ex}

\subsubsection{Determination of $\bruT^0$}
\label{sec:cascade1:zf}

The sequential solution process starts by solving problem~\eqref{pde:generic}-\eqref{zf:rhs:generic} with $n=0$ governing the leading-order term $\bruT^0$ of expansion~\eqref{expansion:4term}. Since $\brg^0=\brh^0=\bfze$, condition~\eqref{g:zeromean} is trivially satisfied (so that no $\Ocal(\epsilon^{-2})$ balance equation is produced) and Lemma~\ref{lem:fredholm} readily provides that all $\Ycal$-periodic solutions $\bruT^0(\brx,\bfy)$ are constant with respect to $\bfy$,
and hence have the form
\begin{equation}
    \bruT^0(\brx,\Cdot) = \brU^0(\brx) \label{U0}
\end{equation}
for some vector function $\brU^0$ to be determined. From a physical point of view, the macroscopic field $\brU^0$ can be seen as large-scale modulations of local translational rigid-body motions for each periodicity cell.

\subsubsection{Determination of $\bruT^1$ and first cell problem}
\label{sec:cascade2:zf}

The next term $\bruT^1$ of expansion~\eqref{expansion:4term} solves problem~\eqref{pde:generic}-\eqref{zf:rhs:generic} with $n=1$, whose right-hand side is given, using expression~\eqref{U0} of $\bruT^0$, by
\begin{equation}
  \brg^1 = \bfze, \qquad \brh^1 = \bfC \dip \bfepsx(\brU^0) \label{pb2:rhs}
\end{equation}
and thus trivially satisfies the solvability condition~\eqref{g:zeromean}. Thus no $\Ocal(\epsilon^{-1})$ balance equation is produced, and problem~\eqref{pde:generic}-\eqref{pb2:rhs} is
solvable for $\bruT^1$ by virtue of Lemma~\ref{lem:fredholm}. Its general solution has, by linearity, the form
\begin{equation}
	\bruT^1(\brx,\bfy)   = \brU^1(\brx) + \bfchi^1(\bfy)\dip\bfepsx\lpar\brU^0(\brx)\rpar. \label{ansatz1}
\end{equation}
This ansatz is inserted in $\Lcal_{yy}\{ \bruT^1 \} = \Divy\brh^1$, which must then hold for any value of the constant (in $\bfy$) tensor $\bfepsx\lpar\brU^0(\brx)\rpar$; this requires the tensor field $\bfchi^1$ (of order 3) to solve the following problem:
\renewcommand{\thmtitre}{First cell problem}
\begin{boxthm}
We call ``first cell problem'' the set of $d(d\shp1)/2$ PDEs on the periodicity cell defined by
\begin{equation}
  \text{Find }\bfchi^{1k\ell} \in \bHzm(\Ycal), \qquad
  \Lcal_{yy}\{\bfchi^{1k\ell}\} = \Divy  \lcb \bfC \dip (\bfe_k\tens\bfe_{\ell}) \rcb \qquad \text{in }\Ycal
  \qquad (1\leq k\leq \ell \leq d). \label{cell_pb 1}
\end{equation}
They have the format~\eqref{cell:generic} with $\bfg=\bfze$ and $\bfh=\bfh^{k\ell}=\bfC\dip(\bfe_k\tens\bfe_{\ell})$. Each $\bfw=\bfchi^{1k\ell}$ solves a variational problem of the form~\eqref{cell:generic:weak}, its linear functional $F = F^{k\ell}$ being defined by
\begin{equation}
  F^{k\ell}(\bfv) := -\iY (\bfe_k\tens\bfe_{\ell})\dip\bfC\dip\bfepsy(\bfv) \dy =  - \iY  C_{k\ell rs} v_{r,s} \dy.
\label{weak form chi1 index}
\end{equation}
Since $\bfh^{k\ell}=\bfh^{\ell k}$, we then have $\bfchi^{1k\ell}=\bfchi^{1\ell k}$ for $k\shg \ell$.
\end{boxthm}
The (symmetric) loading in equation~\eqref{cell_pb 1} corresponds to a prescribed constant strain $\bfepsb^{k\ell} = \bfe_k \stens \bfe_{\ell}$. In addition, let the cell stress tensor $\bSz=\bSz(\bfy)$ (of order 4) be defined in such a way that
\begin{equation}
  \bfC\dip\lsqb \bfepsy\lpar\bfchi^1\dip\bfeps \rpar+ \bfeps \rsqb = \bSz\dip\bfeps \label{S0:prop}
\end{equation}
for any constant tensor $\bfeps\in S^2(\Rbb^d)$. For any pair $(k,\ell)$ of loading indices, $\bSz$ is then defined by
\begin{equation}
  \bSz^{k\ell}(\bfy) = \bfC(\bfy)\dip\big[ \bfepsy\lpar\bfchi^{1k\ell}(\bfy) \rpar + \bfe_k\stens\bfe_{\ell} \big],\qquad\text{i.e.}\quad
	\Scal_{ij}^{0k\ell} =  C_{ijpq}\chi^{1k\ell}_{p,q} + C_{ijk\ell}, \label{S0:def}
\end{equation}
Equation~\eqref{cell_pb 1} therefore expresses that the cell stress $\bSz^{k\ell}$ satisfies $\Divy\bSz^{k\ell}(\bfy) = \bfze$, i.e. is self-equilibrated in $\Ycal$. Consequently, so is the stress $\bfC\dip[\bfepsy (\bruT^1) \shp \bfepsx(\bruT^0)]=\bSz\dip\bfepsx(\bruT^0)$ induced in $\Ycal$ by $\bruT^1$ given by~\eqref{ansatz1}.

\subsubsection{Leading-order balance equation and homogenized tensor}

Then, the subsequent term $\bruT^2$ of expansion~\eqref{expansion:4term} solves problem~\eqref{pde:generic}-\eqref{zf:rhs:generic} with $n=2$. Using the expressions~\eqref{U0} of $\bruT^0$ and~\eqref{ansatz1} of $\bruT^1$ and the defining property~\eqref{S0:prop} of the cell stress tensor $\bSz$ in~\eqref{zf:rhs:generic}, its right-hand side is given by
\begin{equation}
  \brg^2 = \Divx \lcb\bSz\dip\bfepsx(\brU^0) \rcb + \brFstr, \qquad 
  \brh^2 = \bfC \dip \big[ \bfepsx(\brU^1) +  \bfchi^1\dip\bfetax(\brU^0) \big]. \label{zf:pbesp1}
\end{equation}
Problem~\eqref{pde:generic}-\eqref{zf:rhs:generic} with~\eqref{zf:pbesp1} is solvable for $\bruT^2$ provided the compatibility condition $\lbra \brg^2 \rbra_{\Ycal}=\bfze$, is satisfied; this provides the $\Ocal(\epsilon^0)$ macroscopic balance equation
\begin{equation}
	\Divx \lcb \brCz \dip \bfepsx(\brU^0) \rcb + \brFstr = \bfze,
	 \label{zf:macroscopic balance u0}
\end{equation}
where $\brCz$ is the effective elasticity tensor given by
\begin{equation}
  \brCz = \lbra \bSz \rbra_{\Ycal},
  \qquad\text{i.e.}\quad
   \CZ_{ijk\ell} = \lbra C_{ijpq} \chi^{1k\ell}_{p,q} + C_{ijk\ell} \rbra_{\Ycal} \label{C0:comp}
\end{equation}
in component form. $\brCz$ 
has the following ellipticity property: there exists $\kappa^0>0$ such that
\begin{equation}
  \bfeps\dip\brCz\dip\bfeps \geq \kappa^0 |\bfeps|^2 \qquad \text{for all } \bfeps\in S^{2}(\Rbb^d), \label{aux15z}
\end{equation}
see e.g. \cite[Thm.~10.11]{ciora:donato:99} for a proof.

\subsubsection{Determination of $\bruT^2$ and second cell problem}

We now return to the governing problem for $\bruT^2$, and take advantage of the solvability condition $\lbra \brg^2 \rbra_{\Ycal}=0$ by rewriting the governing PDE in~\eqref{pde:generic} as $\Lcal_{yy}\{ \bruT^2 \}=\brg^2-\lbra \brg^2 \rbra_{\Ycal} + \Divy\brh^2$. Recalling~\eqref{zf:pbesp1} and~\eqref{zf:macroscopic balance u0},
$\bruT^2$ is thus found to obey
\begin{equation}
  \Lcal_{yy}\{ \bruT^2 \}
 = \Divx\lcb (\bSz\shm\brCz) \dip \bfepsx(\brU^0) \rcb
  + \Divy \lcb \bfC \dip \big[ \bfepsx(\brU^1) +  \bfchi^1\dip\bfetax(\brU^0) \big] \rcb. \label{zf:moch}
\end{equation}
In view of the form of its right-hand side, the general $\Ycal$-periodic solution of problem~\eqref{zf:moch} has, by Lemma~\ref{lem:fredholm} and linear superposition, the form
\begin{equation}
 \bruT^2(\brx,\bfy) = \brU^2(\brx) + \bfchi^1(\bfy) \dip \bfepsx\lpar \brU^1(\brx) \rpar + \bfchi^2(\bfy) \therefore \bfetax\lpar\brU^0(\brx)\rpar,
 \label{zf:ansatz2}
\end{equation}
where $\brU^2(\brx):=\lbra \bruT^2(\brx,\Cdot) \rbra_{\Ycal}$ is the cell average of $\bruT^2$ and $\bfchi^1$ solves the first cell problem~\eqref{cell_pb 1}. The latter fact in particular ensures that the cofactors of $\bfepsx(\brU^1)$ in both sides of~\eqref{zf:moch} cancel. Equation~\eqref{zf:moch} hence becomes
\begin{equation}
  \Lcal_{yy}\lcb \bfchi^2\therefore \bfetax(\brU^0) \rcb
  = \Divx\lcb (\bSz\shm\brCz) \dip \bfepsx(\brU^0) \rcb + \Divy  \lcb  \bfC\dip \lpar \bfchi^1\dip\bfetax(\brU^0) \rpar \rcb,
\end{equation}
which must hold for any (constant in $\bfy$) tensor $\bfetax(\brU^0)$; this leads to define $\bfchi^2$ by the following cell problem:\enlargethispage*{5ex}

\renewcommand{\thmtitre}{Second cell problem}
\begin{boxthm}
We call ``second cell problem'' the set of $d^2(d+1)/2$ PDEs on $\Ycal$ defined by
\begin{multline}
  \text{Find }\bfchi^{2k\ell m} \in \bHzm(\Ycal), \qquad
  \Lcal_{yy}\{\bfchi^{2k\ell m}\}
 = \bfg^{k\ell m} + \Divy\bfh^{k\ell m} \quad \text{in }\Ycal \qquad (1\shleq k\shleq \ell \shleq d,\ 1\shleq m\shleq d) \\
  \text{with} \qquad
  \bfg^{k\ell m} = (\bSz-\brCz)\therefore\lpar \bfe_m\tens(\bfe_k\stens\bfe_{\ell}) \rpar, \qquad
  \bfh^{k\ell m} = \bfC\dip(\bfchi^{1k\ell}\tens\bfe_m). \label{zf:load:second}
\end{multline}
They have the format~\eqref{cell:generic}.
The loadings $\bfg^{k\ell m},\bfh^{k\ell m}$ are defined in terms of the cell solution $\bfchi^1$ and the material field $\bfC$. Each $\bfw=\bfchi^{2k\ell m}$ solves a variational problem of the form~\eqref{cell:generic:weak}, its linear functional $F = F^{k\ell m}$ being defined by
\begin{equation}
  F^{k\ell m}(\bfv) :=
  \iY \Lcb \lpar \Scal_{mr}^{0k\ell} - \CZ_{k\ell mr} \rpar v_r - \chi^{1k\ell}_p C_{pmrs} v_{r,s} \Rcb \dy \label{zf:weak form chi2 index}
\end{equation}
Since $\bfg^{k\ell m}=\bfg^{\ell km}$ and $\bfh^{k\ell m}=\bfh^{\ell km}$, we have $\bfchi^{2k\ell m}=\bfchi^{2\ell km}$ for $k\shg \ell$.
\end{boxthm}

In addition, similarly to the first cell problem, we define the cell stress tensor $\bSi$ (of order 5) so that
\begin{equation}
  \bfC(\bfy)\dip\lsqb \bfepsy\lpar\bfchi^2(\bfy)\therefore\bfeta \rpar + \bfchi^1(\bfy)\dip\bfeta \rsqb
 = \bSi(\bfy)\therefore\bfeta \label{S1:prop}
\end{equation}
holds for any constant tensor $\bfeta\in S^2(\Rbb^d)\tens\Rbb^d$. For any set $(k,\ell,m)$ of loading indices, $\bSi$ is then given by
\begin{equation}
  \bSi^{k\ell m} = \bfC\dip\lcb \bfepsy\lpar\bfchi^{2k\ell m}\rpar + \bfchi^{1k\ell}\tens\bfe_m \rcb, \qquad\text{i.e.}\quad
  \Scal_{ij}^{1k\ell m} = C_{ijpq} \chi^{2k\ell m}_{p,q} + C_{ijpm} \chi_{p}^{1k\ell} \label{S1:def:zf}
\end{equation}
With this definition, equation~\eqref{zf:load:second} can be recast as the equilibrium equation $\Div_y \bSi^{k\ell m}+\bfg^{k\ell m}=\bfze$ for $\bSi$.

\subsubsection{First-order balance equation and effective tensor}

Next, the term $\bruT^3$ of expansion~\eqref{expansion:4term} solves problem~\eqref{pde:generic}-\eqref{zf:rhs:generic} with $n=3$. Recalling the expressions~\eqref{ansatz1} of $\bruT^1$ and~\eqref{zf:ansatz2} of $\bruT^2$ and using the defining properties~\eqref{S0:prop} and~\eqref{S1:prop} of the cell stress tensors $\bSz$ and $\bSi$ in~\eqref{zf:rhs:generic}, its right-hand side is given by
\begin{equation}
  \brg^3 = \Divx \lcb \bSz \dip \bfepsx(\brU^1) + \bSi \therefore \bfetax(\brU^0) \rcb, \qquad 
  \brh^3 = \bfC \dip \big[ \bfepsx(\brU^2)
  + \bfchi^1 \dip \bfetax(\brU^1) + \bfchi^2 \therefore \bfkax(\brU^0) \big]. \label{zf:g3:expr}
\end{equation}
The solvability condition $\lbra \brg^3 \rbra_{\Ycal}=\bfze$ on $\brg^3$ then provides the $\Ocal(\epsilon)$ macroscopic balance equation
\begin{equation}
	\Divx \lcb \brCz \dip\bfepsx(\brU^1) + \brCi \therefore \bfetax(\brU^0)\rcb = \bfze, \label{zf:macroscopic balance u1}
\end{equation}
where the first-order effective elastic tensor $\brCi$ is given, using explicit component notation, by
\begin{equation}
  \brCi = \lbra \bSi \rbra_{\Ycal}, \qquad \text{i.e.} \quad
  \Ci_{ijk\ell m}
 = \lbra C_{ijpq} \chi^{2k\ell m}_{p,q} + C_{ijpm} \chi_{p}^{1k\ell} \rbra_{\Ycal}. \label{zf:C1:comp}
\end{equation}

\subsubsection{Determination of $\bruT^3$ and third cell problem}

We now come back to the governing problem for $\bruT^3$. Proceeding as before, we take advantage of the solvability condition $\lbra \brg^3 \rbra_{\Ycal}=0$ by rewriting the governing PDE in~\eqref{pde:generic} as $\Lcal_{yy}\{ \bruT^3 \}=\brg^3-\lbra \brg^3 \rbra_{\Ycal} + \Divy\brh^3$. Recalling~\eqref{zf:g3:expr} and~\eqref{zf:macroscopic balance u1},
$\bruT^3$ is thus found to obey
\begin{align}
  \Lcal_{yy}\{ \bruT^3 \}
  &= \Divx\lcb (\bSz\shm\brCz) \dip \bfepsx(\brU^1) + (\bSi\shm\brCi)\therefore \bfetax(\brU^0) \rcb \suite\qquad
  + \Divy \lcb \bfC \dip \big[ \bfepsx(\brU^2)
  + \bfchi^1 \dip \bfetax(\brU^1) + \bfchi^2 \therefore \bfkax(\brU^0) \big] \rcb,\label{zf:aux 03}
\end{align}
whose solvability is assured. By Lemma~\ref{lem:fredholm} and linearity, the $\Ycal$-periodic solutions of~\eqref{zf:aux 03} have the form
\begin{equation}
	 \bruT^3(\brx,\bfy) = \brU^3(\brx) + \bfchi^1(\bfy) \dip \bfepsx\lpar\brU^2(\brx)\rpar + \bfchi^2(\bfy) \therefore \bfetax\lpar\brU^1(\brx)\rpar +  \bfchi^3(\bfy) \qip \bfkax\lpar\brU^0(\brx)\rpar,
 \label{zf:ansatz3}
\end{equation}
where $\brU^3(\brx):=\lbra \bruT^3(\brx,\Cdot) \rbra_{\Ycal}$ is the cell average of $\bruT^3$, $\bfchi^1$ and $\bfchi^2$ solve the respective cell problems~\eqref{cell_pb 1} and~\eqref{zf:load:second}, and the cell tensor function $\bfchi^3$ (of order 5) is to be determined. Indeed, upon substituting the expression~\eqref{zf:ansatz3} of $\bruT^3$ and cancelling the contributions of the first two cell problems, \eqref{zf:aux 03} simplifies to
\begin{equation}
  \Lcal_{yy}\lcb \bfchi^3\qip \bfkax(\brU^0) \rcb
 = \Divx\lcb (\bSi\shm\brCi) \therefore \bfetax(\brU^0) \rcb + \Divy \bfC \dip \lpar \bfchi^2\therefore\bfkax(\brU^0) \rpar \rcb,
\end{equation}
which must hold for any value of $\bfkax(\brU^0)$. This leads to the following definition of the cell solution $\bfchi^3$:

\renewcommand{\thmtitre}{Third cell problem}
\begin{boxthm}
We call ``third cell problem'' the set of $d^3(d+1)/2$ PDEs on $\Ycal$ defined by
\begin{multline}
  \text{Find }\bfchi^{3k\ell mn} \in \bHzm(\Ycal), \qquad
  \Lcal_{yy}\{\bfchi^{3k\ell mn}\}
 = \bfg^{k\ell mn} + \Divy\bfh^{k\ell mn} \quad \text{in }\Ycal \qquad (1\shleq k\shleq \ell \shleq d,\ 1\shleq m,n\shleq d) \\
  \text{with} \qquad
  \bfg^{k\ell mn} = \lpar \bSi - \brCi \rpar \qip \lpar \bfe_n\tens(\bfe_k\stens\bfe_{\ell})\tens\bfe_m \rpar, \quad
  \bfh^{k\ell mn} = \bfC\dip(\bfchi^{2k\ell m}\tens\bfe_n). \label{zf:load:third}
\end{multline}
They have the format~\eqref{cell:generic}.
The loadings $\bfg^{k\ell mn},\bfh^{k\ell mn}$ are defined in terms of the cell functions $\bfchi^1,\bfchi^2$ and the material field $\bfC$. Each $\bfw=\bfchi^{3k\ell mn}$ solves a variational problem of the form~\eqref{cell:generic:weak}, its linear functional $F = F^{k\ell mn}$ being defined by
\begin{equation}
  F^{k\ell mn}(\bfv) :=
  \iY \Lcb \lpar \Scal_{rn}^{1k\ell m} - \Ci_{rnk\ell m} \rpar v_r - \chi^{2k\ell m}_{p} C_{pnrs}  v_{r,s} \Rcb \dy
\label{zf:weak form chi3 index}
\end{equation}
Since $\bfg^{k\ell mn}=\bfg^{\ell kmn}$ and $\bfh^{k\ell mn}=\bfh^{\ell kmn}$, we have $\bfchi^{3k\ell mn}=\bfchi^{3\ell kmn}$ for $k\shg \ell$.
\end{boxthm}

Similarly to the first two cell problems, the cell stress tensor $\bSii$ (of order 6) is defined so that
\begin{equation}
  \bfC(\bfy)\dip\lsqb \bfepsy\lpar\bfchi^3(\bfy)\qip\bfka \rpar + \bfchi^2(\bfy)\therefore\bfka \rsqb = \bSii(\bfy)\qip\bfka
\label{S2:prop}
\end{equation}
holds for any constant tensor $\bfka\in S^2(\Rbb^d)\tens S^2(\Rbb^d)$. 
For any set $(k,\ell,m,n)$ of loading indices, $\bSii$ is then given by
\begin{equation}
  \bSii^{k\ell mn} = \bfC\dip\lcb \bfepsy\lpar\bfchi^{3k\ell mn}\rpar + \bfchi^{2k\ell m}\tens\bfe_n \rcb, \qquad\text{i.e.}\quad
  \Scal_{ij}^{2k\ell mn} = C_{ijpq} \chi^{3k\ell mn}_{p,q} + C_{ijpn} \chi_{p}^{2k\ell m} \label{S2:def:zf}
\end{equation}
This definition allows to recast equation~\eqref{zf:load:third} as the equilibrium equation $\Div\bSii^{k\ell mn}+\bfg^{k\ell mn}=\bfze$ for $\bSii$.

\subsubsection{Second-order balance equation and effective tensor}

Finally, $\bruT^4$ solves problem~\eqref{pde:generic}-\eqref{zf:rhs:generic} with $n=4$. 
Finding $\bruT^4$ is not necessary, as
we only need for this last step of the homogenization process to formulate the solvability condition~\eqref{g:zeromean} on $\brg^4$. On recalling the expressions~\eqref{zf:ansatz2} of $\bruT^2$ and~\eqref{zf:ansatz3} of $\bruT^3$ and using the defining properties~\eqref{S0:prop}, \eqref{S1:prop} and~\eqref{S2:prop} of the three cell stress tensors in~\eqref{zf:rhs:generic}, we find that $\brg^4$ is given by
\begin{equation}
  \brg^4
 = \Divx \lcb \bSz\dip\bfepsx(\brU^2) + \bSi\therefore\bfetax(\brU^1) + \bSii\qip\bfkax(\brU^0) \rcb.
\end{equation}
Enforcing the solvability condition~\eqref{g:zeromean} on $\brg^4$ then results in the $\Ocal(\epsilon^2)$ macroscopic balance equation
\begin{equation}
    \Divx \lcb \brCz \dip\bfepsx(\brU^2) + \brCi \therefore \bfetax(\brU^1) + \Divx \lsqb\brCii\therefore\bfetax(\brU^0)\rsqb \rcb = \bfze, \label{zf:macroscopic balance u2}
\end{equation}
where $\brCii$ is the second-order homogenized elastic tensor defined, for any field $\brU(\brx)$, by
\begin{equation}
  \Divx \lsqb \brCii\therefore\bfetax(\brU) \rsqb = \lbra \mbox{$\bSii$} \rbra_{\Ycal}\qip\bfkax(\brU), \qquad \text{i.e.} \quad
   \Cii_{ijnk\ell m}
 = \lbra C_{ijpq} \lpar \chi^{3k\ell mn}_{p,q}+ \chi_{p}^{2k\ell m}\delta_{qn} \rpar \rbra_{\Ycal}. \label{zf:C2:comp}
\end{equation}

\subsection{Reciprocity identities}
\label{sec:recip:zf}

The following proposition provides alternative expressions for the effective tensors $\brCz$, $\brCi$ and $\brCii$, respectively defined by~\eqref{C0:comp}, \eqref{zf:C1:comp} and~\eqref{zf:C2:comp}, which are established by applying the reciprocity identity of Lemma~\ref{lem:recip} to cell functions. The result for $\brCii$ is, to our best knowledge, original; the others are known (see below) and provided for later use. A proof for all three expressions is given in Section~\ref{sec:proof:recip:zf}, for completeness (and to further emphasize the systematic use of reciprocity) in the case of  $\brCz$ and $\brCi$.\enlargethispage*{1ex}





\begin{prop}\label{prop:recip:zf}
Let $\bfchi^1$ and $\bfchi^2$ be the first and second cell functions, respectively defined by~\eqref{cell_pb 1} and~\eqref{zf:load:second}. Using component notation, the elastostatic leading-order effective tensor $\brCz$ is given in terms of $\bfchi^1$ by
\begin{equation}
  \CZ_{ijk\ell} = \inv{|\Ycal|}A\lpar \bfchi^{1ij}\shp\bfP^{ij},\, \bfchi^{1k\ell}\shp\bfP^{k\ell} \rpar
\label{zf:reciprocity 0}
\end{equation}
(with $\bfP^{ij}(\bfx):=\tdemi(x_i\bfe_j\shp x_j\bfe_i)$ and the bilinear form $A$ given by~\eqref{F:generic}), the elastostatic first-order effective tensor $\brCi$ is given in terms of $\bfchi^1$ by
\begin{equation}
   \Ci_{ijk\ell m}
 = \inv{|\Ycal|} \iY \lpar \Scal_{mp}^{0ij}\chi^{1k\ell}_p - \Scal_{mp}^{0k\ell}\chi^{1ij}_p \rpar \dy,
\label{zf:reciprocity 1}
\end{equation}
(with the cell stress $\bSz$ given by~\eqref{S0:def}), while the elastostatic second-order effective tensor $\brCii$ is given in terms of $\bfchi^1$ and $\bfchi^2$ by
\begin{equation}
  \Cii_{ijnk\ell m}
 = \inv{|\Ycal|} \iY \chi^{2ij n}_{p,q}  C_{pqab} \chi^{2k\ell m}_{a,b} \dy
  - \inv{|\Ycal|} \iY \chi^{1ij}_p C_{pnam} \chi^{1k\ell}_a \dy. \label{zf:reciprocity 2}
\end{equation}
\end{prop}
The above expressions make the symmetry properties of $\brCz$, $\brCi$ and $\brCii$ explicit while reducing the number of cell solutions needed for the practical evaluation of $\brCi$ and $\brCii$. Formula~\eqref{zf:reciprocity 0} is a well-known result of elastic homogenization, see e.g.~\cite[Prop.~10.12]{ciora:donato:99}, and $\bfCz$ is clearly seen to possess the minor and major index symmetries
\begin{equation}
  \CZ_{ijk\ell} = \CZ_{jik\ell} = \CZ_{ij\ell k} = \CZ_{k\ell ij}, \label{C0:sym}
\end{equation}
also well-known. Formula~\eqref{zf:reciprocity 1} was previously established in~\cite[App.~A]{boutin:96}; it shows that $\brCi$ has the minor symmetry and major skew-symmetry properties
\begin{equation}
  \Ci_{ijk\ell m} = \Ci_{jik\ell m} = \Ci_{ij\ell km}, \qquad \Ci_{ijk\ell m} = - \Ci_{k\ell ijm}, \label{K:skew}
\end{equation}
and can be computed from the first cell solution $\bfchi^1$ only, instead of $(\bfchi^1,\bfchi^2)$ if using~\eqref{zf:C1:comp}. Likewise, formula~\eqref{zf:reciprocity 2} shows that $\brCii$ has the minor and major index symmetries
\begin{equation}
  \Cii_{ijnk\ell m} = \Cii_{jink\ell m} = \Cii_{ijn\ell km} = \Cii_{k\ell m ijn}. \label{C2:sym}
\end{equation}
and can be computed from the cell solutions $\bfchi^1,\bfchi^2$, instead of $\bfchi^2,\bfchi^3$ if using~\eqref{zf:C2:comp}. This shows that the third cell problem needs to be formulated for the formal homogenization process, in particular to establish the results of Proposition~\ref{prop:recip:zf}, but its solution $\bfchi^3$ is ultimately  not involved in the practical evaluation of effective tensors.



\begin{remark}
The expression~\eqref{zf:reciprocity 2} of $\brCii$ reveals its symmetry properties but leaves its sign unresolved. Indeed, a similar tensor for scalar conductivity problems, studied by \cite{Allaire2016}, is shown to have no clear sign in general; on an explicit example provided therein involving laminates, that tensor is shown to be sign-indefinite (by having both positive and negative eigenvalues). For elasticity, sign-indefiniteness of the corresponding tensor is observed in various situations, \eg \cite{durand:22}, without supporting proofs to the best of our knowledge. 
\end{remark}

\subsection{Fourth-order effective elastostatic PDE}
\label{sec:mean:zf}

We define the second-order approximation $\brU^{\epsilon}(\brx)$ of the macroscopic displacement field as
\begin{equation}
	 \brU^{\epsilon}(\brx) =   \brU^0(\brx) + \epsilon \brU^1(\brx)  + \epsilon^2 \brU^2(\brx).
	 \label{Ueps:def:zf}
\end{equation}
Its coefficients $\brU^0(\brx),\,\brU^1(\brx),\,\brU^2(\brx)$ satisfy the macroscopic balance equations~\eqref{zf:macroscopic balance u0}, \eqref{zf:macroscopic balance u1} and~\eqref{zf:macroscopic balance u2},
recalled for convenience:
\begin{equation}
\begin{aligned}
 \text{(a) \ }\bfze &= \Divx \lcb \brCz \dip \bfepsx(\brU^0) \rcb + \brFstr \\
 \text{(b) \ }\bfze &= \Divx \lcb \brCz \dip\bfepsx(\brU^1) + \brCi \therefore \bfetax(\brU^0)\rcb = \bfze \\
 \text{(c) \ }\bfze &= \Divx \lcb \brCz \dip\bfepsx(\brU^2) + \brCi \therefore \bfetax(\brU^1) + \Divx \lsqb\brCii\therefore\bfetax(\brU^0)\rsqb \rcb.
\end{aligned} \label{zf:mean:eqs:init}
\end{equation}
An effective PDE governing $\brU^{\epsilon}(\brx)$ can \emph{a priori} be found from the combination (\ref{zf:mean:eqs:init}a)+$\epsilon$(\ref{zf:mean:eqs:init}b)+$\epsilon^2$(\ref{zf:mean:eqs:init}c).  However, the tensor $\brCii$ is potentially sign-indefinite while, to possibly define an enriched linear elastic model, the resulting effective PDE would need (transposing the later Remark~\ref{ke:SGE:rem} to $\brC\sheq\brCz$, $\brMsh\sheq\epsilon\brCi$ and $\brA\sheq-\epsilon^2\brCii$) the tensor $-\brCii-\frac{1}{4}\brCi{}\Tsup\dip(\brCz)^{-1}\dip\brCi\in S^2\lpar S^2(\Rbb)\tens\Rbb \rpar$ to be positive definite.
For this reason, before effecting the above combination, we modify~(\ref{zf:mean:eqs:init}c) into a form that is also verified by the coefficients of any approximation of the form~\eqref{Ueps:def:zf}. Indeed, taking the Laplacian of~(\ref{zf:mean:eqs:init}a) as one of the possible ways (suggested in~\cite{Allaire2016}) to set up such auxiliary identities, $\brU^0$ satisfies
\begin{equation}
  -\Divx\lcb \Divx\lsqb \brAth\therefore\bfeta(\brU^0) \rsqb\rcb - \bfDelx\brFstr = \bfze \label{zf:aux05}
\end{equation}
with the (positive) 6th-order tensor $\brAth$ defined, using component notation, by
\begin{equation}
    \Ath_{ijnl\ell m} = \CZ_{ijk\ell}\delta_{nm} \label{Ath:def}
\end{equation}
in order to satisfy $\bfDelx\Divx \lcb \brCz \dip \bfepsx(\brU^0)\rcb=\Divx\lcb \Divx\lsqb\brCz\dip\bfetax(\brU^0)\rsqb \rcb=\Divx\lcb \Divx\lsqb \brAth\therefore\bfetax(\brU^0) \rsqb \rcb$ for any field $\brU(\brx)$.
We then add to~(\ref{zf:mean:eqs:init}c) a scalar multiple of~\eqref{zf:aux05} featuring an adjustable positive weight $\vartheta$, to obtain the modified form
\begin{equation}
  \Div \lcb \brCz \dip\bfeps(\brU^2) + \brCi \therefore \bfeta(\brU^1) - \Div \lsqb\brAZii(\vartheta)\therefore\bfeta(\brU^0)\rsqb \rcb - \vartheta\bfDelx\brFstr = \bfze \label{zf:mean:eq:order2}
\end{equation}
of the $\Ocal(\epsilon^2)$ macroscopic balance equation~(\ref{zf:mean:eqs:init}c), where the $\vartheta$-dependent 6th-order tensor $\brAZii(\vartheta)$ is defined by
\begin{equation}
    \brAZii(\vartheta) = \vartheta\brAth - \brCii. \label{zf:A:def}
\end{equation}
We observe that $\brAZii(\vartheta)\in S^2\lpar S^2(\Rbb^d)\tens\Rbb^d \rpar$ since $\brAth$ has the same minor and major symmetries~\eqref{C2:sym} as $\brCii$.
We finally formulate the combination (\ref{zf:mean:eqs:init}a)+$\epsilon$(\ref{zf:mean:eqs:init}b)+$\epsilon^2$\eqref{zf:mean:eq:order2}, to obtain
\begin{equation}
   \sfQ_{\rmx}\lsqb \brCz,\epsilon\brCi,\epsilon^2\brAZii(\vartheta) \rsqb\brU^{\epsilon}
 = (1-\vartheta\epsilon^2\Delta_\brx) \brFstr + \epsilon^3\brO^{\epsilon}, \label{zf:eff0}
\end{equation}
using the generic PDO notation introduced in~\eqref{PDO:generic} and with the residual $\brO^{\epsilon}$ given by
\begin{equation}
  \brO^{\epsilon} := -\Div \lcb \brCi\therefore\bfeta(\brU^2) + \brAZii(\vartheta)\therefore\lpar \bfeta(\brU^1) \shp \epsilon\bfeta(\brU^2) \rpar \rcb,
\end{equation}
so that the mean field second-order approximation $\brU^{\epsilon}$ satisfies the effective PDE (of fourth order)
\begin{equation}
  \sfQ_{\rmx}\lsqb \brCz,\epsilon\brCi,\epsilon^2\brAZii(\vartheta) \rsqb\brU^{\epsilon} = (1-\vartheta\epsilon^2\bfDelx) \brFstr 
\label{zf:eff}
\end{equation}
up to a $\Ocal(\epsilon^3)$ perturbation in view of the residual in the right-hand side of~\eqref{zf:eff0}.


\begin{remark}\label{rem:correction}
The  sign-restoring procedure implemented here, known as a "Boussinesq trick" in the literature (see \cite[Remark 5.3]{Allaire2016}), can be seen as correcting internal lengths associated with the tensor $\brCii$ obtained by homogenization.
It may also be noted that constitutive tensors of the form $\rmA_{ijnk\ell m}=\vartheta \CZ_{ijk\ell}\delta_{nm}$, used here as sign-corrective additive terms, appear in the literature on second-gradient elasticity under the name Helmholtz elasticity, see~\cite{lazar:05,lazar:15}, with extensions of the form  $\rmA_{ijnk\ell m}=\CZ_{ijk\ell}\rmd_{nm}$ examined in \cite{polizzotto:18} or approximations $\rmA_{ijnk\ell m} \approx \CZ_{ijnp}\widehat{C}_{k\ell mp}$, with $\widehat{\bfC}$ optimized to obtain the "best approximation" of $\brA$ in terms of $\brCz$, by \cite{fish:02b}. These forms might also be used as sign-restoring corrective terms instead of the simple one adopted here in~\eqref{zf:A:def}, for instance to introduce other parameters in the effective model; we intend to investigate such issues in the future.
\end{remark}

\section{Second-order homogenization of the elastic wave equation}
\label{sec:homog}

We now address the second-order homogenization for the elastodynamic case, following the same steps as in Section~\ref{sec:homog:zf}:
solve the first four equations of~\eqref{cascade epsilon} to establish separated-variable displacement expressions, macroscopic balance equations and their effective tensor coefficients (Sec.~\ref{sec:cascade}), derive reciprocity identities for those effective tensors (Sec.~\ref{sec:recip}), and obtain a one-parameter family of $\Ocal(\epsilon^2)$, second-order in time, effective PDEs satisfied by the second-order mean field approximation (Sec.~\ref{sec:mean}) that are shown to be consistent with their elastostatic counterpart in the zero-frequency limit.

\subsection{Homogenization procedure}
\label{sec:cascade}

We now treat the elastodynamic counterpart of the second-order homogenization process developed in Sec.~\ref{sec:homog:zf}, which as before relies on sequentially solving equation~\eqref{PDE:generic} for $n=0,1,2\ldots$,
with the components $\brg^n$ and $\brh^n$ of the right-hand side now given by
\begin{equation}
  \brg^n = \Divx \lcb \bfC\dip\lsqb \bfepsy(\bruT^{n-1}) + \bfepsx(\bruT^{n-2}) \rsqb \rcb + \rho\oostr^2\bruT^{n-2}, \qquad
  \brh^n = \bfC\dip\bfepsx(\bruT^{n-1})  \label{rhs:generic}
\end{equation}
(with the convention $\bruT^{-1}=\bruT^{-2}=0$) instead of~\eqref{zf:rhs:generic}.  The initial steps described in Secs.~\ref{sec:cascade1:zf} and~\ref{sec:cascade2:zf} apply without change to the elastodynamic case, and thus are not repeated here; in particular the first two coefficients $\bruT^0,\bruT^1$ in expansion~\eqref{expansion:4term}, and the first cell solution $\bfchi^1$, are the same.\enlargethispage*{1ex}

\subsubsection{Leading-order balance equation and homogenized tensors}

Then, the subsequent term $\bruT^2$ of expansion~\eqref{expansion:4term} solves problem~\eqref{pde:generic}-\eqref{rhs:generic} with $n=2$. Using the expressions~\eqref{U0} of $\bruT^0$ and~\eqref{ansatz1} of $\bruT^1$ and the defining property~\eqref{S0:prop} of the cell stress tensor $\bSz$ in~\eqref{rhs:generic}, its right-hand side is given by
\begin{equation}
  \brg^2 = \Divx \lcb\bSz\dip\bfepsx(\brU^0) \rcb + \oostr^2 \rho\brU^0 + \brFstr, \qquad 
  \brh^2 = \bfC \dip \big[ \bfepsx(\brU^1) +  \bfchi^1\dip\bfetax(\brU^0) \big]. \label{pbesp1}
\end{equation}
The solvability condition $\lbra \brg^2 \rbra_{\Ycal}=\bfze$ provides the $\Ocal(\epsilon^0)$ macroscopic balance equation
\begin{equation}
	\Divx \lcb \brCz \dip \bfepsx(\brU^0) \rcb + \oostr^2\rhoz \brU^0 + \brFstr = 0,
	 \label{macroscopic balance u0}
\end{equation}
with the effective elasticity tensor $\brCz$ again given by~\eqref{C0:comp}, while the effective mass density $\rhoz$ is defined by
\begin{equation}
  \rhoz = \lbra \rho \rbra_{\Ycal}. \label{rhoz:def}
\end{equation}
In what follows, it will prove convenient to define subsequent inertial effective parameters in terms of the relative mass density
\begin{equation}
  \beta(\bfy) := \rho(\bfy) / \rhoz \qquad \bfy\in\Ycal. \label{beta:def}
\end{equation}

\subsubsection{Determination of $\bruT^2$ and second inertial cell problem}

We return to the governing problem for $\bruT^2$, and as before use the solvability condition $\lbra \brg^2 \rbra_{\Ycal}=0$ to rewrite the governing PDE in~\eqref{pde:generic} as $\Lcal_{yy}\{ \bruT^2 \}=\brg^2-\lbra \brg^2 \rbra_{\Ycal} + \Divy\brh^2$. Recalling~\eqref{pbesp1} and~\eqref{macroscopic balance u0}, $\bruT^2$ is found to obey
\begin{equation}
    \Lcal_{yy}\{ \bruT^2 \}
 = \Divx\lcb (\bSz\shm\brCz)\dip\bfepsx(\brU^0) \rcb
 + \Divy \lcb \bfC \dip \big[ \bfepsx(\brU^1) +  \bfchi^1\dip\bfetax(\brU^0) \big] \rcb
  + \oostr^2\rhoz(\beta\shm1)\brU^0(\brx). \label{moch}
\end{equation}
Comparing the right-hand side of the above equation to that of its elastostatic counterpart~\eqref{zf:moch} and invoking Lemma~\ref{lem:fredholm} and linear superposition, the general $\Ycal$-periodic solution of problem~\eqref{moch} has the form
\begin{equation}
 \bruT^2(\brx,\bfy) = \brU^2(\brx) + \bfchi^1(\bfy) \dip \bfepsx\lpar\brU^1(\brx)\rpar
 + \bfchi^2(\bfy) \therefore \bfetax\lpar\brU^0(\brx)\rpar
 + \oostr^2\rhoz\bfzet^2(\bfy)\sip\brU^0(\brx),
 \label{ansatz2}
\end{equation}
where $\brU^2(\brx):=\lbra \bruT^2(\brx,\Cdot) \rbra_{\Ycal}$ is the cell average of $\bruT^2$, $\bfchi^1$ and $\bfchi^2$ respectively solve the first and second elastostatic cell problems~\eqref{cell_pb 1} and~\eqref{zf:load:second}, and the additional tensor field $\bfzet^2$ (of order 2) is readily found by identification to be defined by the following cell problem:
\renewcommand{\thmtitre}{Second inertial cell problem}
\begin{boxthm}
We call ``second inertial cell problem'' the set of $d$ PDEs on $\Ycal$ defined by
\begin{equation}
  \text{Find }\bfzet^{2n} \in \bHzm(\Ycal), \qquad \Lcal_{yy}\{\bfzet^{2n}\} = (\beta\shm1)\bfe_n\qquad \text{in }\Ycal. \label{load:second}
\end{equation}
and $1\leq n\leq d$; they have the format~\eqref{cell:generic}. The loadings $\bfg^n:=(\beta\shm1)\bfe_n$ depend on the material field $\rho$. Each $\bfw=\bfzet^{2n}$ solves a variational problem of the form~\eqref{cell:generic:weak}, its linear functional $F = \FHat^n$ being defined by
\begin{equation}
  \FHat^n(\bfv) := \iY (\beta\shm1) v_n \dy. \label{weak form chi2 index}
\end{equation}
\end{boxthm}
\begin{remark}\label{rem:zeta2}
We call problem~\eqref{load:second} the second inertial cell problem (even though no "first" such problem arises) because it occurs in the course of determining the second-order displacement correction $\bfuT^2$. The cell solutions $\bfzet^{2n}$ are nontrivial and linearly independent except when the mass density is homogeneous (i.e. $\beta\sheq1$ a.e.).
\end{remark}

\subsubsection{First-order balance equation and homogenized tensors}

The next term $\bruT^3$ in expansion~\eqref{expansion:4term} solves problem~\eqref{pde:generic}-\eqref{zf:rhs:generic} with $n=3$. Recalling the expressions~\eqref{ansatz1} of $\bruT^1$ and~\eqref{ansatz2} of $\bruT^2$ and using the defining properties~\eqref{S0:prop} and~\eqref{S1:prop} of the cell stress tensors $\bSz$ and $\bSi$ in~\eqref{rhs:generic}, its right-hand side is given by
\begin{equation}
\begin{aligned}
  \brg^3 &= \Divx \lcb \bSz \dip \bfepsx(\brU^1) + \bSi\therefore\bfetax(\brU^0) \rcb
    + \oostr^2\rhoz \lsqb \beta\brU^1 + \bfphc^1\dip\bfnax\brU^0 \rsqb, \\ 
  \brh^3 &= \bfC \dip \big[ \bfepsx(\brU^2)
  + \bfchi^1 \dip \bfetax(\brU^1) + \bfchi^2 \therefore \bfkax(\brU^0) \big], 
\end{aligned} \label{g3:expr}
\end{equation}
with the tensor field $\bfphc^1$ (of order 3) defined in terms of $\bfchi^1$ and $\bfzet^2$ by
\begin{equation}
  \phi^1_{ik\ell} = C_{i\ell pq}\zeta^{2k}_{p,q} + \beta\chi^{1k\ell}_i \label{phi1:def}
\end{equation}
Enforcing the solvability condition~\eqref{g:zeromean} on $\brg^3$ then provides the $\Ocal(\epsilon)$ macroscopic balance equation:
\begin{equation}
  \Divx \lcb \brCz \dip\bfepsx(\brU^1) + \brCi\therefore \bfetax(\brU^0) \rcb
   + \oostr^2\rhoz \lsqb \brU^1 + \bfbsh\dip\bfnax\brU^0 \rsqb = \bfze,
\label{macroscopic balance u1}
\end{equation}
where the homogenized inertial tensor $\bfbsh$ is given by
\begin{equation}
  \bfbsh = \lbra \bfphc^1 \rbra_{\Ycal}, \qquad \text{i.e.} \qquad
  \bsh_{ik\ell} = \lbra \phi^1_{ik\ell} \rbra_{\Ycal} \quad(1\shleq i,k,\ell \shleq d). \label{bfbsh:def}
\end{equation}

\subsubsection{Determination of $\bruT^3$ and third inertial cell problem}

We now come back to the governing problem for $\bruT^3$. As before, we exploit the solvability condition $\lbra \brg^3 \rbra_{\Ycal}=\bfze$ by rewriting the governing PDE in~\eqref{pde:generic} as $\Lcal_{yy}\{ \bruT^3 \}=\brg^3-\lbra \brg^3 \rbra_{\Ycal} + \Divy\brh^3$. Recalling~\eqref{g3:expr} and~\eqref{macroscopic balance u1}, $\bruT^3$ is hence found to obey
\begin{align}
  \Lcal_{yy}\{ \bruT^3 \}
  &= \Divx\lcb (\bSz-\brCz)\dip\bfepsx(\brU^1) + (\bSi-\brCi)\therefore\bfetax(\brU^0) \rcb\suite
  + \Divy \lcb \bfC \dip \lsqb \bfepsx(\brU^2) + \bfchi^1\dip\bfetax(\brU^1) + \bfchi^2\therefore\bfkax(\brU^0) \rsqb\rcb \suite
  + \oostr^2\rhoz \Lpar (\beta\shm1) \brU^1 + (\bfphc^1\shm\bfbsh)\dip\bfnax\brU^0
    + \Divy\lcb \bfC\dip\lsqb \bfzet^2\sip\bfnax\brU^0 \rsqb \rcb \Rpar,
\label{aux 03}
\end{align}
whose solvability is guaranteed by the construction of its right-hand side. Moreover, its first two lines of the latter coincide with the right-hand side of the corresponding elastostatic equation~\eqref{zf:aux 03}, and the term $\oostr^2\rhoz (\beta\shm1) \brU^1$ is the right-hand side of the second inertial cell problem~\eqref{load:second} contracted by a constant vector. Those observations, combined with Lemma~\ref{lem:fredholm} and linear superposition, imply that all $\Ycal$-periodic solutions of~\eqref{aux 03} have the form
\begin{align}
  \bruT^3(\brx,\bfy)
 &= \brU^3(\brx) + \bfchi^1(\bfy) \dip \bfepsx\lpar\brU^2(\brx)\rpar + \bfchi^2(\bfy) \therefore \bfetax(\brU^1(\brx))
 + \bfchi^{3}(\bfy) \qip \bfkax\lpar\brU^0(\brx)\rpar \suite\quad
  + \oostr^2\rhoz \lsqb\bfzet^2(\bfy)\sip\brU^1(\brx) + \bfzet^3(\bfy)\dip\bfnax\brU^0(\brx) \rsqb, \label{ansatz3}
\end{align}
where $\brU^3(\brx):=\lbra \bruT^3(\brx,\Cdot) \rbra_{\Ycal}$ is the cell average of $\bruT^3$, $\bfchi^1,\bfchi^2,\bfchi^3$ respectively solve the cell problems~\eqref{cell_pb 1}, \eqref{zf:load:second} and \eqref{zf:load:third}, and identification shows the new tensor field $\bfzet^3$ (of order 3) to be defined by the following cell problem:

\renewcommand{\thmtitre}{Third inertial cell problem}
\begin{boxthm}
We call ``third inertial cell problem'' the set of $d^2$ PDEs on $\Ycal$ defined by
\begin{equation}
  \text{Find }\bfzet^{3k\ell} \in \bHzm(\Ycal), \qquad
  \Lcal_{yy}\{\bfzet^{3k\ell}\}
 = (\bfphc^1\shm\bfbsh) \dip (\bfe_k\tens\bfe_{\ell})
   + \Divy\lcb \bfC\dip\lpar\bfzet^{2k}\tens\bfe_{\ell}\rpar \rcb \quad \text{in }\Ycal \label{load:third}
\end{equation}
and $1\shleq k,\ell \shleq d$; they have the format~\eqref{cell:generic}. The right-hand side is defined, through $\bfphc^1$ given by~\eqref{phi1:def}, in terms of the cell solutions $\bfchi^1,\,\bfzet^2$ and the material fields $\bfC,\rho$. Each $\bfw=\bfzet^{3k\ell}$ solves a variational problem of the form~\eqref{cell:generic:weak}, whose linear functional $F = \FHat^{k\ell}$ is defined by
\begin{equation}
  \FHat^{k\ell}(\bfv)
 := \iY \lpar \phi^{1}_{rk\ell} - \bsh_{rk\ell} \rpar v_r \dy
   -\iY \zeta^{2k}_a C_{a\ell rs} v_{r,s}\dy. \label{weak form chi3 index}
\end{equation}
\end{boxthm}

\subsubsection{Second-order balance equation and homogenized tensors}

Finally, $\bruT^4$ solves problem~\eqref{pde:generic}-\eqref{rhs:generic} with $n=4$.  As in the elastostatic case, finding $\bruT^4$ is not necessary, and this last step of the homogenization process rests only on the solvability condition $\lbra\brg^4\rbra_{\Ycal}=\bfze$. On recalling the expressions~\eqref{ansatz2} of $\bruT^2$ and~\eqref{ansatz3} of $\bruT^3$ and using the defining properties~\eqref{S0:prop}, \eqref{S1:prop} and~\eqref{S2:prop} of the elastostatic cell stress tensors $\bSz,\bSi,\bSii$ in~\eqref{zf:rhs:generic}, we find after some algebra and rearrangement that $\brg^4$ is given by
\begin{align}
  \brg^4
 &= \Divx \lcb \bSz\dip\bfepsx(\brU^2) + \bSi\therefore\bfetax(\brU^1)
    + \bSii\qip\bfkax(\brU^0) \rcb \suite\qquad
    + \oostr^2\rhoz \lsqb \beta \brU^2 + \bfphc^1\dip\bfnax\brU^1 + \bfphc^2 \therefore\bfnax^2\brU^0 \rsqb
	+ (\oostr^2\rhoz)^2 \beta \bfzet^2\sip\brU^0
\end{align}
with the tensor function $\bfphc^2$ (of order 4) defined in terms of $\bfchi^1,\bfchi^2$ and $\bfzet^2,\bfzet^3$ by
\begin{equation}
  \phi^2_{i\ell km}
 = C_{i\ell ab}\zeta^{3km}_{a,b} + C_{i\ell am}\zeta^{2k}_a + \beta\chi^{2mk\ell}_i \label{phi2:def}
\end{equation}
Enforcing the solvability condition~\eqref{g:zeromean} for $\brg^4$ then results in the $\Ocal(\epsilon^2)$ macroscopic balance equation
\begin{align}
\MoveEqLeft[3]{
  \Divx \lcb \brCz\dip\dip\bfepsx(\brU^2) + \brCi\therefore\bfetax(\brU^1)
   + \Divx\lsqb \brCii\therefore\bfetax(\brU^0)\rsqb \rcb } \suite
    + \oostr^2\rhoz \lsqb \brU^2 + \bfbsh\dip\bfnax\brU^1 + \Divx\lcb \bfBii\dip\bfnax\brU^0\rcb \rsqb + (\oostr^2\rhoz)^2\brcii\sip\brU^0 = \bfze
\label{macroscopic balance u2}
\end{align}
where the homogenized second-order inertial tensors $\bfBii$ and $\brcii$ are given by
\begin{equation}
   \bfBii = \lbra \bfphc^2 \rbra_{\Ycal}, \quad\text{i.e.}\quad \Bii_{i\ell km} = \lbra \phi^2_{i\ell km} \rbra_{\Ycal}; \qquad
  \brcii = \lbra \bfzet^2 \rbra_{\Ycal}, \quad\text{i.e.}\quad \cii_{ik} = \lbra \beta\zeta^{2k}_i \rbra_{\Ycal}.
  \label{bfBii:def}
\end{equation}


\subsection{Reciprocity relationships}
\label{sec:recip}
Like in the elastostatic homogenization case, the reciprocity identity of Lemma~\ref{lem:recip} allows to derive alternative expressions for the inertial effective tensors given by~\eqref{bfbsh:def} and~\eqref{bfBii:def} arising in the time-harmonic homogenization procedure. They are given in the following proposition, whose proof is provided in Section~\ref{sec:proof:recip}:\enlargethispage*{1ex}

\begin{prop}\label{prop:recip}
The first-order inertial tensor $\bfbsh$ is given in terms of $\bfchi^1$ by
\begin{equation}
  \bsh_{ik\ell}
  = \inv{|\Ycal|} \iY \beta \lpar \chi^{1k\ell}_i - \chi^{1i\ell}_k \rpar \dy
  = \lbra \beta \chi^{1k\ell}_i \rbra_{\Ycal} - \lbra \beta \chi^{1i\ell}_k \rbra_{\Ycal}. \label{bsh:expr}
\end{equation}
The second-order inertial tensor $\bfBii$ is given in terms of $\bfchi^1$ and $\bfchi^2$ by
\begin{equation}
  \Bii_{i\ell km}
 = \inv{|\Ycal|} \iY \beta
     \lsqb \chi^{2mk\ell}_i + \chi^{2i\ell m}_k - \chi^{1i\ell}_r \chi^{1km}_r \rsqb \dy. \label{bii:expr}
\end{equation}
Finally, the tensor $\brcii$ is given in terms of $\bfzet^2$ by
\begin{equation}
  |\Ycal|\cii_{ik} = A(\bfzet^{2i},\bfzet^{2k}) \qquad (1\shleq n,p\shleq d). \label{c2:recipr}
\end{equation}
and is therefore positive definite unless the mass density is homogeneous, in which case $\brcii\sheq\bfze$ (see Remark~\ref{rem:zeta2}).
\end{prop}

From the above results, we see in particular that $\bfbsh$ has the skew-symmetry property
\begin{equation}
  \bsh_{ik\ell} = -\bsh_{ki\ell} \label{bsh:sym}
\end{equation}
and can be computed from the first cell solution $\bfchi^1$ instead of $(\bfchi^1,\bfzet^2$), $\bfBii$ has the major symmetry
\begin{equation} \label{eq:B2:c2:sym}
  \Bii_{i\ell km}= \Bii_{kmi\ell}, 
\end{equation}
and can be computed from the two cell solutions $(\bfchi^1,\bfchi^2)$ instead of $(\bfchi^1,\bfzet^2,\bfzet^3$), and $\brcii$ is symmetric:
\begin{equation} \label{eq:c2:sym}
  \cii_{ik} = \cii_{ki}.
\end{equation}
Similarly to the elastostatic case, the third inertial cell problem is seen to be needed in the homogenization process but the practical evaluation of the effective tensors does not use its solution $\bfzet^3$. The inertial solution $\bfzet^2$ is therefore the only additional cell function to be computed in practice. Moreover, we remark for future use that $\bfzet^2$ is actually used only in the evaluation of the effective tensor $\brcii$, since $\bfbsh$ and $\bfBii$ can be computed from the relative mass density $\beta$ and the elastostatic cell solutions $\bfchi^1,\bfchi^2$.


\subsection{Effective wave equation}
\label{sec:mean}

We define the second-order approximation $\brU^{\epsilon}(\brx)$ of the macroscopic displacement field as:
\begin{equation}
  \brU^{\epsilon}(\brx) = \brU^0(\brx) + \epsilon \brU^1(\brx)  + \epsilon^2 \brU^2(\brx) + o(\epsilon^2), \label{Ueps:def}
\end{equation}
whose coefficients $\brU^0(\brx),\,\brU^1(\brx),\,\brU^2(\brx)$ satisfy the macroscopic balance equations~\eqref{macroscopic balance u0} (leading order), \eqref{macroscopic balance u1} (first order) and~\eqref{macroscopic balance u2} (second order) derived in Section~\ref{sec:cascade}, recalled here for convenience:
\begin{equation} 
\begin{aligned} 
 \text{(a) \ }\bfze &= \Divx \lcb \brCz \dip \bfepsx(\brU^0) \rcb + \oostr^2\rhoz \brU^0 + \brFstr, \\
 \text{(b) \ }\bfze &= \Divx \lcb \brCz \dip\bfepsx(\brU^1) + \brCi\therefore \bfetax(\brU^0) \rcb
   + \oostr^2\rhoz \lsqb \brU^1 + \bfbsh\dip\bfnax\brU^0 \rsqb, \\
  \text{(c) \ }\bfze &= \Divx \lcb \brCz\dip\dip\bfepsx(\brU^2) + \brCi\therefore\bfetax(\brU^1)
   + \Divx\lsqb \brCii\therefore\bfetax(\brU^0)\rsqb \rcb \\ & \qquad
    + \oostr^2\rhoz \lsqb \brU^2 + \bfbsh\dip\bfnax\brU^1 + \Divx\lcb \bfBii\dip\bfnax\brU^0\rcb \rsqb + (\oostr^2\rhoz)^2\brcii\sip\brU^0
\end{aligned} \label{mean:eqs:init}
\end{equation}
Like in the elastostatic case, and keeping in mind that any resulting effective PDE needs only be satisfied at order $\Ocal(\epsilon^2)$, effective wave equations for $\brU^{\epsilon}(\brx)$ would \emph{a priori} follow from the combination~(\ref{mean:eqs:init}a)+$\epsilon$(\ref{mean:eqs:init}b)+$\epsilon^2$(\ref{mean:eqs:init}c) of the macroscopic balance equations. The foregoing combination however raises two issues. Firstly, the tensor coefficients of the PDE emerging from the above-defined combination may, like in Section~\ref{sec:mean:zf}, fail to possess satisfactory sign and ellipticity properties, a shortcoming that as before is to be remedied by recasting~(\ref{mean:eqs:init}c) into an equivalent alternative form prior to deriving an effective PDE. Secondly, the presence of the $(\oostr^2\rhoz)^2$ factor in~(\ref{mean:eqs:init}c) indicates that the obtained effective PDE would \emph{a priori} be of fourth order in time, a form that do not conform to the SGE format and also lacks solvability guarantees even with suitably-behaved tensor coefficients (an issue also raised by \cite{Allaire2022,allaire:24}). For those reasons, we next pursue a reformulation of~(\ref{mean:eqs:init}c) that leads to an effective PDE that is of second order in time and endows tensor coefficients with satisfactory sign properties. This version will then be shown in Section~\ref{sec:homo:SGE} to define a valid strain-gradient homogenized model and to lead to well-posed transient evolution problems. The  discussion of a fourth-order-in-time effective PDE is deferred to Section~\ref{sec:fourth:order:in:time}.\enlargethispage*{3ex}




Proceeding along lines similar to Sec.~\ref{sec:mean:zf}, we take the Laplacian of~(\ref{mean:eqs:init}a), to obtain
\begin{equation}
  -\Divx\lcb \Divx\lsqb \brAth\therefore\bfeta(\brU^0) \rsqb\rcb - \oostr^2\rhoz\bfDelx\brU^0 - \bfDelx\brFstr = \bfze. \label{aux23}
\end{equation}
We then derive an additional identity by applying the operator $-\Divx \lcb \brCz \dip \bfepsx\lpar \brcii\sip(\dotp)\rpar \rcb - \oostr^2\rhoz\lpar \brcii\sip(\dotp)\rpar$ to~(\ref{mean:eqs:init}a) and obtain
\begin{equation}
  - \Divx\lcb \Divx\lsqb \bfcii\therefore\bfetax(\brU^0) \rsqb\rcb
  - \oostr^2\rhoz \Divx\lcb \bfbii\dip\bfna\brU^0 \rcb - (\oostr^2\rhoz)^2 \brcii\sip\brU^0
  - \brcii\sip\Divx\lcb \brCz \dip \bfepsx(\brFstr) \rcb - \oostr^2\rhoz\brcii\sip\brFstr = \bfze,
	 \label{aux25}
\end{equation}
where the 6th order tensor $\bfcii$ and the 4th order tensor $\bfbii$, given using component notation by
\begin{equation}
  \rmcii_{ijnk\ell m} = \cii_{ba}\CZ_{anij} \CZ_{bmk\ell} , \qquad
  \bii_{i\ell km} = \cii_{ia} \CZ_{a\ell km} + \cii_{ka} \CZ_{ami\ell}.
   \label{b2:c2:def}
\end{equation}
are defined in order to satisfy
\begin{equation}
\begin{aligned}
  \Divx \lcb \brCz\dip\bfepsx\lsqb\brcii\sip\Divx\lpar\brCz \dip \bfepsx(\brU)\rpar\rsqb \rcb
 &= \Divx\lcb \Divx\lsqb \bfcii\therefore\bfetax(\brU) \rsqb\rcb \\
  \Divx\lcb \brCz \dip \bfepsx(\brcii\sip\brU^0) \rcb + \brcii\sip\Divx\lcb \brCz \dip \bfepsx(\brU^0) \rcb
 &= \Divx\lcb \bfbii\dip\bfna\brU^0 \rcb
\end{aligned}
\end{equation}
for any field $\brU(\brx)$.
We then add to~(\ref{mean:eqs:init}c) the combination $\vartheta$\eqref{aux23}+\eqref{aux25}, where $\vartheta$ is like in Section~\ref{sec:mean:zf} an adjustable weight parameter. This leads to the alternative form
\begin{multline}
  \bfze = \Divx \lcb \brCz\dip\dip\bfepsx(\brU^2) + \brCi\therefore\bfetax(\brU^1)
   - \Divx\lsqb \brAii(\vartheta)\therefore\bfetax(\brU^0)\rsqb \rcb \\
    + \oostr^2\rhoz \lsqb \brU^2 + \bfbsh\dip\bfnax\brU^1 - \Divx\lcb \brJii(\vartheta)\dip\bfnax\brU^0\rcb \rsqb
    - \vartheta\bfDelx\brFstr - \brcii\sip\Divx\lcb \brCz \dip \bfepsx(\brFstr) \rcb - \oostr^2\rhoz\brcii\sip\brFstr
    \label{mean:eq:order2:var}
\end{multline}
of the $\Ocal(\epsilon^2)$ balance equation~(\ref{mean:eqs:init}c) where the $\vartheta$-dependent tensors $\brAii(\vartheta)\in S^2\lpar S^2(\Rbb^d)\tens\Rbb^d \rpar$ and $\brJii(\vartheta)\in S^2(\Rbb^d\tens\Rbb^d)$ are defined by
\begin{equation}
  \brAii(\vartheta) = \vartheta\brAth - \bfCii + \bfcii = \brAZii(\vartheta) + \bfcii , \qquad
  \brJii(\vartheta) = \vartheta\brI^2 - \bfBii + \bfbii \label{A2:def},
\end{equation}
where the elastostatic correction $\brAth$ and tensor $\brAZii(\vartheta)$ are given by \eqref{Ath:def} and \eqref{zf:A:def}, $\brI^2$ is the identity on second-order tensors (i.e. $\rmI^2_{i\ell km}=\delta_{ik} \delta_{\ell m}$), and the tensors $\bfcii$ and $\bfbii$ are given by~\eqref{b2:c2:def}.
The foregoing manipulation is thus seen to convert~(\ref{mean:eqs:init}c) into an equivalent form~\eqref{mean:eq:order2:var} that involves only second-order time derivatives and features tensors possessing suitable symmetries.





We then proceed to derive the sought effective PDE. Writing the relevant combination~(\ref{mean:eqs:init}a)+$\epsilon$(\ref{mean:eqs:init}b)+$\epsilon^2$\eqref{mean:eq:order2:var} of the macroscopic balance equations, the second-order approximation~\eqref{Ueps:def} of the macroscopic displacement is found to satisfy the generalized wave equation
\begin{equation}
  \sfQ_{\rmx}\lsqb \brCz,\epsilon\brCi,\epsilon^2\brAii(\vartheta) \rsqb\brU
  - \oostr^2 \rhoz \sfR_{\rmx}\lsqb 1,\epsilon \bfbsh,\epsilon^2 \brJii(\vartheta) \rsqb\brU = \brFstrth \label{SGE:freq}
\end{equation}
up to a $\Ocal(\epsilon^3)$ residual,
using again the generic PDOs introduced in~\eqref{PDO:generic} and with the 
loading $\brFstrth$ defined by
\begin{equation}
  \brFstrth
 := \brFstr - \epsilon^2 \lsqb \vartheta\bfDelx\brFstr + \brcii\sip\Divx\lcb \brCz \dip \bfepsx(\brFstr) \rcb + \oostr^2\rhoz\brcii\sip\brFstr \rsqb
\label{F:def}
\end{equation}
Then, applying the inverse Fourier transform to~\eqref{SGE:freq} yields the transient, second-order in time, effective equation
\begin{equation}
  \sfQ_{\rmx}\lsqb \brCz,\epsilon\brCi,\epsilon^2\brAii(\vartheta) \rsqb\brU
  + \rhoz \sfR_{\rmx}\lsqb 1,\epsilon\bfbsh,\epsilon^2\brJii(\vartheta) \rsqb\brU'' = \brFstrth, \label{SGE:time}
\end{equation}
whose right-hand side features second-order derivatives in both space and time of the body force density $\brFstr$.

\begin{remark}\label{rem:c2:pos:def}
The sixth-order "corrective" tensor $\brc^2$ has the same major and minor symmetries as $\brCii$ and $\brD^2$ (which are therefore shared by the effective tensor $\brA^2(\vartheta)$), and is positive as a consequence of $\brcii$ being positive (see Prop. \ref{prop:recip}). Indeed one has:
\begin{equation}
\bfeta \therefore \brc^2 \therefore \bfeta = \bfv\cdot\brcii\cdot\bfv \geq 0 \quad \text{ with \ } v_a = \CZ_{anij} \eta_{nij} \qquad \text{for any \ }\bfeta\in S(\Rbb^d)\tens\Rbb^d.    
\end{equation}
The fourth-order corrective tensor $\brb^2$ has the same major symmetry as $\brB^2$, and therefore so does the effective tensor $\brJ^2(\vartheta)$. Unlike $\brc^2$, however, $\brb^2$ is in general sign-indefinite. Indeed, let $\bfw\shin \Lambda^2(\Rbb^d)$ and let $\bfga(q)=\bfw\sip\brcii\shm\brcii\sip\bfw\shp 2q\bfw$ for $q\shin\Rbb$, noting that $\bfga\Esub\shdeq\bfw\sip\brcii\shm\brcii\sip\bfw$ and $\bfga\Osub(q)\shdeq 2q\bfw$ are the symmetric and skew-symmetric parts of $\bfga(q)$, with $\bfga\Esub\not=\bfze$ except if $\brcii$ is isotropic. Since $\brCz\dip\bfw=\bfze$, we then obtain
\[
  g(q) := 
  \bfga(q)\dip\brb^2\dip\bfga(q)
  = (\bfga\Esub\sip\brcii)\dip\brCz\dip\bfga\Esub + q(2\bfw\sip\brcii)\dip\brCz\dip\bfga\Esub
  = (\bfga\Esub\sip\brcii)\dip\brCz\dip\bfga\Esub + q\bfga\Esub\dip\brCz\dip\bfga\Esub,
\]
since $2\bfw\sip\brcii\shm\bfga\Esub=\bfw\sip\brcii\shp\brcii\sip\bfw$ is skew-symmetric. Hence $g(\Rbb)\sheq\Rbb$, implying that $\brb^2$ is sign-indefinite. The foregoing argument fails only if $\brcii$ is isotropic, in which case $\brb^2$ is a positive multiple of $\brCz$, hence is positive.
\end{remark}


\subsection{Consistency with the elastostatic effective equation}
\label{sec:consist:zf}

Since $\brAii(\vartheta)\not=\brAZii(\vartheta)$ (a fact further discussed in Section \ref{sec:statics:dynamics}), the stiffness PDO featured in the dynamic PDEs~\eqref{SGE:freq} and~\eqref{SGE:time} differs from that used in the static PDE~\eqref{zf:eff}. The consistency between the dynamic and static effective PDEs thus needs to be ascertained.


To this aim, we first apply the operator $-\Divx \lcb \brCz \dip \bfepsx\lpar \brcii\sip(\dotp)\rpar \rcb$ to~(\ref{mean:eqs:init}a), to obtain
\begin{equation}
  - \Divx\lcb \Divx\lsqb \bfcii\therefore\bfetax(\brU^0) \rsqb\rcb
  - \oostr^2\rhoz\Divx\lcb \brCz \dip \bfepsx(\brcii\sip\brU^0) \rcb - \brcii\sip\Divx\lcb \brCz \dip \bfepsx(\brFstr) \rcb = \bfze, \label{aux27}
\end{equation}
Then, since $\bfcii=\brAii(\vartheta)\shm\brAZii(\vartheta)$ by~\eqref{A2:def}, we find
\begin{multline}
  \Divx\lcb \Divx\lsqb\brAii(\vartheta)\therefore\bfetax(\brU^0) \rsqb\rcb \\
 = \Divx\lcb \Divx\lsqb\brAZii(\vartheta)\therefore\bfetax(\brU^0) \rsqb\rcb
  - \oostr^2\rhoz\Divx\lcb \brCz \dip \bfepsx(\brcii\sip\brU^0) \rcb - \brcii\sip\Divx\lcb \brCz \dip \bfepsx(\brFstr) \rcb,
\end{multline}
which, used in~\eqref{mean:eq:order2:var}, produces the equality
\begin{multline}
  \bfze = \Divx \lcb \brCz\dip\dip\bfepsx(\brU^2) + \brCi\therefore\bfetax(\brU^1)
   - \Divx\lsqb \brAZii(\vartheta)\therefore\bfetax(\brU^0)\rsqb \rcb \\
    + \oostr^2\rhoz \lsqb \brU^2 + \bfbsh\dip\bfnax\brU^1
    - \Divx\lcb \brJii(\vartheta)\dip\bfnax\brU^0 - \brCz \dip \bfepsx(\brcii\sip\brU^0) \rcb \rsqb
    - \vartheta\bfDelx\brFstr + \oostr^2\rhoz\brcii\sip\brFstr
    \label{mean:eq:order2:mod}
\end{multline}
Setting $\oostr=0$ in the $\Ocal(\epsilon^2)$ effective equation produced by the combination~(\ref{mean:eqs:init}a)+$\epsilon$(\ref{mean:eqs:init}b)+$\epsilon^2$\eqref{mean:eq:order2:mod} thus restores the elastostatic effective equation~\eqref{zf:eff}.\enlargethispage*{1ex}

\section{Homogenized strain gradient elasticity models}
\label{sec:homo:SGE}

With the static and dynamic high-order effective PDEs and all associated effective tensors now in hand, we address in this section their interpretation as SGE model components. After summarizing in Section~\ref{sec:SGE} the SGE framework, we show that the foregoing effective PDEs define valid SGE models in both the static case (Section~\ref{sec:compar:zf}) and the dynamic case (Section~\ref{sec:compar}), their stiffness and inertial components being in particular proved therein to have the requisite sign definiteness and symmetry properties for suitably chosen values of the weight $\vartheta$. In particular, those properties crucially allow to establish the well-posedness of evolution problems in unbounded media based on the obtained transient effective PDE (Section~\ref{sec:transient_equation}).

\subsection{Strain-gradient elasticity}
\label{sec:SGE}

We summarize in this section the SGE framework, and introduce for later reference its associated variational version.  
In the case of SGE theory (of Mindlin's Type II used in this work) the strain and kinetic energy densities $\Ecal$ and $\Kcal$ are taken to be functions of the displacement and its first and second derivatives given by
\begin{equation}
  \Kcal := \tdemi (\brp\sip\bru') + \tdemi(\brq\dip\bfna\bru'), \qquad
  \Ecal := \tdemi (\bfsig\dip\bfeps) + \tdemi(\bftau\therefore\bfeta), \label{ek:SGE}
\end{equation}
where the strain $\bfeps=\bfeps(\bru)$ and strain gradient $\bfeta=\bfeta(\bru)$ are defined as in~\eqref{diffops:def} and ~\eqref{diffops:def:eta:kappa}, $\bfsig$ is the Cauchy stress tensor, $\brp$ is the momentum (vector) field, $\brq$ is the hyper-momentum tensor field (of order 2) and $\bftau$ is the hyper-stress tensor field (of order 3). The total stress $\brs$ and the total momentum $\bfpi$ are then defined as
\begin{equation}
  \brs = \bfsig - \Div\bftau - \bfupph, \qquad \bfpi = \brp - \Div\brq, \label{sge 2}
\end{equation}
where $\bfupph$ is the tensor field (of order 2) of double body forces to which the medium may be subjected in addition to the usual (vector) body force density $\brf$~\cite{germain:73}.
The least action principle~\cite{mindlin:68} then
yields the balance field equations:
\begin{equation}
	\Div\brs + \brf = \bfpi', \qquad\text{i.e.} \quad
	\Div(\bfsig - \Div\bftau) + \brf - \Div\bfupph = \brp' - \Div\brq'.
\end{equation}
In addition, the constitutive relations for this model are
\begin{equation}
\begin{Bmatrix} \brp \\ \brq \\ \bfsig \\ \bftau \end{Bmatrix}
= \begin{bmatrix}
	\rhoSG \brI & \frac{1}{2}\brK & \bfze &\bfze \\
	\frac{1}{2} \brK^{\text{T}} & \brJ & \bfze & \bfze\\
	\bfze & \bfze &\brC & \frac{1}{2} \brM \\
	\bfze & \bfze & \frac{1}{2} \brM^{\text{T}} & \brA
  \end{bmatrix}
  \begin{Bmatrix} \bru' \\ \bfna\bru' \\ \bfeps \\ \bfeta \end{Bmatrix}, \label{SGE:const}
\end{equation}
in which $\brC$ is the usual elasticity tensor, $\brM$ and $\brA$ are the coupling and second-order elasticity tensors (of tensorial order 5 and 6, respectively), $\rhoSG\brI$ is the inertia tensor (of order 2), and $\brK$ and $\brJ$ are the coupling and second-order inertia tensors (of tensorial order 3 and 4, respectively).
The index symmetries of the constitutive tensors are determined by the structure of the tensor spaces they inhabit, i.e.
\begin{equation}
\begin{aligned}
    \brI &\in S^{2}(\Rbb^d), &\quad \brK &\in \Rbb^d\otimes\Rbb^d\otimes\Rbb^d, &\quad
    \brJ &\in S^{2}(\Rbb^d\otimes\Rbb^d) \\
    \brC &\in S^{2}(S^{2}(\Rbb^d)), &\quad
    \brM &\in S^{2}(\Rbb^d)\otimes S^{2}(\Rbb^d)\otimes\Rbb^d, &\quad
    \brA &\in  S^{2}(S^{2}(\Rbb^d)\otimes\Rbb^d).
\end{aligned}
\end{equation}
On substituting the constitutive equations~\eqref{SGE:const} into the field equations~\eqref{sge 2} and expressing all quantities in terms of the displacement and its derivatives, $\bru$ is finally found to obey the (linear, fourth-order in space) generalized wave equation\enlargethispage*{1ex}
\begin{equation}
  \sfQ\bru + \sfR\bru'' = \brf - \Div\bfupph \qquad \text{with} \quad
  \sfQ := \sfQ[\brC,\brMsh,\brA], \quad \sfR := \sfR[\rhoSG,\brKsh,\brJ], \label{SGE:PDE}
\end{equation}
having used once again the generic PDOs~\eqref{PDO:generic} and with the coupling tensors $\brKsh,\brMsh$ defined in component form by
\begin{equation}
  \rmK_{ijk}^{\sharp} := \pinv{2} \left( \rmK_{ijk}-\rmK_{jik} \right), \qquad
  \rmM_{ijk\ell m}^{\sharp} := \pinv{2}  \left( \rmM_{ijk\ell m}-\rmM_{k\ell ijm} \right). \label{SGE:sharp}
\end{equation}
Letting $\Lambda^{2}(\mathbb{V})$ denote the tensor space generated by the antisymmetric tensor product of $\mathbb{V}$ with itself, the (partially skew-symmetric) coupling tensors $\brKsh,\brMsh$ thus verify $\brKsh\in \Lambda^{2}(\Rbb^d)\otimes\Rbb^d$ and $\brMsh\in \Lambda^{2}(S^2(\Rbb^d))\otimes\Rbb^d$.


\begin{remark}
Compared to previous works, \eg \cite{auffray:15,rosi:24}, the constitutive relation \eqref{SGE:const} features additional $\pinv{2}$ factors in front of the heuristic coupling tensors $\brK,\,\brM$, so that the tensors $\brK^\sharp,~\brM^\sharp$ featured in the SGE balance equation and analysis to follow are the skew-symmetric parts of $\brK,\,\brM$ as per~\eqref{SGE:sharp}.
\end{remark}

\subsubsection*{Weak formulation}

Well-posedness results and numerical approximation methods such as the FEM for initial-value problems based on the evolution PDE~\eqref{SGE:PDE}, and similarly for boundary-value problems involving time-harmonic or time-independent versions of~\eqref{SGE:PDE}, rest on their variational formulation~\cite{germain:73}. To this end, in keeping with the present unbounded-domain setting, we take the scalar product of~\eqref{SGE:PDE} with a vector-valued test function $\brw$ belonging to the space $\textbf{C}^{\infty}_0(\Rbb^d)$ of vector-valued smooth functions which vanish, together with all their derivatives, outside some bounded subset of $\Rbb^d$ (i.e. are test functions in the sense of the theory of distributions).
The weak form of~\eqref{SGE:PDE} at any time instant is found as
\begin{equation}
  \Qcal(\bru,\brw) + \Rcal(\bru'',\brw) = \Fcal(\brw) \qquad \text{for all }\brw\in\textbf{C}^{\infty}_0(\Rbb^d), \label{SGE:weak}
\end{equation}
where $(\mathbf{a},\mathbf{b})$ denotes the $\bfL^2(\Rbb^d)$ scalar product of tensor fields $\mathbf{a},\mathbf{b}$ of the same order, and the bilinear forms $\Qcal,\Rcal$ and the linear functional $\Fcal$ are defined by
\begin{equation}
\begin{aligned}
  \Qcal(\bru,\brw)
 &= \iRd \lpar \bfeps(\brw)\dip\brC\dip\bfeps(\bru) + \bfeps(\brw)\dip\brMsh\therefore\bfeta(\bru)
 + \bfeta(\brw)\therefore \brA\therefore\bfeta(\bru) \rpar \dx, \\
  \Rcal(\bru,\brw)
 &= \iRd \rhoSG \lpar \brw\sip\bru + \brw\sip\brKsh\dip\bfna\bru + \bfna\brw\dip\brJ\dip\bfna\bru \rpar \dx \\
  \Fcal(\brw)
 &= \iRd \lpar \brf\sip\brw + \bfupph\dip\bfna\brw \rpar \dx
\end{aligned} \label{PQ:CGE:bil:def}
\end{equation}
Obtaining the foregoing formulation relies on the fact that, for any $\brw\in \textbf{C}^{\infty}_0(\Rbb^d)$, integrations by parts produce no boundary terms and thus yield the Green identities
\begin{equation}
  \text{(a) }\lpar \sfQ\bru,\brw \rpar = \Qcal(\bru,\brw), \qquad
  \text{(b) }\lpar \sfR\bru,\brw \rpar = \Rcal(\bru,\brw) \label{green:AB}
\end{equation}
and (by the skew-symmetry properties~\eqref{SGE:sharp} of $\brKsh$ and $\brMsh$) the symmetry properties
\[
  \lpar \brKsh\dip\bfna\bru, \brw \rpar
  = \lpar \brKsh\dip\bfna\brw, \bru \rpar, \qquad
  \lpar \brMsh\therefore \bfeta(\bru), \bfna\brw \rpar
  = \lpar \brMsh\therefore \bfeta(\brw), \bfna\bru \rpar.
\]
The above equalities ensure that the bilinear forms $\Qcal$ and $\Rcal$ are symmetric. In addition, the defining expressions~\eqref{PQ:CGE:bil:def} imply that $\Qcal$ and $\Rcal$ extend to bilinear forms that are  well-defined and continuous on $\bfH^2(\Rbb^d)$. The latter property follows from applying the Cauchy-Schwarz inequality to the integrals in~\eqref{PQ:CGE:bil:def} and using the definition of $\|\bru\|^2_{\brH^2(\Rbb^d)}$ to majorize the resulting upper bounds, from which the inequalities
\begin{equation}
  \Qcal(\bru,\brw) \leq C\|\bru\|_{\brH^2(\Rbb^d)}\,\|\brw\|_{\brH^2(\Rbb^d)}, \qquad
  \Rcal(\bru,\brw) \leq C\|\bru\|_{\brH^1(\Rbb^d)}\,\|\brw\|_{\brH^1(\Rbb^d)} \leq C\|\bru\|_{\brH^2(\Rbb^d)}\,\|\brw\|_{\brH^2(\Rbb^d)} \label{AB:cont}
\end{equation}
hold for some constant $C>0$. As a result, the weak form~\eqref{SGE:weak} of~\eqref{SGE:PDE} extends, by the known density of $\textbf{C}^{\infty}_0(\Rbb^d)$ in $\bfH^2(\Rbb^d)$, to test functions $\brw\in\bfH^2(\Rbb^d)$.

The variational form~\eqref{SGE:weak} of the generalized transient wave equation~\eqref{SGE:PDE} expresses the stationarity of the functional
\begin{equation}
    \Lcal(\bru,\bru')
 := \iT \Lcb \int_{\Rbb^d} \rmk(\bru',\bfna\bru') \dx - \int_{\Rbb^d} \rme\lpar \bfeps(\bru),\bfeta(\bru) \rpar \dx
 + \Fcal(\bru) \Rcb \dt,
\end{equation}
with the strain and kinetic energy
densities $\rmk$ and $\rme$ given by
\begin{equation}
\text{(a) \ }
  2\rme(\bfeps,\bfeta)
 := \bfeps\dip\brC\dip\bfeps + \bfeps\dip\brMsh\therefore\bfeta + \bfeta\therefore \brA\therefore\bfeta, \qquad
\text{(b) \ }
  2\rmk(\bru,\bfga)
 := \rhoSG \bru\sip\bru + \bru\sip\brKsh\dip\bfga + \bfga\dip\brJ\dip\bfga.
\label{ke:SGE:def}
\end{equation}
\begin{remark}\label{ke:SGE:rem}
To define positive definite forms of their arguments, the higher-order tensors entering the above energy densities require the tensors $\brA-\tquart\brMsh{}\Tsup\dip\brC^{-1}\dip\brMsh\in S^2\lpar S^2(\Rbb^d)\tens\Rbb^d \rpar$ and $\brJ-\tquart\rhoSG^{-1}\brKsh{}\Tsup\sip\brKsh\in S^2\lpar \Rbb^d\tens\Rbb^d \rpar$ to be positive definite, in addition to the usual sign conditions on $\brC$ and $\rhoSG$. The first condition was used in \cite{nazarenko:21} to establish positive definiteness conditions for the elastic energy density of hemitropic SG materials; the second is its inertial-tensor analogue.
\end{remark}

\begin{remark}
\label{rem:energy:vs:PDE}
The kinetic and strain energy densities $\rmk,\rme$ associated with the PDE~\eqref{SGE:PDE}, given by~\eqref{ke:SGE:def}, generally differ from the original densities~\eqref{ek:SGE} which, using the constitutive relations \eqref{SGE:const}, are found to be given by
\begin{equation}
  2\Ecal
 = \bfeps\dip\brC\dip\bfeps + \bfeps\dip\brM\therefore\bfeta + \bfeta\therefore \brA\therefore\bfeta. \qquad   2\Kcal
 = \rhoSG \bru'\sip\bru' + \bru'\sip\brK\dip \bfna\bru' + \bfna\bru'\dip \brJ  \dip \bfna\bru',
 \label{ek:SGE:full}
\end{equation}
In particular, $\rmk$ and $\rme$ involve only the skew-symmetric part of $\brK$ and $\brM$, respectively, the contributions of their symmetric part to the total energies $\Rcal$ and $\Qcal$ vanishing under the foregoing integration-by-parts process based on compactly-supported test functions. For bounded domains $\OO$, the variational approach must be modified by carefully accounting for boundary contributions (except when homogeneous kinematic boundary conditions make $\bfH^2_0(\OO)$ the relevant variational solution space), with terms missing in the "PDE energy" expected to re-appear as boundary terms. Similar reconsideration involving the proper definition of radiation conditions, not yet available for SGE to our best knowledge, is likewise in order to address time-harmonic SGE in unbounded media.

\end{remark}


\subsection{Homogenized elastostatic strain-gradient model}
\label{sec:compar:zf}

We now examine the interpretation of the elastostatic effective PDE~\eqref{zf:eff} as a SGE model. To this aim, \eqref{zf:eff} is set to fully-dimensional form by reversing~\eqref{oo:F:stretch} and invoking the scaling identity~\eqref{PDO:scaling} verified by the PDO (together with $\epsilon^2\Delta_\brx=\ell^2\Delta_{\brX}$). The dimensional effective PDE
is obtained as
\begin{equation}
   \sfQZth\brU^{\epsilon} = (1-\vartheta\ell^2\Delta_\brX) \brF, \qquad\text{with}\quad
   \sfQZth := \sfQ_\brX\lsqb \brCz,\ell\brCi,\ell^2\brAZii(\vartheta) \rsqb\brU^{\epsilon}, \label{zf:SGE:eff}
\end{equation}
and is readily seen to be of the same form as the time-independent version of the strain-gradient PDE~\eqref{SGE:PDE} if the SGE tensors and the body force and double force densities are chosen as
\begin{equation}
  \brC = \brCz, \qquad \brMsh := \ell \brCi, \qquad \brA = \brA(\vartheta) := \ell^2 \brAZii(\vartheta), \qquad 
  \brf := \brF, \qquad \bfupph = \bfupph(\vartheta) := \vartheta\ell^2\bfna\brF \label{zf:corr:tens}
\end{equation}

As expected, (i) the tensor-valued coefficients of~\eqref{zf:SGE:eff} possess the symmetries expected of those featured in~\eqref{SGE:PDE}, (ii) the first-order and second-order tensors scale with $\ell$ and $\ell^2$ compared to their homogenized counterparts computed on the reference unit cell $\Ycal$. We emphasize that, unlike $\brC$ and $\brMsh$, the resulting SG elasticity tensor $\brA=\brA(\vartheta)$ is not uniquely determined by the homogenization process since it also depends on the adjustable parameter $\vartheta$, and the same comment applies to the double force $\bfupph$.



As a consequence, the strain-gradient effective PDE can be set in variational form as
\begin{equation}
  \rmQZth(\brU,\brW)
 = \rmFth(\brW) \qquad \text{for all }\brW\in\textbf{C}^{\infty}_0(\Rbb^d), \label{zf:weak}
\end{equation}
where the bilinear form $\rmQZth$ and the linear functional $\rmFth(\brW)$ are defined by the first and last of~\eqref{PQ:CGE:bil:def} with the replacements~\eqref{zf:corr:tens} (hence their dependence in $\vartheta$, emphasized by the notation).
%
The variational equation~\eqref{zf:weak} expresses the stationarity of the functional
\begin{equation}
    \Lcal_{\vartheta}(\brU)
 := \int_{\Rbb^d} \eZ_{\vartheta}\lpar \bfeps(\brU),\bfeta(\brU) \rpar \dV - \rmFth(\brU),
\end{equation}
with 
the strain energy density function 
$\eZ_{\vartheta}(\bfeps,\bfeta)$ given for arbitrary tensors $\bfeps\shin S^{2}(\Rbb^d)$ and $\bfeta\shin S^{2}(\Rbb^d)\tens\Rbb^d$ by~(\ref{ke:SGE:def}a) with replacements~\eqref{zf:corr:tens}, i.e.
\begin{equation}
  2\eZ_{\vartheta}(\bfeps,\bfeta) := \bfeps\dip\brCz\dip\bfeps + \eps\bfeps\dip\brCi\therefore\bfeta + \eps^2\bfeta\therefore\brAZii(\vartheta)\therefore\bfeta. \label{eZ:def}
\end{equation}
For the effective PDE~\eqref{zf:SGE:eff} to define a physically valid homogenized strain-gradient model, its coefficients must be such that the strain energy density~\eqref{eZ:def} is positive definite. The following result on the density function $\eZ_{\vartheta}$ and the bilinear form $\rmQZth$ (whose proof is given in Section~\ref{sec:proof:coercivity:zf:fourier}) states that this is indeed the case provided $\vartheta$ is properly adjusted:
\begin{prop}\label{zf:coercivity}
There exists a finite positive value $\thZ_e$ of $\vartheta$, which does not depend on $\eps$ and is such that the following properties hold for any $\vartheta>\thZ_e$:
\begin{compactenum}[(a)]
\item The quadratic form $(\bfeps,\bfeta)\mapsto \eZ_{\vartheta}(\bfeps,\bfeta)$ defined by~\eqref{eZ:def} is positive definite on $S^{2}(\Rbb^d) \times \lpar S^{2}(\Rbb^d)\tens\Rbb^d \rpar$;
\item The bilinear form $\rmQZth$ of~\eqref{zf:weak} is positive and $\brH^2(\Rbb^d)$-coercive (in the terminology of e.g.~\cite{aubin}), as it verifies
\begin{equation}
  \rmQZth(\brU,\brU) \geq C(\vartheta) \lpar \|\brU\|^2_{\brH^2(\Rbb^d)} - \|\brU\|^2_{\brL^2(\Rbb^d)} \rpar \geq 0
  \qquad \text{for all } \brU\in\brH^2(\Rbb^d) \text{ and some }C=C(\vartheta)>0.
\end{equation}
\end{compactenum}
\end{prop}

While boundary-value problems (BVPs) based on effective PDE~\eqref{zf:eff} are not studied in this work, Proposition~\ref{zf:coercivity}, or variations thereof dictated by the boundary condition setting, potentially plays a major role in assessing their well-posedness. Some results of this kind are available for SGE BVPs unrelated to homogenization, see e.g.~\cite{eremeyev:23}. Exterior BVPs for unbounded media would require versions of Prop.~\ref{zf:coercivity} using weighted (at infinity) forms of the $\brH^2(\Rbb^d)$ norm, as done in~\cite{amrouche:04} for exterior problems involving the (fourth-order) biharmonic equation. Here, Prop.~\ref{zf:coercivity} is mainly intended for completeness and as a precursor of corresponding results for the elastodynamic case that will directly benefit this work.

\subsection{Homogenized elastodynamic strain-gradient model}
\label{sec:compar}

Proceeding similarly to the elastostatic case, we now show that the dynamic effective PDE~\eqref{SGE:time} yields for suitable values of $\vartheta$ an elastodynamic SGE model. Reversing the conventions~\eqref{oo:F:stretch} and invoking the scaling identities~\eqref{PDO:scaling} verified by the stiffness and inertial PDOs (together with $\epsilon^2\Delta_\brx=\ell^2\Delta_{\brX}$), the transient homogenized strain-gradient PDE is obtained as
\begin{equation}
  \sfQth\brU^{\epsilon} + \sfRth\brU^{\epsilon}{}'' = \brFth, \label{SGE:eff}
\end{equation}
with the partial differential operators $\sfQth,\sfRth$ and the loading $\brFth$ defined by
\begin{equation}
\begin{gathered}
  \sfQth := \sfQ_{\rmX}\lsqb \brCz,\ell\brCi,\ell^2\brAii(\vartheta) \rsqb, \qquad
  \sfRth := \sfR_{\rmX}\lsqb \rhoz ,\ell \rhoz \bfbsh,\ell^2 \rhoz \brJii(\vartheta) \rsqb, \\
  \brFth := \brF + \eps^2\rhoz\brcii\sip\brF''
  - \ell^2\Div\lpar \vartheta\,\bfna\brF + \brCz\dip(\brcii\sip\bfna\brF) \rpar.
\end{gathered} \label{ABF:def}
\end{equation}
Equation~\eqref{SGE:eff:freq} has the same form as the strain-gradient PDE~\eqref{SGE:PDE} if the SGE tensors in~\eqref{SGE:PDE} are chosen as
\begin{equation}
  \brC := \brCz, \quad \brMsh := \ell\brCi, \quad \brA = \brA(\vartheta) := \ell^2\brAii(\vartheta), \qquad \rhoSG = \rhoz, \quad \brKsh := \ell \rhoz \bfbsh, \quad \brJ = \brJ(\vartheta) := \ell^2 \rhoz \brJii(\vartheta) \label{corr:tens}
\end{equation}
while the force and double force densities are set to
\begin{equation}
  \brf := \brF + \eps^2\rhoz\brcii\sip\brF'', \qquad
  \bfupph(\vartheta) := \eps^2\lsqb \vartheta\,\bfna\brF + \brCz\dip\lpar \brcii\sip\bfna\brF \rpar \rsqb
\end{equation}
In particular, and as in the previously-examined elastostatic case, all tensor-valued coefficients of~\eqref{SGE:eff:freq} have the symmetry properties expected of those featured in~\eqref{SGE:PDE}. Corresponding identifications apply between the effective PDE~\eqref{SGE:freq} and the time-harmonic version of the strain-gradient PDE~\eqref{SGE:PDE}, to obtain the time-harmonic homogenized strain-gradient PDE
\begin{equation}
  \sfQth\brU^{\epsilon} - \oo^2\sfRth\brU^{\epsilon} = \brFth. \label{SGE:eff:freq}
\end{equation}

As in Section~\ref{sec:compar:zf}, the PDE~\eqref{SGE:eff} can be set in variational form as
\begin{equation}
  \rmQth(\brU,\brW) + \rmRth(\brU'',\brW) = \rmFth(\brW) \qquad
  \text{for all }\brW\in\textbf{C}^{\infty}_0(\Rbb^d), \label{SGE:time:weak}
\end{equation}
with the $\vartheta$-dependent bilinear forms $\rmQth,\rmRth$ and the linear functional $\rmFth(\brW)$ defined by~\eqref{PQ:CGE:bil:def} with the replacements~\eqref{corr:tens}.
The variational problem~\eqref{SGE:time:weak} therefore expresses the stationarity of the functional
\begin{equation}
    \Lcal_{\vartheta}(\brU,\brU')
 := \iT \Lcb \int_{\Rbb^d} k_{\vartheta}(\brU',\bfna\brU') \dx
   - \int_{\Rbb^d} e_{\vartheta}\lpar \bfeps(\brU),\bfeta(\brU) \rpar \dx + \rmFth(\brU) \Rcb \dt,
\end{equation}
with 
the strain and kinetic energy density functions $e_{\vartheta}(\bfeps,\bfeta),\ k_{\vartheta}(\bru,\bfga)$ given for arbitrary vectors $\bru\shin\Rbb^d$ and tensors $\bfeps\shin S^2(\Rbb^d),\ \bfga\shin\Rbb^d\tens\Rbb^d,\ \bfeta\in S^{2}(\Rbb^d)\tens\Rbb^d$ by~\eqref{ke:SGE:def} with replacements~\eqref{corr:tens}, i.e.
\begin{equation}
\begin{aligned}
  2e_{\vartheta}(\bfeps,\bfeta) &:= \bfeps\dip\brCz\dip\bfeps + \eps\lpar \bfeps\dip\brCi\therefore\bfeta 
  + \eps^2\ \bfeta\therefore\brAii(\vartheta)\therefore\bfeta, \\
  2k_{\vartheta}(\bru,\bfga)
 &:= \rhoz \lcb |\bru|^2 + \eps\bru\sip\bfbsh\dip\bfga + \eps^2 \bfga\dip\brJii(\vartheta)\dip\bfga \rcb. 
 \end{aligned} \label{ek:def}
\end{equation}
Comparing the above definition of $e_{\vartheta}$ with its elastostatic counterpart $\eZ_{\vartheta}$ given by~\eqref{eZ:def}, we observe that the positive sign of $\brc^2$ (see Remark~\ref{rem:c2:pos:def}) implies
\begin{equation}
  2e_{\vartheta}(\bfeps,\bfeta) = 2\eZ_{\vartheta}(\bfeps,\bfeta)
  + \eps^2 \ \bfetaB\therefore\brc^2 \therefore\bfeta \geq 2\eZ_{\vartheta}(\bfeps,\bfeta). \label{eZ:e:compar}
\end{equation}

Here again, the main benefit of having introduced the one-parameter families $\brJii(\vartheta),\brAii(\vartheta)$ of tensors given by~\eqref{A2:def} consists in the following essential properties of the energy density functions $e_{\vartheta},k_{\vartheta}$, the matrix-valued symbols $\sfQth(\rmi\bfxi)$, $\sfRth(\rmi\bfxi)$ and the bilinear forms $\rmQth$, $\rmRth$, whose proof is given in Section~\ref{sec:proof:coercivity:fourier}:
\begin{prop}\label{prop:coercivity}
There exist finite positive values $\vartheta_e,\vartheta_k$ of the weight $\vartheta$, which do not depend on $\eps$, such that:
\begin{compactenum}[(a)]
\item The quadratic form $(\bfeps,\bfeta)\mapsto e_{\vartheta}(\bfeps,\bfeta)$ defined by~\eqref{ek:def} is positive definite on $S^{2}(\Rbb^d)\shtimes \lpar S^{2}(\Rbb^d)\tens\Rbb^d\rpar$, and the bilinear form $\rmQth$ of~\eqref{SGE:time} is positive, for any $\vartheta\shg\vartheta_e$.
\item The quadratic form $(\bru,\bfga)\mapsto k_{\vartheta}(\bru,\bfga)$ defined by~\eqref{ek:def} is positive definite on $\Rbb^d\shtimes(\Rbb^d\tens\Rbb^d)$, and the bilinear form $\rmRth$ of~\eqref{SGE:time} is positive, for any $\vartheta\shg\vartheta_k$.
\end{compactenum}
Then, letting $\vartheta_0 := \max(\vartheta_e,\vartheta_k)$, the following properties hold for any $\vartheta>\vartheta_0$:
\begin{compactenum}[(a)]
\setcounter{enumi}{2}
\item Let the PDOs $\sfQth,\sfRth$, given by~\eqref{ABF:def}, be written as $d$-variate (matrix-valued) polynomials $\sfQth[\del],\sfRth[\del]$ evaluated on the "partial derivative vector" $\del=(\del[1],\ldots,\del[d])$. Then, both Fourier symbols $\sfQth(\rmi\bfxi)$, $\sfRth(\rmi\bfxi)\in\Cbb\tens\Cbb$ of $\sfQth,\sfRth$ are positive definite Hermitian matrices for any $\bfxi\shin\Rbb^d,\,\bfxi\shneq\bfze$;
\item The bilinear form $\rmQth$ is $\brH^2(\Rbb^d)$-coercive, i.e. has property (b) of Proposition~\ref{zf:coercivity};
\item The bilinear form $\rmRth$ is $\brH^1(\Rbb^d)$-elliptic, i.e. verifies
\begin{equation}
  \rmRth(\brW,\brW) \geq C(\vartheta) \|\brW\|^2_{\brH^1(\Rbb^d)} \quad \text{for all }\brW\in\brH^1(\Rbb^d) \text{ and some } C=C(\vartheta)>0 \qquad (\vartheta>\vartheta_0)
\end{equation}
\end{compactenum}
\end{prop}


The above sign and coercivity properties (for suitable choices of $\vartheta$) in particular mean that $\brU\mapsto\rmQth(\brU,\brU)$ and $\brU\mapsto\rmRth(\brU,\brU)$ have the sign properties expected of strain and kinetic energy functionals, see~\eqref{ke:SGE:def}. They will shortly be seen to play a key role in the well-posedness of evolution problems involving the transient effective PDE~\eqref{SGE:time}. They also ensure, as discussed in Section~\ref{sec:christoffel}, that the Christoffel equation based on the homogeneous effective PDE~\eqref{SGE:eff:freq} and used for dispersion analysis behaves as expected.
\begin{remark}\label{vartheta:eval}
Importantly from the computational standpoint, the proofs of Propositions~\ref{zf:coercivity} and~\ref{prop:coercivity} also provide a practical method that yields the optimal value of the threshold $\vartheta_0$ ensuring the positivity of both energy densities $e_{\vartheta}$ and $k_{\vartheta}$, based on solving the low-dimensional eigenvalue problems~\eqref{e:eigenvalue:problem} and \eqref{k:eigenvalue:problem} that provide the respective dynamic thresholds $\vartheta_e$ and $\vartheta_k$. This method is used for the results shown in Section~\ref{sec:numerics}. Moreover, as shown in the proof of Proposition~\ref{prop:coercivity}, we have $\vartheta_e \shleq \thZ_e$, with $\thZ_e$ the static stiffness threshold yielded by the eigenvalue problem~\eqref{e:eigenvalue:problem:zf}. When $\vartheta_k \leq \vartheta_e$ (which is the case in the numerical examples of Sec.~\ref{sec:numerics}), one obtains $\vartheta_0 \leq \thZ_e$, \ie the corrective weight $\vartheta$ may be taken smaller for dynamics than for statics.
\end{remark}


\subsection{Well-posedness of transient effective PDE}
\label{sec:transient_equation}

In this section, we establish the well-posedness of evolution problems in $\Rbb^d$ based on the second-order transient effective PDE~\eqref{SGE:eff}, of the form\footnote{for which we adopt the usual function-analytic viewpoint whereby space-time functions (e.g. the displacement $\brU$) are treated as functions of time with values in a Hilbert space $\Xcal$ (e.g. $t\mapsto \brU(t)\in \brH^2(\Rbb^d)$). We recall that $C^m([0,T];\Xcal)$ denotes the space of all $C^m([0,T])$ functions $\brU:[0,T]\to\Xcal$.}
\begin{equation}
  \sfQth\brU + \sfRth\brU'' = \brG \ \ \text{in } [0,T]\times\Rbb^d, \qquad \brU(0)=\brU\isub,\,\brU'(0)=\brV\isub, \label{IVP}
\end{equation}
which governs transient effective motions triggered by given initial values $\brU(0)=\brU\isub,\,\brU'(0)=\brV\isub$ and a given body force density $\brG$ (where for example $\brG=\brFth$ with $\brFth$ given by~\eqref{ABF:def} for the non-homogeneous effective PDE~\eqref{SGE:eff}).

In preparation to establishing the well-posedness of problem~\eqref{IVP} by application of the Hille-Yosida theorem, we set $\brV:=\brU'$, so that problem~\eqref{IVP} becomes a first-order differential system for $\bsfU:=(\brU,\brV)^{\text{T}}$:
\begin{equation}
\left\{ \begin{aligned} \brU' - \brV &= 0 \\ \sfRth\brV' + \sfQth\brU &= \brG \end{aligned} \right. \quad \text{in } [0,T]\times\Rbb^d, \qquad
  \left\{ \begin{aligned} \brU(0) &= \brU\isub \\ \brV(0) &= \brV\isub \end{aligned}. \right.
\end{equation}
To fit the standard function-analytic framework for linear evolution PDEs, the above system is rewritten as
\begin{multline}
  \bsfU'(t) + \Acal \bsfU(t) = \bsfF, \qquad \bsfU(0) = \bsfU_0, \\[-1ex] \text{with} \quad 
  \bsfU := \begin{pmatrix} \brU \\ \brV \end{pmatrix}, \quad
  \bsfU_0 = \begin{pmatrix} \brU\isub \\ \brV\isub \end{pmatrix}, \quad  
  \bsfF = \begin{pmatrix} \bfze \\ \sfRth^{-1}\brG \end{pmatrix}, \quad  
  \Acal\bsfU = 
  \begin{pmatrix} -\brV \\ \sfRth^{-1}\sfQth \brU \end{pmatrix} \label{IVP:order1}
\end{multline}
where, taking advantage of the $\brH^1(\Rbb^d)$-ellipticity of $\rmRth$ (Proposition~\ref{prop:coercivity}), the operator $\sfRth^{-1}:\textbf{L}^2(\Rbb^d)\to \brH^1(\Rbb^d)$ is defined variationally, for any $\textbf{f}\in \mathbf{L}^2(\Rbb^d)$, by
\begin{equation}
  \textbf{z}=\sfRth^{-1}\textbf{f} \ \ \Leftrightarrow \ \
  \rmRth(\textbf{z},\brW) = (\textbf{f},\brW) \quad \text{for all} \ \brW\in \brH^1(\Rbb^d) \label{Binv:def}
\end{equation}
The variational form of problem~\eqref{IVP} suggests to seek $\bsfU(t)$ in the space $\Hcal := \brH^2(\Rbb^d) \times \brH^1(\Rbb^d)$
for each $t$. The operator $\Acal$ being unbounded\footnote{since $\Acal$ essentially acts as a second-order PDO in the spatial coordinates.}, its domain $D(\Acal)$ is chosen here as
$D(\Acal) = \brH^2(\Rbb^d,\sfQth)\times \brH^2(\Rbb^d)$ (with $\brH^2(\Rbb^d,\sfQth):=\lcb \brW\in\brH^2(\Rbb^d),\,\sfQth\brW\in\brH^{-1}(\Rbb^d) \rcb$), which ensures that $\bsfU(t) \in D(\Acal) \ \implies \ \Acal \bsfU(t) \in \Hcal$. The above framework allows to establish the following well-posedness result for  evolution problems in $\Rbb^d$ based on the second-order transient effective PDE~\eqref{SGE:time}, whose proof is deferred to Section~\ref{sec:transient_equation:proof}.

\begin{prop}\label{prop:evol}
Let $\bsfU_0\in D(\Acal)$, and let $\brG$ be such that $\sfRth^{-1}\brG\in C^1([0,T];\brH^1(\Rbb^d))$ or $\sfRth^{-1}\brG\in C^0([0,T];\brH^2(\Rbb^d))$. Let $\vartheta\shg\vartheta_0$, so that the operators $\sfQth,\sfRth$ enjoy properties (a)-(e) of Prop.~\ref{prop:coercivity}. Then:
\begin{compactenum}[(a)]
\item The initial-value problem~\eqref{IVP} has a unique solution $\bsfU\in C^1([0,T];\Hcal) \cap C^0([0,T];D(\Acal))$;
\item If $\brG\sheq\bfze$, the solution $\bsfU$ has the energy conservation property
\[
  E(t) = E(0), \qquad  E(t):= \tdemi\rmQth(\brU(t),\brU(t)) + \tdemi\rmRth(\brU'(t),\brU'(t)).
\]
Otherwise, the solution $\bsfU$ verifies the estimate
\[
  \|\bsfU\|_{\Hcal}^2(t) \leq C(T) \lpar \|\bsfU_0\|_{\Hcal}^2 + \| \bsfF \|^2_{L^2([0,T];\Hcal)} \rpar \quad 
  \text{with}\quad C(T)=\exp(2T), \qquad t\in[0,T].
\]
\end{compactenum}
\end{prop}
\begin{remark}
Proposition~\ref{prop:evol} applies in fact to evolution problems of the form~\eqref{IVP} for any SGE material, for which the PDOs must verify properties (a)-(e) of Prop.~\ref{prop:coercivity}.\enlargethispage*{1ex}
\end{remark}
\begin{remark}\label{rem:RHS:reg}
When $\brG=\brFth$, with $\brFth$ the body force distribution~\eqref{ABF:def} resulting from the homogenization process, the assumption $\sfRth^{-1}\brG\in C^0([0,T];\brH^2(\Rbb^d))$ in Prop.~\ref{prop:evol} is met if the original body force distribution $\brF$ entering $\brFth$, introduced in~\eqref{F:slow}, satisfies $\brF\in C^0([0,T];\brH^2(\Rbb^d))\cap C^1([0,T];\brH^1(\Rbb^d)) \cap C^2([0,T];\brL^2(\Rbb^d))$.
\end{remark}
\begin{remark}
Part (b) of Proposition~\ref{prop:evol} shows that the strong (in time) solution $\bsfU(t)$ has its (weaker) energy norm controlled by the energy norm of the initial datum (which is weaker than its $D(\Acal)$ norm) for the initial-value problem, and by the weaker norm $\|\cdot\|_{L^2([0,T];\Hcal)}$ of the forcing term for the forced-response problem. This may be used to establish solvability results for the evolution problems in variational form under regularity assumptions on the data that are weaker and consistent with spatial discretization by the finite element method.
\end{remark}

\section{Practical aspects: homogenization procedure, special situations, dispersion analysis}
\label{sec:special}

In this section, we survey the results obtained in Sections~\ref{sec:homog:zf} to~\ref{sec:homo:SGE} by providing a practical homogenization procedure yielding well-posed homogenized SG models for statics and dynamics from a given microstructure and discuss some characteristics of these models. We begin in Section~\ref{sec:special:general} with the general case, \ie without assumptions beyond the minimal requirements on elasticity $\bfC$ and mass density $\rho$ stated in Section~\ref{sec:periodic}. Then, several particular situations are considered in Sections~\ref{sec:special:perforated} to~\ref{sec:special:homogeneous:C}, showing how additional assumptions on the constitutive material or the symmetry of the unit cell may simplify the computations and the final model. We also discuss in Section~\ref{sec:fourth:order:in:time} an alternative form of effective PDE that is of fourth order in time. All those aspects largely rely on the alternative reciprocity-based expressions of effective tensors given in Propositions~\ref{prop:recip:zf} and~\ref{prop:recip}. Finally, the generalized Christoffel equation arising from the time-harmonic effective PDE and its asymptotic properties are presented in Sec.~\ref{sec:christoffel}.

\subsection{General case}\label{sec:special:general} Given a periodicity cell $\Ycal$ and material properties $(\bfC,\rho)$ satisfying the requirements of Section \ref{sec:periodic}, we provide the steps to derive homogenized SG models, and then discuss the differences between tensors featured in the static and dynamic models.

\subsubsection{Solving cell problems.} Cell functions $\bfchi^1$ and $\bfchi^2$ are computed by solving the $d(d\shp1)/2$ problems~\eqref{cell_pb 1} and the $d(d\shp1)^2/2$ problems~\eqref{zf:load:second}. For dynamics, the second inertial cell functions $\bfzet^2$ is also computed, by solving the $d$ problems~\eqref{load:second}. All these problems are standard linear elastostatic problems of the form \eqref{cell:generic}, posed on the unit cell $\Ycal$. While few exact solutions exist (except for laminates), they can be solved using standard numerical methods, typically finite elements or FFT-based methods, see \cite{tran:12}. Solving the cell problems may entail significant cumulative computational work, especially for complex 3D cells. By removing the need to solve the $d(d\shp1)^3/2$ problems~\eqref{zf:load:third} (to computute $\bfchi^3$) and the $d^2$ problems~\eqref{load:third} (to compute $\bfzet^2$), the reciprocity properties translate into substantial computational advantage, e.g. by requiring 27 cell solutions instead of 90 for 3D dynamics.

\subsubsection{Computing effective tensors.} Elasticity tensors $\brCz$, $\brCi$ and $\brCii$ are evaluated using~\eqref{C0:comp} and the "reciprocity" expressions~\eqref{zf:reciprocity 1} and~\eqref{zf:reciprocity 2}, respectively, while the inertial tensors $\bfbsh$, $\bfBii$ and $\brcii$ are evaluated using the "reciprocity" expressions~\eqref{bsh:expr}, \eqref{bii:expr} and \eqref{c2:recipr}, respectively.
\subsubsection{Definition of strain-gradient tensors.} The leading and first-order effective tensors $\brCz$, $\brCi$ and, for dynamics, $\rhoz$ and $\bfbsh$, directly provide strain-gradient tensors $\brC,\brMsh,\rhoSG,\brKsh$ up to rescaling by the periodicity length $\ell$ as given by \eqref{zf:corr:tens} and \eqref{corr:tens}: 
\begin{equation}\label{eq:C:rho:M:K:general}
\brC = \brCz,  \quad \brMsh = \ell\brCi, \quad \rhoSG = \rhoz, \quad \brKsh = \ell \rhoz \bfbsh.
\end{equation}
On the other-hand, we showed that the second-order tensors $\brCii$ and $\bfBii$ cannot be used "as they are" to define SGE higher-order tensors with correct sign properties, and instead proposed for the latter the expressions
\begin{equation}
  \brA(\vartheta) = \ell^2 \lsqb \vartheta\brAth - \bfCii + \bfcii \rsqb, \qquad
  \brJ(\vartheta) = \ell^2 \rhoz \lsqb \vartheta\brI^2 - \bfBii + \bfbii \rsqb \label{eq:A:J:general},
\end{equation}
with:
\begin{equation}
  \Ath_{ijnk\ell m} = \CZ_{ijk\ell}\delta_{nm}, \quad 
  \rmc^2_{ijnk\ell m} = \cii_{ba}\CZ_{anij} \CZ_{bmk\ell}, \quad
  \rmI^2_{i\ell km} = \delta_{ik}\delta_{\ell m}, \quad
  \rmb^2_{i\ell km} = \cii_{ia} \CZ_{a\ell km} + \cii_{ka} \CZ_{ami\ell},
\end{equation}
and where (i) terms $\bfcii$ and $\bfbii$, which involve the inertial tensor $\brcii$, arose from eliminating the fourth-order-in-time derivative initially appearing in the homogenization process as discussed in Section \ref{sec:mean}, and (ii) $\vartheta$-dependent terms are added thanks to ensure the positivity of the elastic and inertial energy densities $e_\vartheta$, $k_\vartheta$ associated with the SG model. To this end, the crucial weight parameter $\vartheta$ must satisfy $\vartheta\shg\max(\vartheta_e,\vartheta_k)$. The threshold $\vartheta_e$ is computed by solving the eigenvalue problem \eqref{e:eigenvalue:problem:zf} (for statics) or \eqref{e:eigenvalue:problem} (for dynamics), involving the elasticity tensors $\brCz$, $\brCi$, $\brCii$ and also, for dynamics, the tensor $\brcii$. For dynamics, the second threshold $\vartheta_k$ is computed by solving the eigenvalue problem \eqref{k:eigenvalue:problem}, involving the inertial tensors $\bfbsh$, $\bfBii$, $\brcii$ together with $\brCz$. We emphasize again that the tensors $\brAth$ and $\brI^2$ in~\eqref{eq:A:J:general} are not prescribed by the raw homogenization process and result from our choice of subsequent sign restoration method (or "Boussinesq trick"), other choices being possible as discussed in Remark~\ref{rem:correction}.

\subsubsection{Differences between statics and dynamics.}\label{sec:statics:dynamics}

As is well-known in this classical homogenization framework, the effective elasticity tensor $\brCz$ depends on the elastic properties only and is therefore the same in dynamics. A less-known result, to our knowledge, is that this is also found here to be true for the coupling tensor $\brMsh$, which is the same for statics and dynamics.

On the other hand, the dynamic SG elasticity tensor $\brA(\vartheta)$ differs from its static counterpart due to the role of the inertial tensor $\brcii$ which (i) enters the defining formula~\eqref{eq:A:J:general} and (ii) influences the value of the threshold $\vartheta_e$ by being featured in the eigenvalue problem~\eqref{e:eigenvalue:problem}, see Remark~\ref{vartheta:eval}. Except for configurations where $\brcii$ vanishes (such as cells with homogeneous mass density, see Sec.~\ref{sec:special:homogeneous:rho}), this seemingly inconsistent feature comes from our choice to avoid the fourth-order time derivative in the effective PDE (which also modifies $\brJ(\vartheta)$) and leads to a PDE consistent with the low-frequency limit, see Sec.~\ref{sec:fourth:order:in:time}.
Using results obtained from statics (typically a SG tensor determined by elastostatic homogenization or other procedures) to model dynamics thus warrants caution: if the mass density if heterogeneous, the transient effective PDE has to feature either a fourth-order time derivative or a modified SG elasticity tensor.

\subsection{Periodically-perforated media} \label{sec:special:perforated}

Studies on the leading-order homogenization of periodically perforated media $\OO$ with traction-free holes (introduced in Section~\ref{sec:periodic}) show that the homogenization methodology and results for full media ($\OO=\Rbb^d$) carry over, with all cell averages still defined as integrals over $\Ycal$, see e.g.~\cite{cioranescu:reticulated:99}. With the defining variational formulations for all cell solutions expressed as integrals over $\Ycal\ssub$ instead of $\Ycal$, all the definitions and results of this work for cell problems, mean displacements, effective tensors... apply as given for full as well as perforated media.

\subsection{Centrosymmetric periodicity cell}
\label{sec:special:centrosym}


Consider now cases where the periodicity cell (by assumption centered at the origin) is centrosymmetric, i.e. verifies $\Ycal=-\Ycal$, with the material parameters satisfying the symmetry assumptions $\bfC(-\bfy) = \bfC(\bfy)$ and $\rho(-\bfy) = \rho(\bfy)$ for any $\bfy\shin\Ycal$. Then the cell solution $\bfchi^1$ is odd:
\begin{lemma}\label{centrosym}
Let the cell $\Ycal$ be centrosymmetric. The cell functions $\bfchi^{1k\ell}$ verify $\bfchi^{1k\ell}(-\bfy)=-\bfchi^{1k\ell}(\bfy)$ (and therefore $\bfna\bfchi^{1k\ell}(-\bfy)=\bfna\bfchi^{1k\ell}(\bfy)$), for any $\bfy\shin\Ycal$ and any $1\shleq k,\ell\shleq d$. 
\end{lemma}
\begin{proof}
Let $\bfchih^{1k\ell}(\bfy):=\bfchi^{1k\ell}(-\bfy)$. On making the change of variable $\bfy\mapsto-\bfy$ in the governing variational problem~\eqref{cell:generic},~\eqref{weak form chi1 index} form $\bfchi^{1k\ell}$ and setting $\bfvH(\bfy)=\bfv(-\bfy)$ (so that $\bfna\bfvH(\bfy)=-\bfna\bfv(-\bfy)$) for any test function $\bfv\shin\bHzm(\Ycal)$, we obtain
\begin{equation}
  A\lpar \bfchih^{1k\ell},\bfvH\rpar = -F^{k\ell}(\bfvH) \qquad \forall\bfvH\shin\bHzm(\Ycal)
\end{equation}
thanks to the constitutive symmetry assumption. Comparing the above variational problem to the original problem~\eqref{cell:generic},~\eqref{weak form chi1 index}, we see that $\bfchih^{1k\ell}=-\bfchi^{1k\ell}$ in $\Ycal$. 
\end{proof}

This property has important consequences on the SG models, already known from symmetry-based arguments \cite{auffray:15,auffray:19} and recovered here by the homogenization procedure. First, considering the reciprocity-based value~\eqref{zf:reciprocity 1} of $\brCi$ and making the change of variable $\bfy\mapsto-\bfy$ therein, we obtain $\brCi=\bfze$, so that $\brMsh=\bfze$. Similarly, considering the expression~\eqref{bsh:expr} leads to $\bfbsh = \bfze$ so that $\brKsh = \bfze$. This cancels all first-order terms in the SG models. With reference to Remark~\ref{ke:SGE:rem}, the positivity of the energy densities $\rme$ and $\rmk$ thus hinges on the less-stringent requirement that $\brA(\vartheta)$ and $\brJ(\vartheta)$ be positive definite, leading to lower values of the thresholds $\vartheta_e,\vartheta_k$.

The cancellation of these terms cause many couplings to disappear. In statics, extension–bending couplings are eliminated \cite{papanicolopulos2011chirality,poncelet:18}. It should be noted that these couplings are generally studied within the framework of Cosserat continuum mechanics and under the somewhat improper term “chiral” (which is sufficient but not necessary)\cite{lakes2001elastic}. In dynamics, these couplings are responsible for the circular polarization of elastic waves \cite{toupin:62,Por68,rosi:24}.

\subsection{Homogeneous mass density} \label{sec:special:homogeneous:rho}

We now address the case where $\rho$ is homogeneous in the whole cell $\Ycal$ or in $\Ycal\ssub$ for periodically-perforated media, irrespective of possibly-heterogeneous elasticity properties in $\Ycal$ or $\Ycal\ssub$. This case notably includes the common situation of periodically-perforated media made of a unique homogeneous material. Then, letting $\phi = |\Ycal_\mathrm{s}|/|\Ycal|$ denote the volume fraction of the solid phase, the effective mass density is $\rhoz = \phi \rho$, and the relative mass density $\beta$ reduces to the constant $\beta = \phi^{-1}$. Consequently, all cell solutions defined in this work being chosen to have zero mean, the three inertial tensors considerably simplify, as their expressions given by Proposition \ref{prop:recip} become
\begin{equation}
    \bfbsh = \bfze, \qquad  \Bii_{i\ell km}  = -\beta \lbra \chi^{1i\ell}_r \chi^{1km}_r \rbra_{\Ycal}, \qquad \brcii = \bfze. 
\end{equation}
and the inertial solution $\bfzet^{2}$ is no longer needed. In particular, $\bfbsh = \bfze$ implies that the inertial coupling tensor $\brKsh$ vanishes, and $\brcii = \bfze$ implies that the SG tensors are given by the simpler expressions
\begin{equation}
  \brA(\vartheta) = \ell^2 \lsqb \vartheta\brAth - \bfCii \rsqb, \qquad
  \brJ(\vartheta) = \ell^2 \rhoz  \lsqb  \vartheta\brI^2 - \bfBii \rsqb \label{eq:A:J:homog:rho},
\end{equation}
An important consequence is that the threshold $\vartheta_e$ is the same for statics and dynamics, so that the SGE tensor $\brA(\vartheta)$ is the same for statics and dynamics if $\vartheta_e > \vartheta_k$ (which is observed in practice in the numerical examples of Section \ref{sec:numerics}).
\begin{remark}
Since the literature has dealt essentially with centrosymmetric cells the vanishing of $\brKsh$ for homogeneous architectured media has not, to the best of the authors' knowledge, been noticed so far. It would be interesting to investigate in more depth the effect of this contribution in the case of, for instance, three-dimensional heterogeneous gyroid structures.
\end{remark}

\subsection{Homogeneous elasticity: inertial SG model} \label{sec:special:homogeneous:C} Consider now the less-usual case of a full medium ($\OO=\Rbb^d$) whose elastic tensor $\bfC$ is homogeneous but not the mass density $\rho$. A concrete setting might consist of a thin medium, modeled using 2D elasticity, featuring spatially-periodic added mass that leaves its elastic properties unperturbed. 

For statics, the homogenization procedure naturally stops at the leading-order with $\brCz=\bfC$, and in particular produces vanishing cell solutions: $\bfchi^1 = \bfchi^2 = \bfze$. For dynamics, however, the mass density heterogeneity leads to a special, simplified "inertial SG model". Indeed, the inertial tensors $\bfbsh$ and $\bfBii$, depending on $(\bfchi^1, \bfchi^2)$ as discussed in Section \ref{sec:compar}, also vanish. The only cell solution to compute is the inertial solution $\bfzet^2$, and the second-order perturbation of leading-order classical elasticity is then given by the sole tensor $\brcii$. The obtained "inertial SG model" features null coupling tensors ($\brMsh=\bfze$, $\brKsh=\bfze$), and the second-order tensors $(\brA,\brJ)$ are defined in terms of $(\brCz = \bfC,\rhoz,\brcii)$ only:
\begin{equation}
  \brA(\vartheta) = \ell^2 \lsqb \vartheta\brAth + \bfcii \rsqb, \qquad
  \brJ(\vartheta) = \ell^2 \rhoz \lsqb  \vartheta\brI^2 + \bfbii \rsqb \label{eq:A:J:homog:C},
\end{equation}
Recall from remark \ref{rem:c2:pos:def} that, in this case, $\brA(0) = \ell^2\bfcii$ is positive while $\brJ(0) = \ell^2\rhoz\bfbii$ is either (i) positive if $\brcii$ is isotropic or (ii) sign-indefinite. In case (i), any $\vartheta > 0$ ensures the well-posedness of the model while in case (ii) $\vartheta$ should satisfy $\vartheta > \vartheta_k$.
Moreover, since the second-order tensor $\brcii$ is the only term depending on the microstructure in the model, the tensor $\brA$ given above is only \emph{cosmetically} of order 6, but behaves, with respect to the group of isometric transformations, as a second order tensor. In fact, it is not able to distinguish as many anisotropy classes as a genuinely  sixth-order tensor. As an example, a microstructure with cubic symmetry will be effectively seen as isotropic.

\subsection{Fourth-order-in-time effective wave equation}
\label{sec:fourth:order:in:time}

A similar, somewhat simpler, version of the procedure followed in Section~\ref{sec:mean} allows to obtain a fourth-order-in-time effective PDE with tensor coefficients possessing the same symmetry and sign-definiteness properties as those of its second-order-in-time version. In fact, effecting the combination~(\ref{mean:eqs:init}a)+$\epsilon$(\ref{mean:eqs:init}b)+$\epsilon^2$(\ref{mean:eqs:init}c)-$\epsilon^2\vartheta$\eqref{aux23} of macroscopic balance equations, where the adjustable weight $\vartheta$ serves as before to control sign properties of tensor coefficients, yields the following PDE satisfied, up to a $\Ocal(\epsilon^3)$ perturbation, by  the $\Ocal(\epsilon^2)$ approximation~\eqref{Ueps:def} of the macroscopic displacement:
\begin{equation}
  \sfQZ_{\vartheta}\brU - \oo^2 \rhoz \sfR\lsqb 1, \ell\bfbsh, \ell^2(\vartheta\brI^2\shm\bfBii) \rsqb\brU - \oo^4\ell^2(\rhoz)^2 \brcii\sip\brU
 = (1-\ell^2\vartheta\Delta)\brF \label{SGE:freq:order4}
\end{equation}
up to a $\Ocal(\epsilon^3)$ residual, where $\sfQZ_{\vartheta}$ is the fourth-order elastostatic PDO given by~\eqref{zf:SGE:eff} and the second term uses the generic PDO notation~\eqref{PDO:generic}.
Then, applying the inverse Fourier transform to~\eqref{SGE:freq:order4} yields the transient effective equation
\begin{equation}
  \sfQZ_{\vartheta}\brU + \rhoz \sfR\lsqb 1, \ell\bfbsh, \ell^2(\vartheta\brI^2\shm\bfBii) \rsqb\brU'' - \ell^2(\rhoz)^2 \brcii\sip\brU'''' = (1-\ell^2\vartheta\Delta)\brF, \label{SGE:time:order4}
\end{equation}
whose right-hand side features second-order space derivatives of the original body force density $\brF$. The fourth-order in time effective PDEs~\eqref{SGE:freq:order4} and~\eqref{SGE:time:order4} clearly reduce to the elastostatic equation~\eqref{zf:SGE:eff} in the zero-frequency, i.e. time-independent, cases.


It is unclear whether evolution problems based on the effective PDE~\eqref{SGE:time:order4} enjoy well-posedness guarantees of the kind given in Proposition~\ref{prop:evol} for the second-order-in-time PDE. We did not succeed in transposing to~\eqref{SGE:time:order4} the proof method of the latter, and other authors \cite{Allaire2022,allaire:24} even derive a systematic method to guarantee a second-order-in-time effective equation. Available theory on fourth-order-in-time PDEs, which is scarce relatively to that covering PDEs of first- and second-order in time, does not address PDEs of the form~\eqref{SGE:time:order4}. However, \cite[Thm.~10.2]{goldstein} hints at potential difficulties, as it asserts that fourth-order-in-time problems of the form $\sfQ\brU + \brU''''=\brF$ are well-posed if and only if $\sfQ:\Hcal\to\Hcal$ is bounded (where the space $\Hcal$ defines the target regularity of $\brU$ in space), which rules out $\sfQ$ being a PDO in the spatial coordinates in that case. Further investigation is needed to settle (positively or negatively) the well-posedness of evolution problems based on~\eqref{SGE:time:order4}. Similarly, dispersion analysis based on injecting trial plane waves in~\eqref{SGE:freq:order4} leads to a quadratic eigenvalue problem~\cite{tisseur:qep} instead of the usual (generalized Christoffel) eigenvalue problem obtained using the second-order PDE, see Section~\ref{sec:christoffel}.

In physical terms, the tensor $(\rhoz)^2 \brcii$ in~\eqref{SGE:time:order4} can be interpreted as the tensor of a quadratic form acting on accelerations in a Lagrangian or Hamiltonian associated with the PDE. It can be viewed as a term penalizing variations of acceleration. This leads to a field theory whose Lagrangian is of second order in time, and therefore to Euler–Lagrange equations of fourth order in time. In the general case, such a model, whose archetype is the Pais–Uhlenbeck oscillator \cite{pavvsivc2016pais}, leads to unstable behavior, known as Ostrogradsky instability~\cite{ganz2021reconsidering}. In certain situations, however, notably when the tensor $\brcii$ is degenerate, this instability can be avoided.


We nevertheless think that this fourth-order-in-time alternative to the SG model, which we might call the "time-SG" model, is worthy of attention. We first note that if the model is taken "as is", the second-order tensor $\brcii$ appears as the sole factor of the fourth-order time derivative, which somehow makes its role clearer than when it contributes to the dynamic SG tensors $\brA,\brJ$. In addition, one can build a \emph{family} of time-SG models using similar manipulations as in Section \ref{sec:mean:zf} (involving the second-order time derivative of the leading-order macroscopic balance equation rather than its Laplacian). Members of this family could be interpreted as \emph{stress-gradient} models that feature fourth-order time derivatives \cite{forest:17}.

Then, in the (much) simpler 1D case, models that include a fourth-order time derivative have been discussed for a long time \cite{pichugin:08}. Keeping this term  was shown to (i) lead to a reformulation of the effective equation into a well-posed hyperbolic system inspired by a stress-gradient model \cite{forest:17,corn:lomb:23} and (ii) improve the accuracy of the model when predicting dispersion relations \cite{pichugin:08,corn:guz:20,corn:lomb:23}. We do not expect a similar result to readily hold for arbitrary materials and propagation direction in elasticity, because the low-order tensor $\brcii$ cannot describe complex anisotropy as the elasticity tensors $\bfCz$ and $\brA$ do, but take this 1D example as a hint that the influence of this time derivative on the dispersive properties of the model should be explored.

Finally, in the particular case of homogeneous elasticity (but heterogeneous mass density) discussed in Section \ref{sec:special:homogeneous:C}, the fourth-order time derivative is the only higher-order term in the effective model. Indeed the time-SG model \eqref{SGE:time:order4} reduces to:
\begin{equation}
  -\Div(\bfCz:\bfeps[\brU]) + \rhoz \brU'' + \ell^2(\rhoz)^2 \brcii\sip\brU'''' = \brF, \label{SGE:time:order4:homog:rho}
\end{equation}
\ie a single-term correction of the classical, leading-order, elastodynamic wave equation. In this later case, the time-SG model is therefore second-order in space, and arguably simpler to discuss and implement than the "inertial SG model". If it is proven to be well-posed and suitable for implementation,  classical $H^1$-conforming finite elements may suffice for numerical simulations, and there is no additional boundary conditions to derive compared to classical elasticity (but additional initial conditions to impose). This simplified case could therefore be used a first step into the study of time-SG models. 


\subsection{Plane wave solutions and generalized Christoffel equation} 
\label{sec:christoffel}


Consider a time-harmonic plane wave of the form
\begin{equation}
  \brU(\brX,t) = \brV\exp\lpar \rmi(\upkappa\bfn\sip\brX-\oo t) \rpar \label{planewave}
\end{equation}
characterized by its real wavenumber $\upkappa$, real unit propagation direction $\bfn$ and complex polarization $\brV\shin\Cbb^d$. To satisfy the homogeneous effective PDE~\eqref{SGE:eff:freq}, the above plane wave must verify the generalized Christoffel equation~\cite{rosi:24}
\begin{equation}
  \sfQth\gsup(\upkappa,\bfn)\sip\brV - \mu\,\bfrhv\esup(\upkappa,\bfn)\sip\brV = \brze, \qquad \mu=\rhoSG v^2 = \rhoSG(\oo/\upkappa)^2 \label{planewave:eig}
\end{equation}
where $v\sheq\oo/\upkappa$ is the phase velocity and, setting $\delta:=\kappa \eps$, the generalized acoustic tensor $\sfQth\gsup(\upkappa,\bfn)$ and directional apparent mass tensor $\bfrhv\esup(\upkappa,\bfn)$ are given by
\begin{equation}
  \sfQth\gsup(\upkappa,\bfn)
 = \sfQ^0(\bfn) + \rmi\delta\sfQ^1(\bfn) + \delta^2\sfQ^2_{\vartheta}(\bfn), \qquad
  \bfrhv\esup(\upkappa,\bfn)
 = \sfR^0(\bfn) + \rmi\delta\sfR^1(\bfn) + \delta^2\sfR^2_{\vartheta}(\bfn) \label{planewave:expand}
\end{equation}
with the various matrices given (in component form and suppressing for brevity the dependence on $\bfn$) by
\begin{equation}
\begin{aligned}
  \sfQ^0_{ik} &= \CZ_{ijk\ell} n_j n_{\ell}, & \quad
  \sfQ^1_{ik} &= \Ci_{ijk\ell m} n_j n_{\ell} n_m, & \quad
  \sfQ^2_{\vartheta ik} &= \vartheta\sfQ^0_{ik} + \sfQ^0_{ia}\cii_{ab}\sfQ^0_{bk} - \sfC^2_{ik}, \\
  \sfR^0_{ik} &= \delta_{ik}, &
  \sfR^1_{ik} &= \bsh_{ik\ell} n_{\ell}, &
  \sfR^2_{\vartheta ik} &= \vartheta\sfR^0_{ik} + \sfQ^0_{ia}\cii_{ak} + \cii_{ib}\sfQ^0_{bk} - \sfB^2_{ik},
\end{aligned} \label{planewave:matrices}
\end{equation}
with $\sfC^2_{ik}\shdeq\Cii_{ijnk\ell m} n_j n_{\ell} n_m n_n$ and $\sfB^2_{ik}\shdeq\Bii_{i\ell km} n_{\ell} n_m$, noting that $\sfQ^0,\,\sfQ^2_{\vartheta},\,\sfR^0,\,\sfR^2_{\vartheta}$ are real symmetric and $\rmi\sfQ^1$, $\rmi\sfR^1$ are Hermitian ($\sfQ^1,\,\sfR^1$ being real skew-symmetric). 
For example, given $\bfn$ and $\upkappa$, a plane wave~\eqref{planewave} can only propagate at angular frequencies $\oo\sheq\oo(\bfn,\upkappa)$ such that $v^2\sheq\oo^2/\upkappa^2$ is an eigenvalue of equation~\eqref{planewave:eig} and with polarizations $\brV$ that are corresponding eigenvectors. Proposition~\ref{prop:coercivity}(a) with $\bfxi\sheq\upkappa\bfn$ implies that $v^2>0$ for any $\vartheta\shg\vartheta_0$.
Using matrices~\eqref{planewave:expand} with $\delta\sheq0$ (i.e. $\eps\sheq0$) therein, \eqref{planewave:eig} is the standard Christoffel equation for the periodic medium homogenized at leading order only.

The following proposition gives the main properties of the $\Ocal(\delta^2)$ expansion of phase velocities $v\sheq v(\delta)$, through that of eigenvalues $\mu(\delta)\sheq\rhoSG v^2(\delta)$ solving the generalized Christoffel equation~\eqref{planewave:eig} with~\eqref{planewave:expand}, for the homogenized SG model:

\begin{prop}\label{prop:disp:expansion}
Let $\mu_0\shg0$ be an eigenvalue of the standard Christoffel eigenvalue problem $\sfQ^0\brV - \mu\sfR^0\sip\brV = \brze$.
\begin{compactenum}[(a)]
\item If $\mu_0\shg0$ has unit multiplicity, we have $\mu(\delta)=\mu_0+\delta^2\mu_2+o(\delta^2)$ with $\mu_2\shin\Rbb$ for any $\Ycal$ (possibly non-centrosymmetric).
\item If $\mu_0\shg0$ has multiplicity 2 and $\Ycal$ is centrosymmetric, we still have $\mu(\delta)=\mu_0\shp\delta^2\mu_2\shp o(\delta^2)$ with $\mu_2\shin\Rbb$.
\item If $\mu_0\shg0$ has multiplicity 2 and $\Ycal$ is not centrosymmetric, we have $\mu(\delta)=\mu_0+\delta\mu_1+\delta^2\mu_2+o(\delta^2)$ with $\mu_1,\mu_2\shin\Rbb$.
\end{compactenum}
\end{prop}
\begin{remark}\label{mu1=0}
Proposition~\ref{prop:disp:expansion} indicates that the $\Ocal(\delta)$ term of $v(\delta)$ vanishes not only for centrosymmetric cells (for which the effective PDE contains no $\Ocal(\delta)$ term) but also for all non-centrosymmetric cells except those producing a standard Christoffel equation having a double eigenvalue (whereas the polarization $\brV$ has a $\Ocal(\delta)$ term for any non-centrosymmetric cell, see e.g.~(\ref{planewave:expand:series}b)). Determining whether the remainder $o(\delta^2)$ of the $\Ocal(\delta^2)$ expansion of $v(\delta)$ is $\Ocal(\delta^3)$ or $\Ocal(\delta^4)$ requires pushing the homogenization process to one additional order in $\epsilon$.
\end{remark}

\begin{remark}
Setting $\vartheta\shg\vartheta_0$ with $\vartheta_0$ taken as the suitable threshold for the effective PDE (see Remark~\ref{vartheta:eval}) also ensures that the Christoffel eigenvalue equation~\eqref{planewave:eig} with~\eqref{planewave:expand} has real positive eigenvalues for any $\delta\shgeq0$ (see Prop.~\ref{prop:coercivity}(b)), whereas~\eqref{planewave:eig} with $\vartheta=0$ is guaranteed to enjoy this property only if $\delta$ is small enough. While the threshold $\vartheta_0$ was found to not depend on $\ell$, a value $\vartheta(\delta)\shleq\vartheta_0$ adjusted according to $\delta$ might be defined instead for the Christoffel equation~\eqref{planewave:eig} (in particular $\vartheta(\delta)=0$ for small enough $\delta$). This difference reflects the fact that $\ell\not=0$ induces singular perturbations on the effective PDE (since higher-order derivatives then appear) but only regular perturbations on the Christoffel equation.
\end{remark}
\begin{remark}[Generalized Christoffel equation from the fourth-order-in-time effective equation] Small-$\delta$ expansions can also be conducted on the generalized Christoffel equation arising from plugging plane waves into the homogeneous fourth-order effective equation~\eqref{SGE:freq:order4}, producing the quadratic eigenvalue problem
\begin{equation}
  \lsqb \sfQ^0 + \rmi\delta\sfQ^1 + \delta^2(\vartheta\sfQ^0-\brC^2) \rsqb\sip\brV
  - \mu \lsqb \sfR^0 + \rmi\delta\sfR^1 + \delta^2(\vartheta\sfR^0-\brB^2) \rsqb \sip \brV - \mu^2\lsqb \delta^2\brcii \rsqb \sip \brV = \brze.\label{dispersion:expand:4th}
\end{equation}
The arising $\Ocal(1)$ and $\Ocal(\delta)$ equalities are the same, and thus so are their outcomes. Then, for $\mu_0$ a single eigenvalue, left-contracting the $\Ocal(\delta)$ equality by $\brV_0$ and using previous-order results as appropriate, we have
\begin{equation}
  \mu_2 =
  \rmi \brV_0\sip\lpar \sfQ^1 - \mu_0\sfR^1 \rpar\sip\brV_1 - \brV_0\sip\lpar \brC^2 - \mu_0\brB^2 \rpar\sip\brV_0 - (\mu_0)^2\brV_0\sip\brcii\sip\brV_0,
\end{equation}
so that the eigenvalue correction $\mu_2$ is also the same as before. We leave the case where $\mu_0$ is a double eigenvalue, for which the same corrections $\mu_{2\pm}$ as before are again obtained, to the reader.
\end{remark}

\section{Numerical illustrations}
\label{sec:numerics}

In this section, the proposed models are implemented and their outcomes are compared to (i) reference Floquet-Bloch
dispersion curves and (ii) wave simulation in architectured materials. All examples assume 2D plane-strain linearly elastic conditions. The numerical set-up is presented first, and then three microstructures are explored. For convenience, the following acronyms will be used in the forthcoming discussion:
\begin{itemize}
    \item LOE for the classical leading-order model characterized by effective elasticity $\brCz$ and mass density $\rhoz$;
    \item \SGvt for the well-posed SG model obtained using~\eqref{eq:A:J:general} with an appropriate choice of $\vartheta$;
    \item SG(0) for the (potentially ill-posed) SG model obtained using~\eqref{eq:A:J:general} with $\vartheta=0$,    
    that we will use to illustrate the importance of ensuring well-posedness;
    \item FB for Floquet-Bloch reference dispersion computations.
\end{itemize}

\subsection{Numerical set-up}

We chose three periodicity cell configurations, depicted in Figure \ref{fig:cells}:
\begin{itemize}
 \item a square lattice ($\rmD_4$ symmetry);
 \item a hexagonal lattice (honeycomb, $\rmD_6$ symmetry);
 \item a more intricate lattice based on a hexagonal cell, to obtain a  non-centrosymmetric trichriral microstructure with $\mathrm{Z}_3$ symmetry.
\end{itemize}
These configurations will be referred to as $\rmD_4$, $\rmD_6$ and $\mathrm{Z}_3$ in the sequel.
All three periodic media are perforated (see Fig.~\ref{fig:cells} and Sec.~\ref{sec:special:perforated}), with their material fraction made, to simplify computations, of one or two isotropic materials whose material properties are given in Table \ref{tab:materials}. We refer to \cite{auffray:15} for a comprehensive description of the symmetries of 2D cells and their associated effective tensors.

\begin{figure}[t]
\centering
\includegraphics[width=0.25\textwidth]{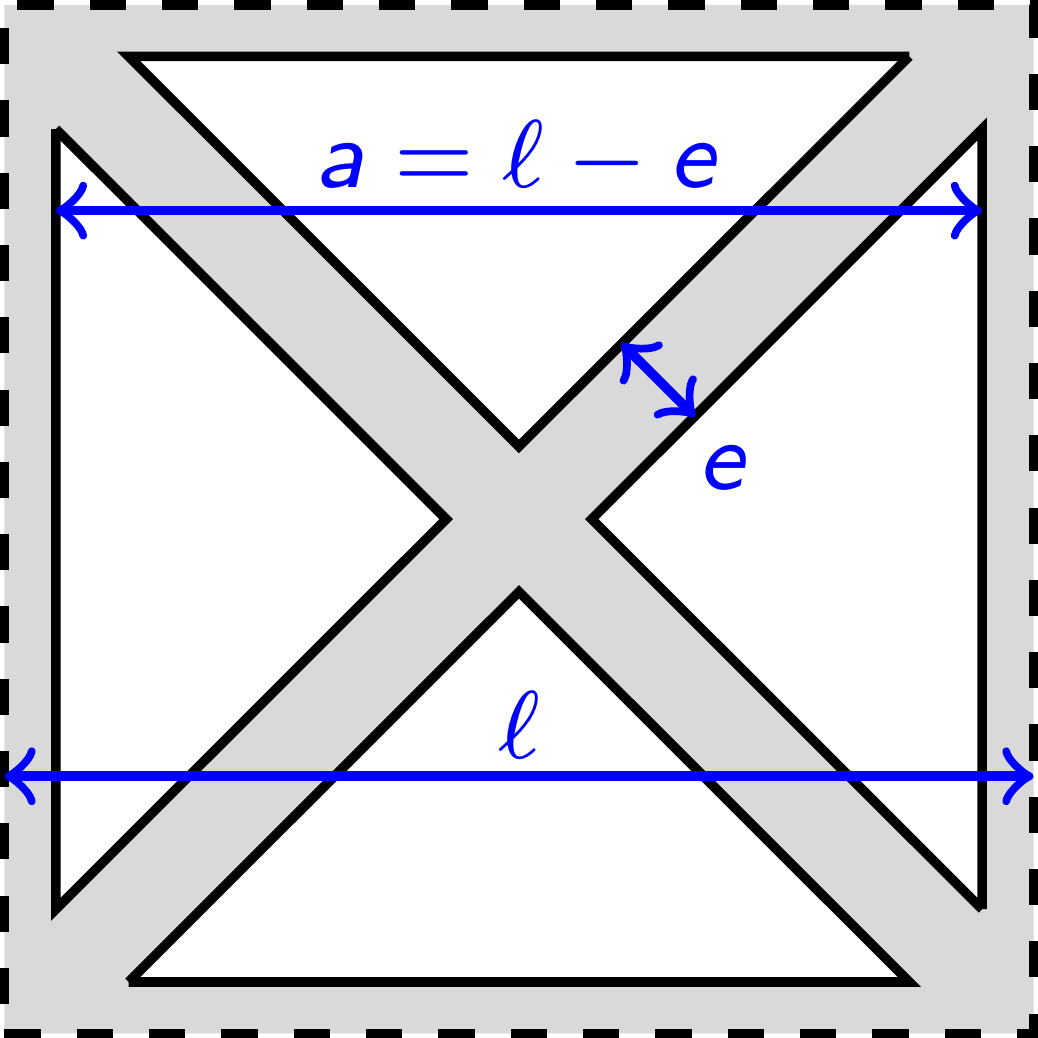}\hspace{0.05\textwidth}
\includegraphics[width=0.25\textwidth]{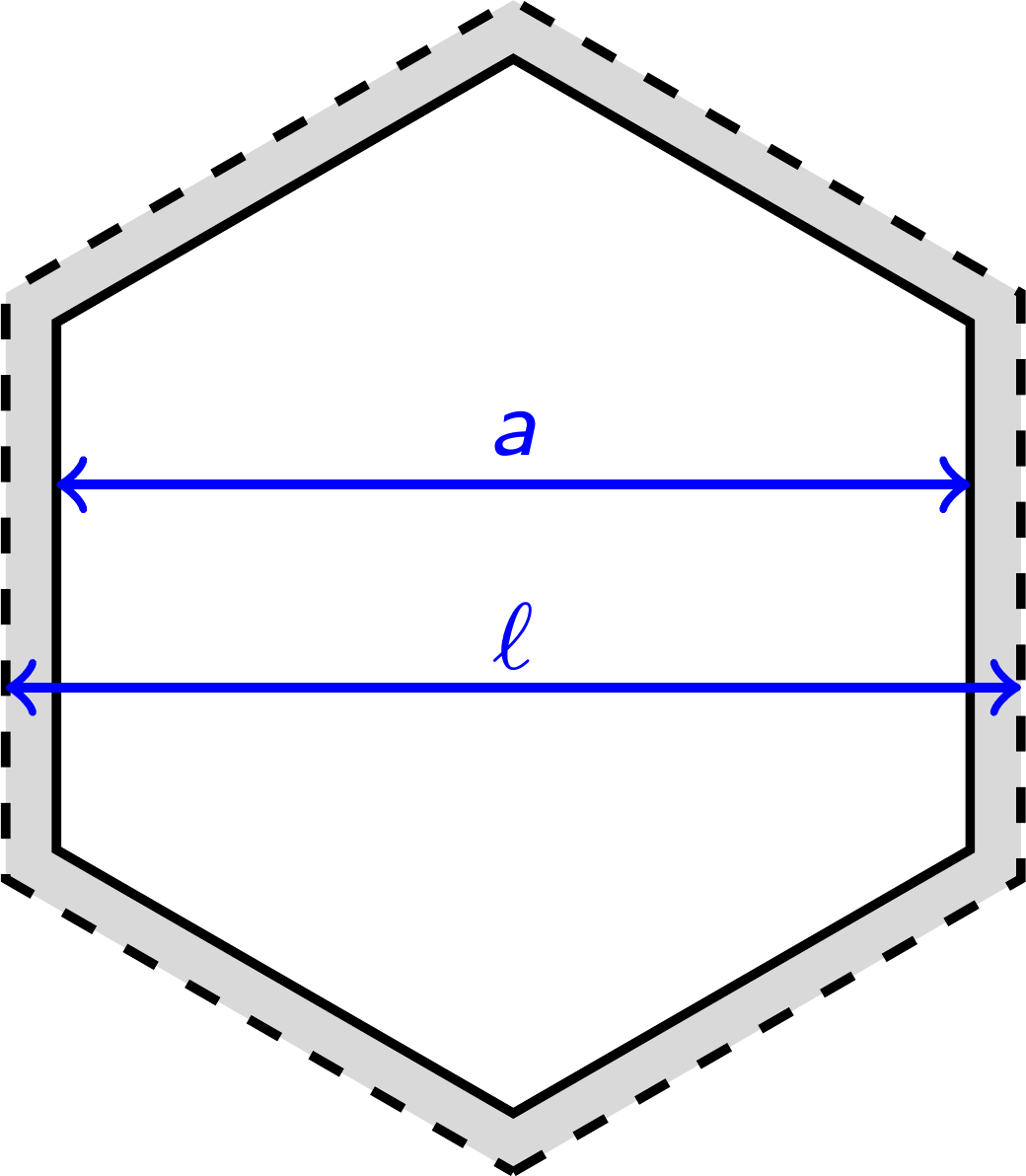}\hspace{0.05\textwidth}
\includegraphics[width=0.25\textwidth]{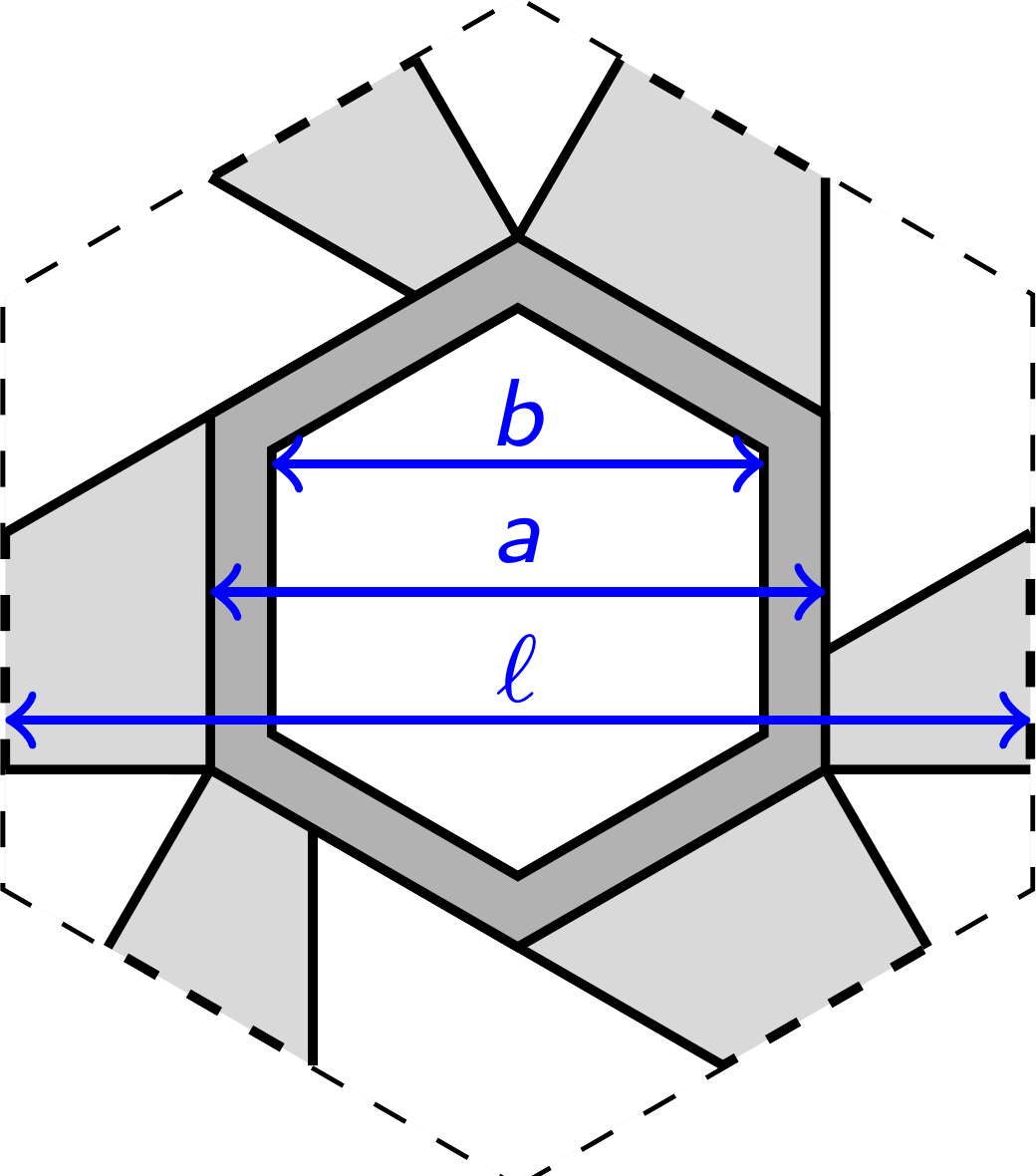}
\caption{$\rmD_4$, $\rmD_{6}$ and $\mathrm{Z}_{3}$ unit cells. The dimensions are given in Table \ref{tab:cells}. }
\label{fig:cells}
\end{figure}

\begin{table}[t]
\begin{tabular}{c|cccccccccc}
 Cell Type & $a/\ell$ & $e/\ell$ & $b/\ell$  & Vol. fraction $\phi$ & $\thZ_e$ & $\vartheta_e$ & $\vartheta_k$ & Chosen $\vartheta$ & DOFs & CPU time \\ \hline
 $\rmD_4$ & 0.9 & 0.1 & $-$  & 0.49 & 0.152 & 0.152 & 0 & 0.16 & $8.96\times10^6$ & 9 h\\
 $\rmD_6$ & 0.9 & $-$ & $-$ & 0.19 & 0.325 & 0.325 & 0 & 0.35  & $6.22\times10^6$ & 7 h\\ 
 $\mathrm{Z}_{3}$ & 0.6 & $-$ & 0.48 & 0.45 & 0.0476 & 0.0472  & 0.003 & 0.05 &  $1.79\times10^7$ & 45 h\\
\end{tabular}
\caption{Geometrical parameters, stabilization parameters, and computational characteristics of the benchmark microstructure simulations (see Figure~\ref{fig:cells} for the geometrical notations). By comparison, the SGE simulations use $2.27\times10^5$ degrees of freedom and require approximately $5$~min.}
\label{tab:cells}
\end{table}

\begin{figure}[t]
\centering
\includegraphics[width=0.25\textwidth]{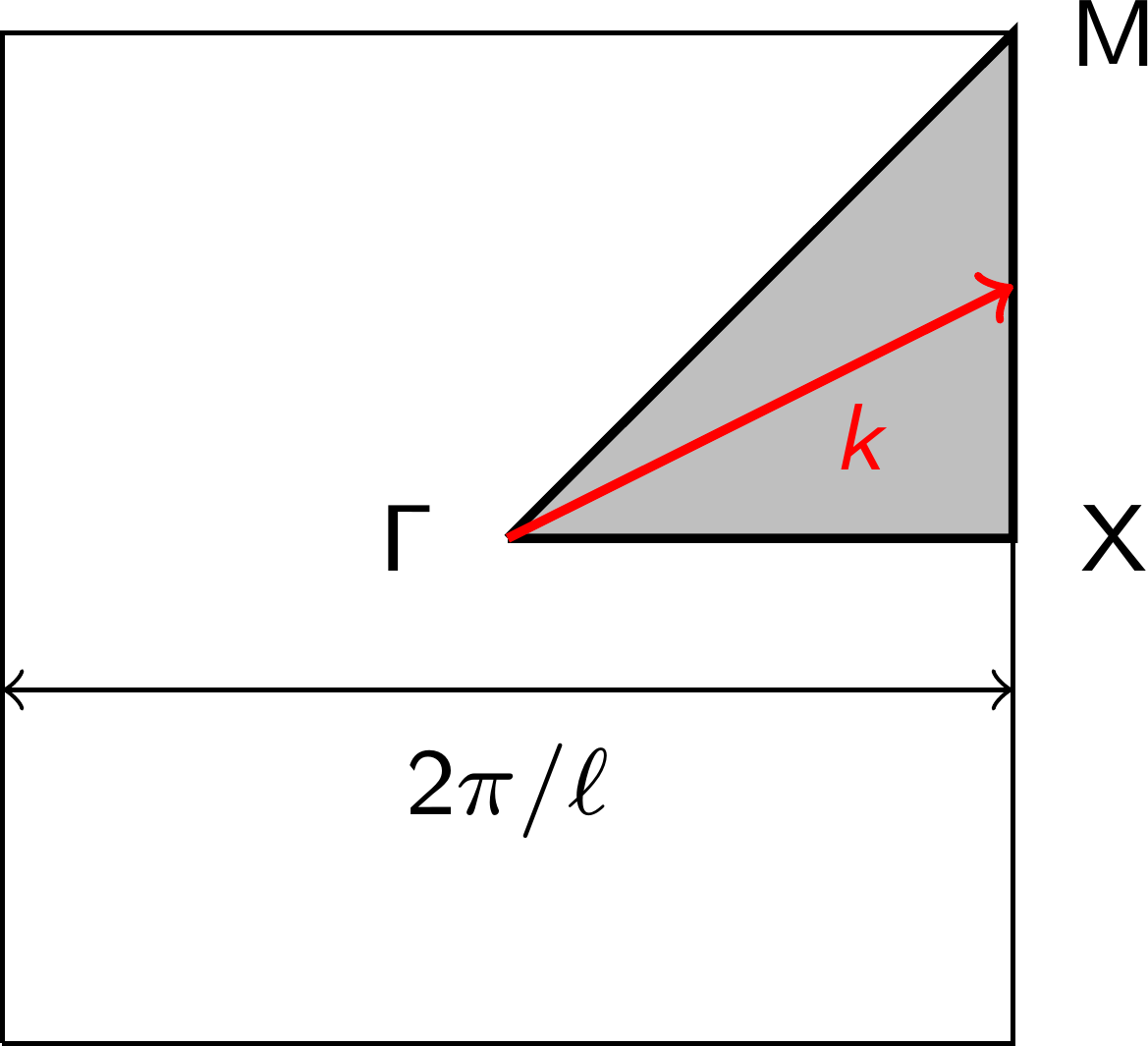}\hspace{0.05\textwidth}
\includegraphics[width=0.25\textwidth]{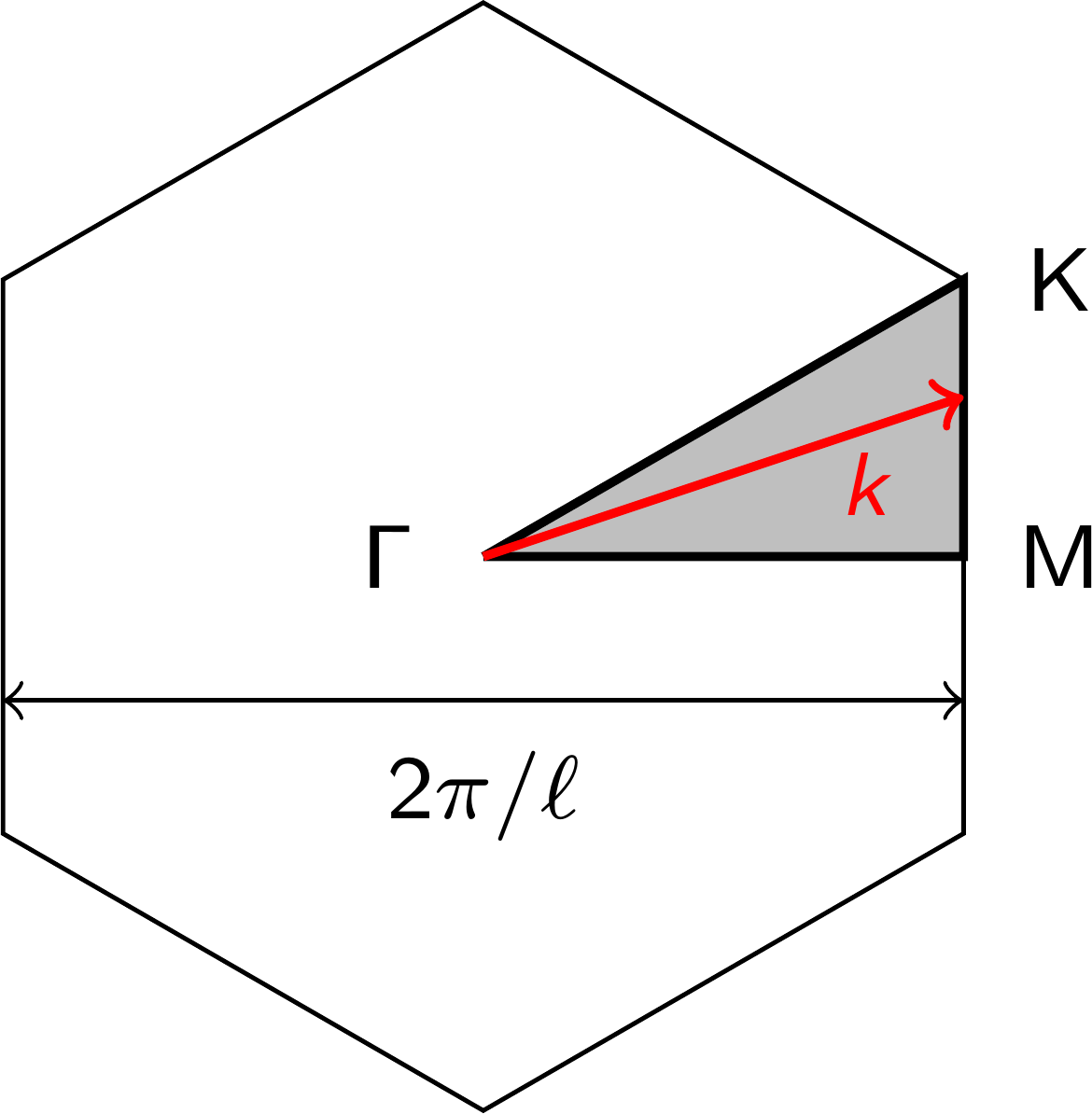}
\caption{Reciprocal unit cells for square and hexagonal cells, with irreducible Brillouin zone (gray triangle).}
\label{fig:cells:reciprocal}
\end{figure}

\subsubsection{Computations on the unit cell.} For each choice of periodicity cells, the steps outlined in Section \ref{sec:special:general} are performed as follows:
\begin{enumerate}
\item Cell geometry and meshes are generated using the free GMSH mesher~\cite{gmsh:09}.
\item Computations on the periodicity cell using the FreeFEM finite element platform~\cite{hecht:freefem}, with standard triangular elements and piecewise-linear $P_1$ interpolation.
\begin{enumerate} 
 \item Cell problems are solved. 
\item Homogenized tensors are computed. 
\item For comparison purposes, reference dispersion results (described below) are generated by Floquet-Bloch (FB) analysis~\cite{LaudeBook2015}.
\end{enumerate}
\item Definition of strain-gradient tensors and effective dispersion relations, using Python:
 \begin{enumerate}
\item Homogenized tensors yielded by steps (1b) and (1c) are imported.
\item The "strain" and "kinetic" eigenvalue problems~\eqref{e:eigenvalue:problem} and~\eqref{k:eigenvalue:problem} are solved to determine the thresholds $\vartheta_e$ and $\vartheta_k$. The weight parameter $\vartheta \geq \max(\vartheta_e, \vartheta_k)$ is then chosen; its value used for each testing configuration is given in Table \ref{tab:cells}. The "static" threshold $\thZ_e$ is also given for completeness (it differs from $\vartheta_e$ for the third case only).
\item Strain-gradient tensors are computed using~\eqref{corr:tens} with the chosen values of $\vartheta$ and cell size $\ell$.
\item Dispersion relations obtained with these models are computed by solving the generalized Christoffel equation \eqref{planewave:eig}, adding successively the first- and second-order terms to the leading-order Christoffel equation. 
\end{enumerate}
\end{enumerate}

Qualitative and quantitative comparisons between FB reference values and LOE and SG dispersion relations are then performed. In all cases, the cell size is kept constant (we chose $\ell = 1$ mm), while the wavevector $\upkappa\bfn$ evolves inside the irreducible Brillouin zone, see Figure \ref{fig:cells:reciprocal}. With reference to Remark \ref{rem:Lspec}, the macroscopic scale $L = \lambda = \upkappa/2\pi$ varies with the wavenumber $\upkappa$, and we therefore define the variable asymptotic parameter $\epsilon$ as
\begin{equation}
    \epsilon = \frac{\ell}{\lambda} = \frac{\upkappa \ell}{2 \pi} = \frac{\delta}{2\pi},
\end{equation}
see~\eqref{scale:sep}. Then, $\epsilon \shin [0,\tdemi]$ for waves propagating along the horizontal direction for all cells, see again Figure \ref{fig:cells:reciprocal}.

\begin{table}[t]
\begin{tabular}{l|ccc}
Material & Young's modulus $E$ (GPa) & Poisson's ratio $\nu$ & Mass density $\rho$ (kg.m$^{-3}$) \\ \hline
1 (light grey) & 200 & 0.3 & 7\,850 \\
2 (dark grey, $\rmZ_3$ cell only) & 50 & 0.35 & 31\,400 \\
\end{tabular}
\caption{Material properties (with reference to Figure \ref{fig:cells} for colors).}
\label{tab:materials}
\end{table}

\subsubsection{Wave propagation simulations}

To investigate wave propagation characteristics, transient dynamic analyses are performed on both the microstructured continuum and the homogenized LOE and SG approximating models. A shear pulse is introduced at the center of a sufficiently large finite domain (100mm$\times$100mm), ensuring that wave reflections from the boundaries do not affect the results during the simulation time. To simulate wave propagation in the actual microstructured medium and its homogenized counterparts, two distinct computational models are implemented in COMSOL Multiphysics$^\text{\textcopyright}$:
\begin{enumerate}[\hspace*{0.5em} (1)]
\item A benchmark code that solves the transient initial value problem~\eqref{ch5}, on the meshes shown in Figure~\ref{fig:mesh}. This code uses the plane strain assumption and triangular second-order serendipity elements. 
\begin{figure}[b]
\centering
\includegraphics[width=0.32\textwidth]{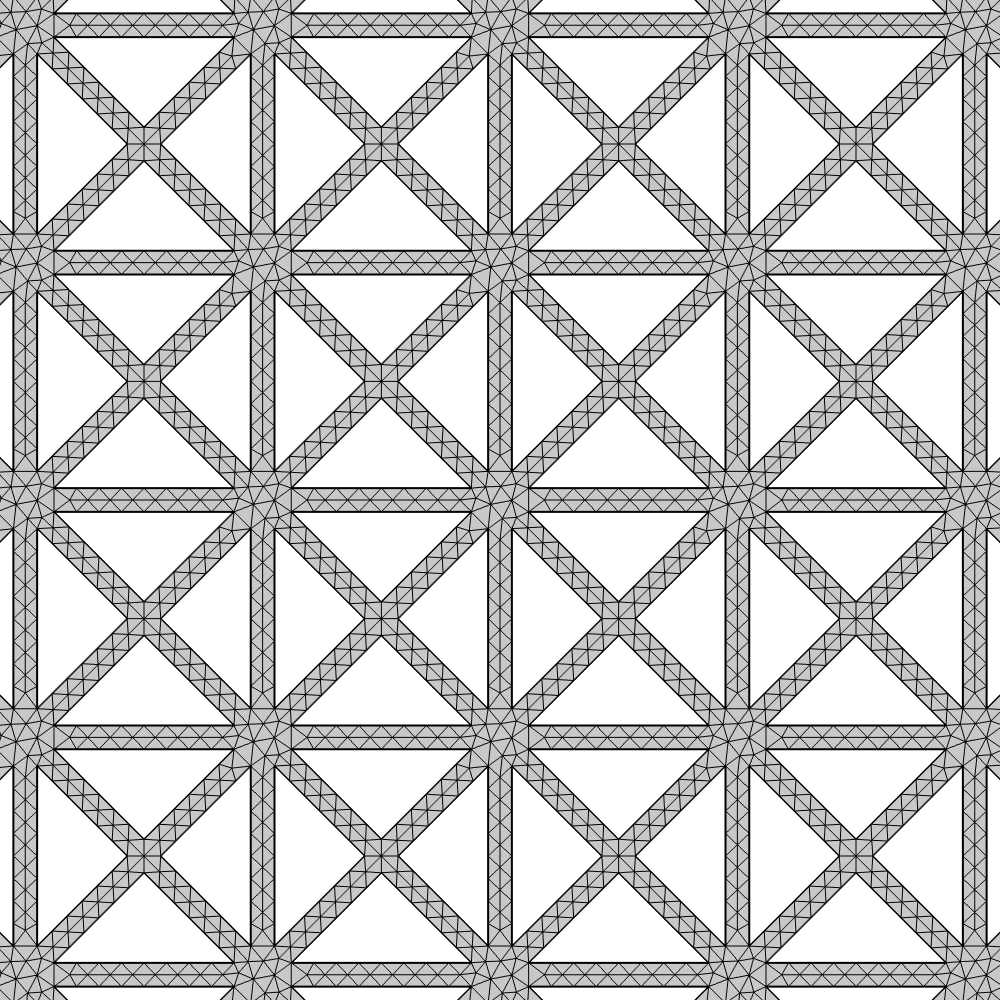} \
\includegraphics[width=0.32\textwidth]{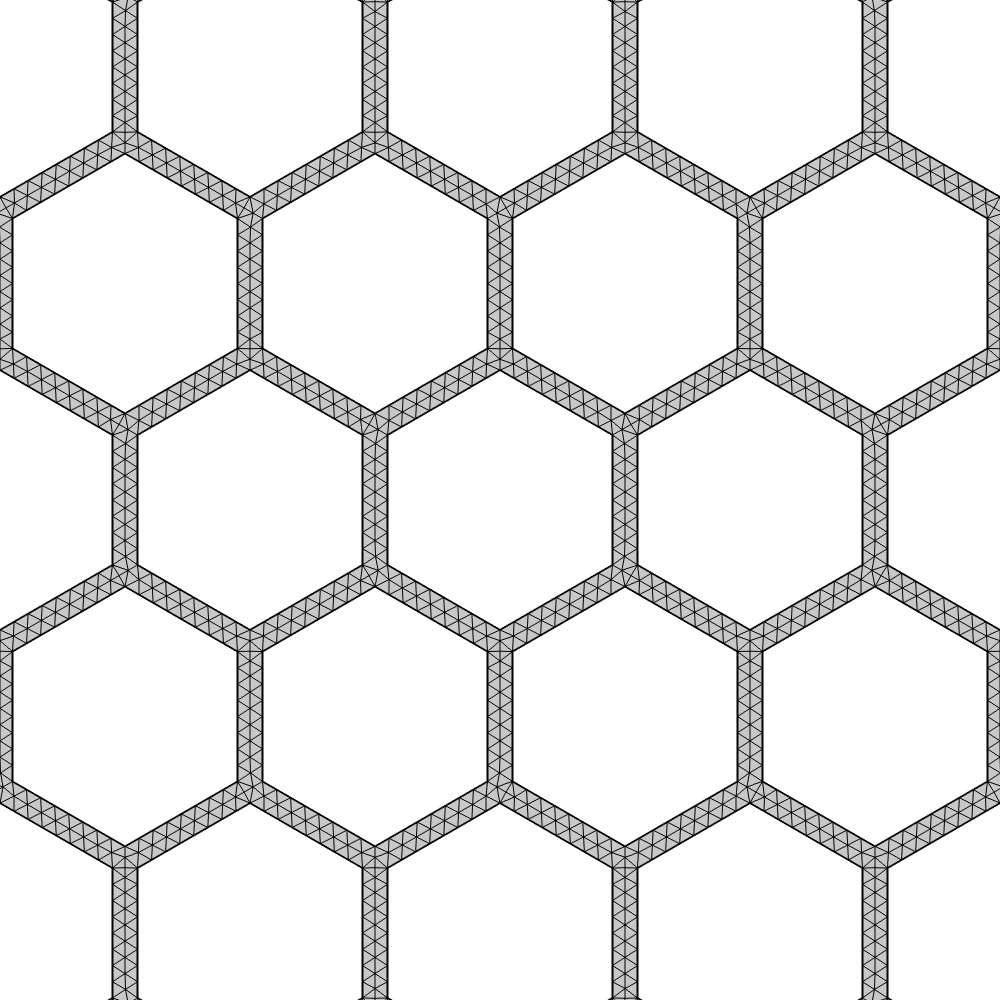} \
\includegraphics[width=0.32\textwidth]{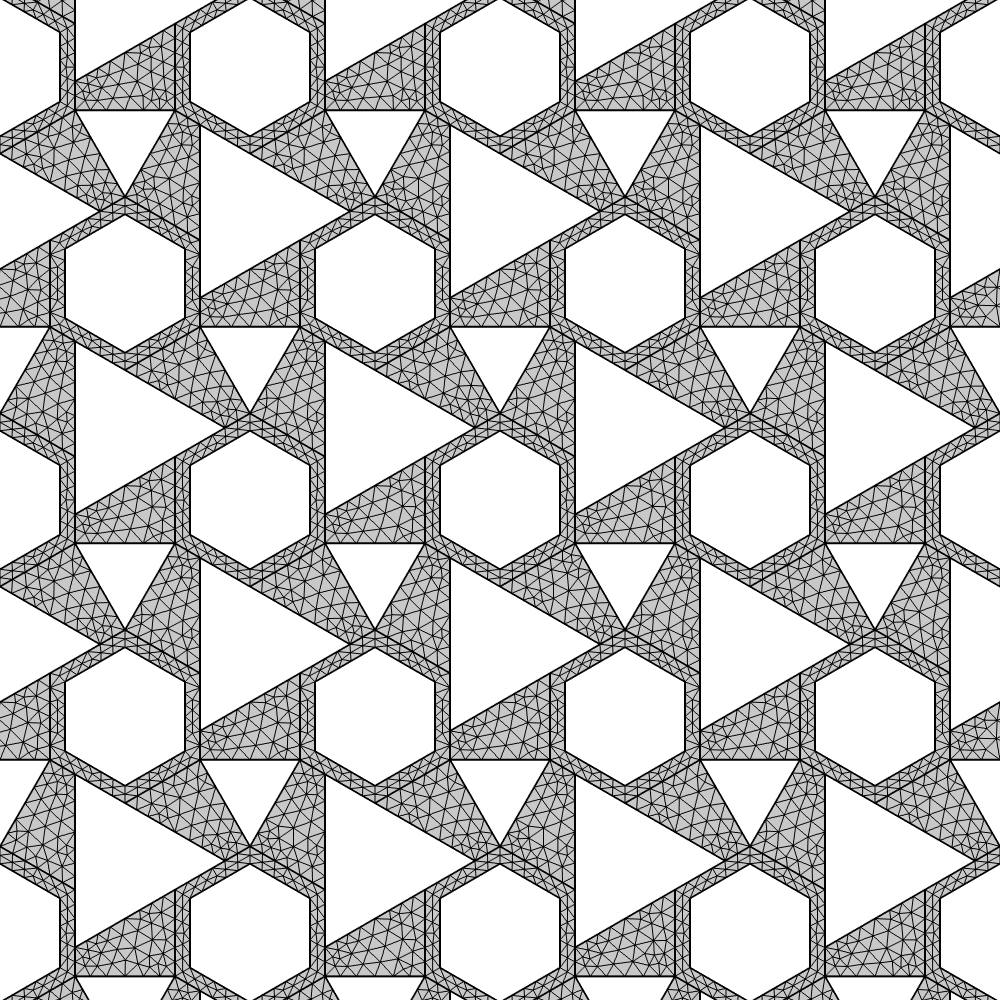}
\caption{From left to right: a zoom on the $\rmD_4$, $\rmD_{6}$ and $\mathrm{Z}_{3}$ meshes used for FEM computations. }
\label{fig:mesh}
\end{figure}

The following body force density, corresponding to a source for a shear wave, is considered:
\begin{equation}
  \brf^S(\brX,t)
 = f_0\exp\Lpar -4(f_c t-\tilde t_0)^2+\tdemi k_c^2(X_1^2+X_2^2) \Rpar\sin\left(2\pi f_c t\right)\left(X_2 \mathbf{e}_1-X_1 \mathbf{e}_2 \right),\label{eq:source_comp}
\end{equation}
where $f_c$ is the central frequency of the pulse, $\tilde{t}_0$ denotes the dimensionless temporal shift (expressed in units of the period) used to place the pulse in the center of the time window and $f_0$ is the amplitude of the force. The wavenumber $k_c$ is calculated so that the wave's wavenumber aligns with one-third of the Brillouin zone, i.e. $k_c=\frac{\pi}{3\ell}$, so that the "central" asymptotic parameter is $\epsilon_c=\frac16$. This choice provides a representative test in the dispersive regime while remaining sufficiently far from the Brillouin-zone boundary. The computational domain boundary is treated as traction-free, but this is not relevant as the simulation duration is set so that the waves do not reach the boundary.
\item A code that solves the transient initial value problem ~\eqref{IVP} based on the SGE model. As the weak formulation~\eqref{SGE:weak} involves second-order spatial derivatives of displacements, fifth-order Argyris triangular elements, which are $H^2$-conforming~\cite[Thm.~2.2.13]{cia:80}, are employed. The expression for the body force density is modified as follows: 
\begin{equation}
  \brF^S(\brx,t)
 = \phi f_0\exp\Lpar -4(f_c t-\tilde t_0)^2+\tdemi k_c^2(X_1^2+X_2^2) \Rpar\sin\left(2\pi f_c t\right)\left(-x_2 \mathbf{e}_1+x_1 \mathbf{e}_2 \right),\label{eq:source_cont}
\end{equation}
where $\phi$ denotes the volume fraction of material within the periodicity cell (see Table~\ref{tab:cells}). This factor accounts for the presence of void regions in the unit cell of the microstructured medium. Finally, the boundary of the computational domain is assumed to be traction free. In the case of SG, this condition is particularly relevant because the solution consists of a propagating component with finite phase velocity and a decaying component that propagates with infinite velocity.
\end{enumerate}

The computational cost of the benchmark simulations depends on the complexity of the considered microstructure, with between $6.22\times10^6$ and $1.79\times10^7$ degrees of freedom (see Table~\ref{tab:cells}), resulting in computation times ranging from approximately $7$~h to $45$~h. By contrast, the homogenized SG model involves only $2.27\times10^5$ degrees of freedom and requires approximately $5$~min of computation time for all three periodicity cells. All computations were performed in COMSOL Multiphysics$^\text{\textcopyright}$ using the direct solver MUMPS on an Apple M1 Ultra workstation (16 CPU cores). In all cases, the time step was selected according to a CFL condition, ensuring a consistent comparison of the reported computation times.

In what follows, we present (i) dispersion curves to assess the asymptotic accuracy of the models and (ii) snapshots of computed transient responses to assess the ability of the proposed effective SGE model to reproduce the main features of wave propagation predicted by the benchmark simulations, for each periodicity cell of Figure~\ref{fig:cells}. In the transient computations, the displacement magnitude is plotted.\enlargethispage*{1ex}


\subsection{$\rmD_4$ periodicity cell (square with reinforced diagonals)}
\label{sec:D4X}

We first consider the $\rmD_4$ cell represented in Figure \ref{fig:cells}. It is centrosymmetric (so that the coupling tensors $\brKsh,\brMsh$ vanish as discussed in Section \ref{sec:special:centrosym}) and the mass density is homogeneous (so that the SG tensors simplify as discussed in Section \ref{sec:special:homogeneous:rho}). Ultimately, the SG model is defined by the effective properties $(\rhoz,\brCz)$ and SG tensors $(\brA(\vartheta),\brI(\vartheta))$ given by \eqref{eq:A:J:homog:rho}.
The threshold values $\vartheta_e$ and $\vartheta_k$ are given in Table \ref{tab:cells}: we chose a corrective weight $\vartheta=0.16$.

\begin{figure}[b]
 \centering
 \includegraphics[width=\textwidth]{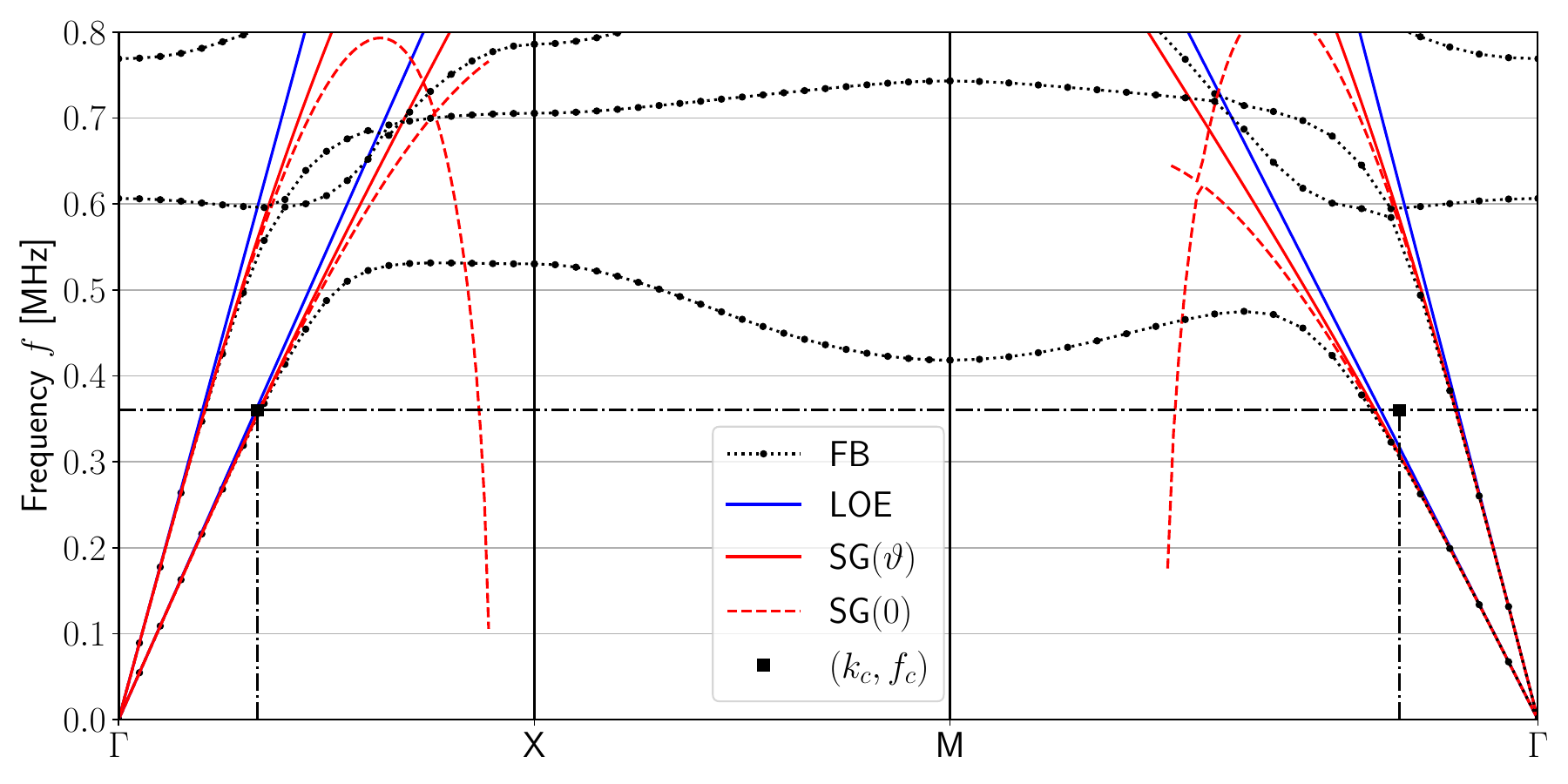}
 \caption{$\rmD_4$ unit cell: dispersion diagram and dispersion branches of the LOE and SG models.}
 \label{fig:D4X:GXM}
\end{figure}

In Figure
\ref{fig:D4X:GXM}, the dispersion diagram is plotted for a wavevector following the edges of the irreducible Brillouin zone, see Figure \ref{fig:cells:reciprocal}: we plot the reference FB frequencies and the approximations given by the LOE and SG models. As expected, the LOE provides a non-dispersive approximation, valid only for low wavenumbers, while the SG models provide a dispersive approximation valid for higher wavenumbers. The non-modified SG(0) arising from homogenization provides a more accurate approximation, but is ill-posed, as seen here by the non-definiteness of the dispersion branches as the wavenumber increases (the eigenvalue $\omega^2$ of the Christoffel equation becomes negative, and $\omega$ cease to be real, hence the incomplete plot). By contrast, the well-posed \SGvt model provides valid dispersion branches in the whole Brillouin zone, but loses some accuracy.  For future reference, we also display the position of the point $(k_c,f_c)$ in the dispersion diagram, $k_c$ and $f_c$ being the "central" wavenumber and frequency that enter the definition of the source terms \eqref{eq:source_comp} and \eqref{eq:source_cont} used for upcoming transient propagation simulations.

Figure \ref{fig:D4X:disp:errors} presents the errors on the dispersion relation made using these models. This error is computed up to the middle of the first Brillouin zone, \ie $\upkappa \ell \leq \pi/2$, so that the studied asymptotic range is $\epsilon \in [0,0.25]$. It is averaged on ten directions of propagation $\theta = 0,5,\dots,45^\circ$. The LOE approximates the exact relation up to first order (an expected result since there is no odd-order term in the homogenized model), while the SG models provide a third-order approximation. The slight accuracy decrease when using the well-posed \SGvt model is again noticeable.

\begin{figure}[t]
 \centering
 \includegraphics[width=\textwidth]{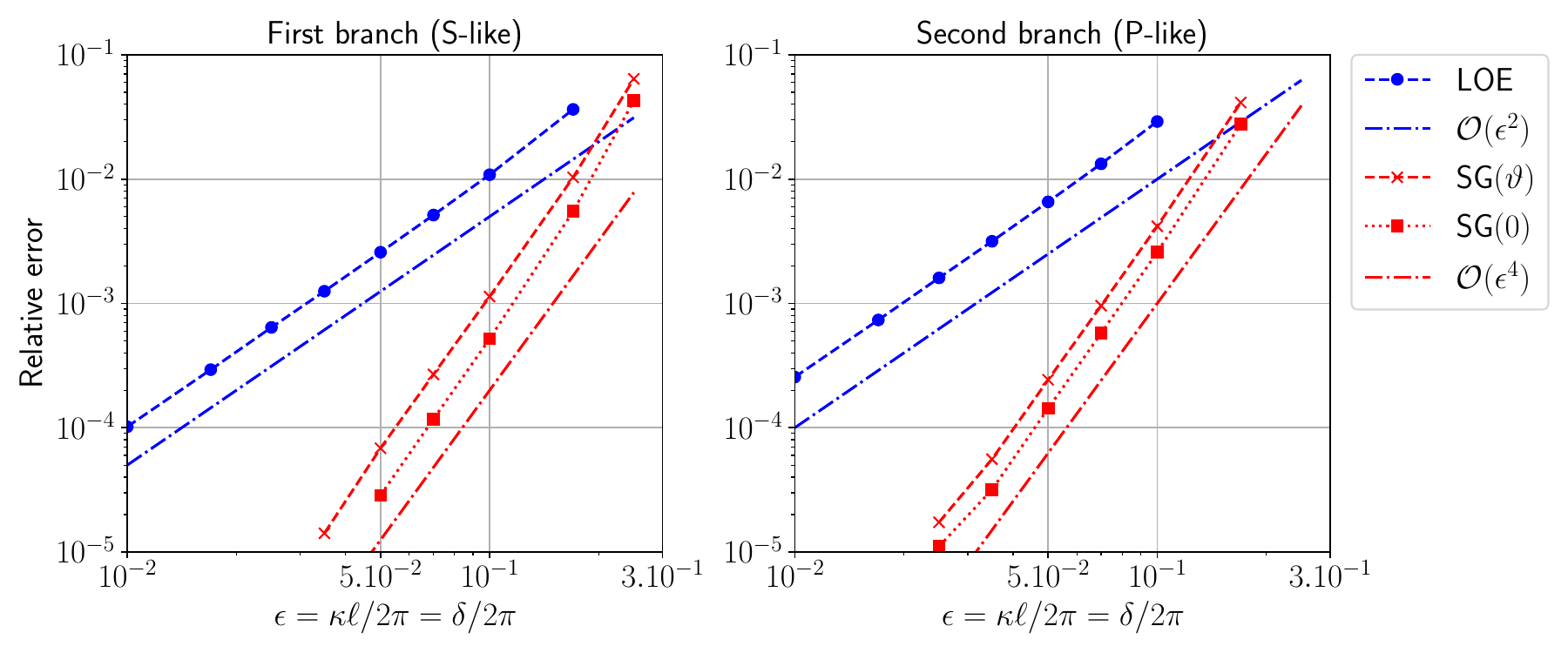}
 \caption{$\rmD_4$ unit cell: errors on the dispersion relations for all models, averaged on
the propagation direction $\theta$ (here errors for branches at $\theta = 0,5,\dots,45^\circ$ were computed), and up to $\upkappa=\pi/2\ell$, \ie the middle of the Brillouin zone in the horizontal direction.}
 \label{fig:D4X:disp:errors}
 \end{figure}

 \begin{figure}[t]
 \centering
 \includegraphics[width=0.3\textwidth]{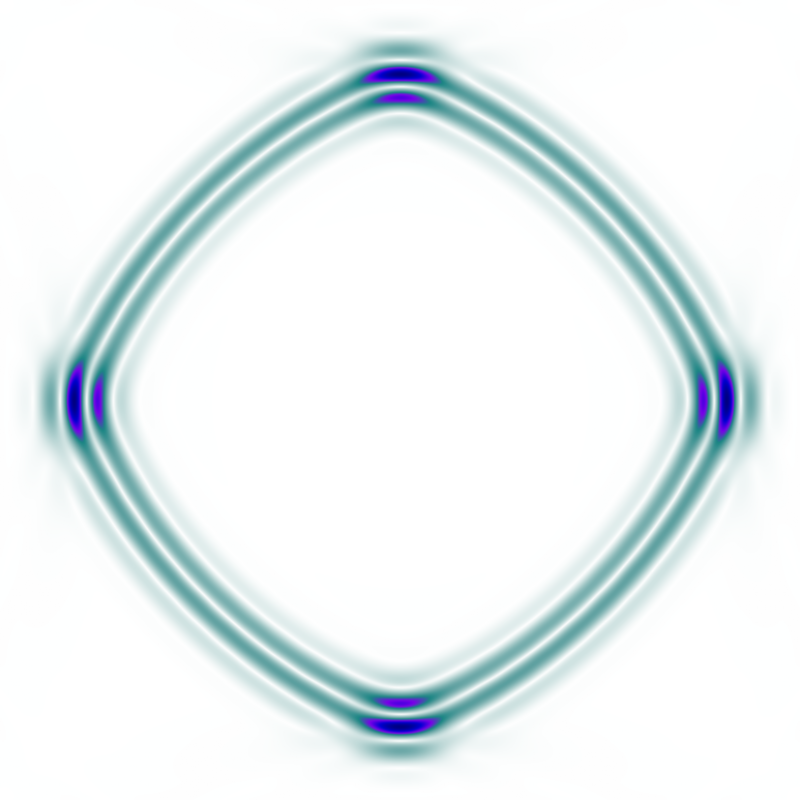} \
 \includegraphics[width=0.3\textwidth]{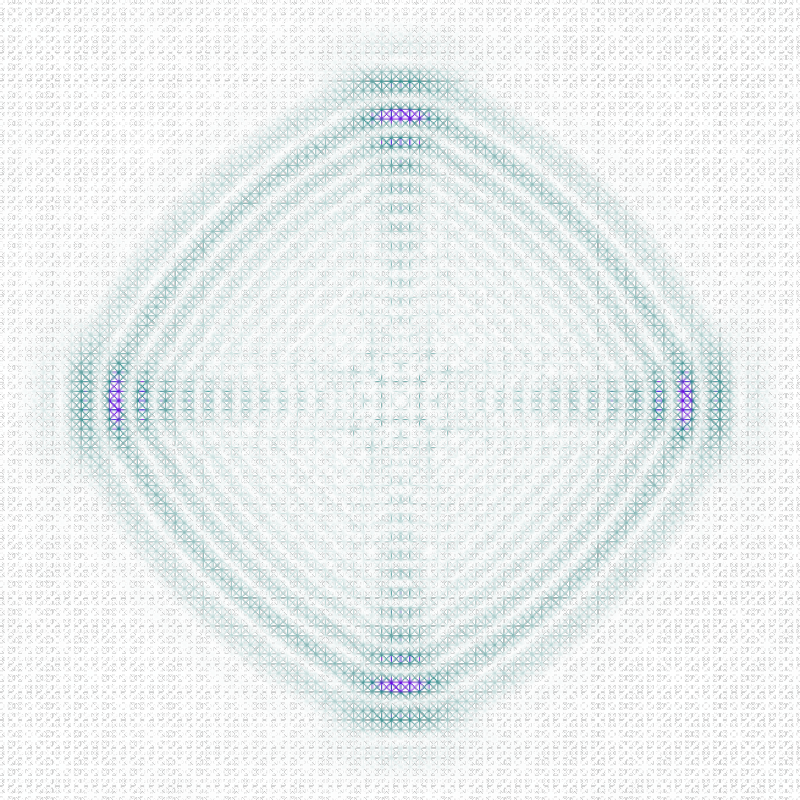} \
 \includegraphics[width=0.3\textwidth]{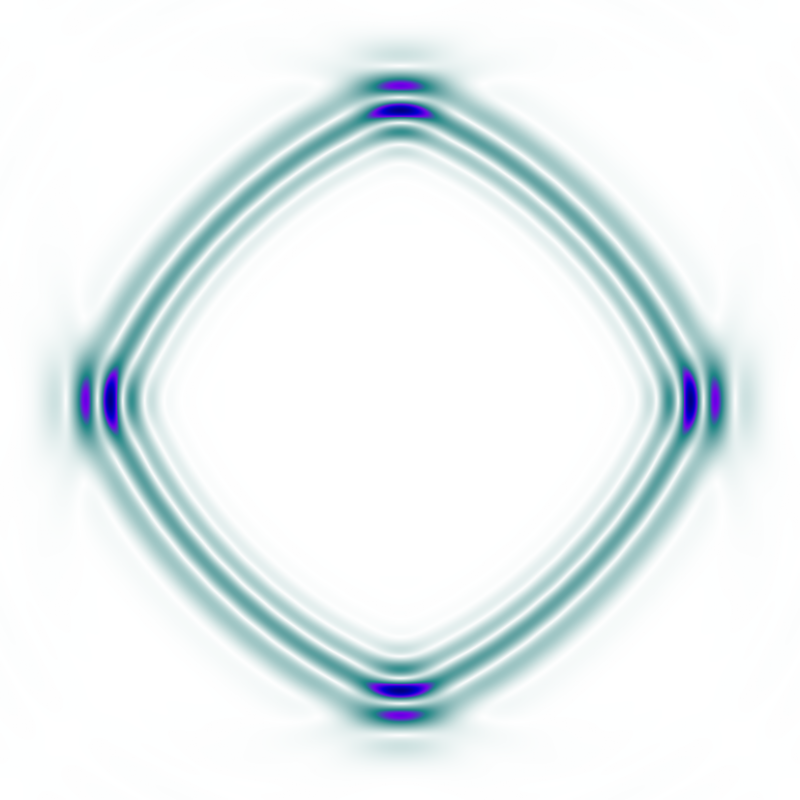}

  \includegraphics[width=0.4\textwidth]{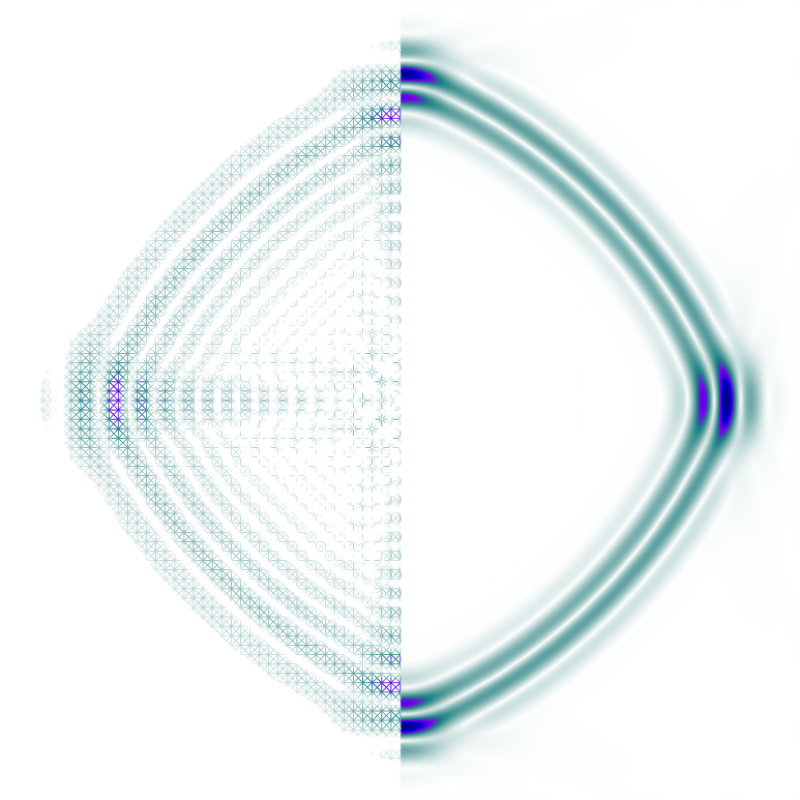} \qquad
  \includegraphics[width=0.4\textwidth]{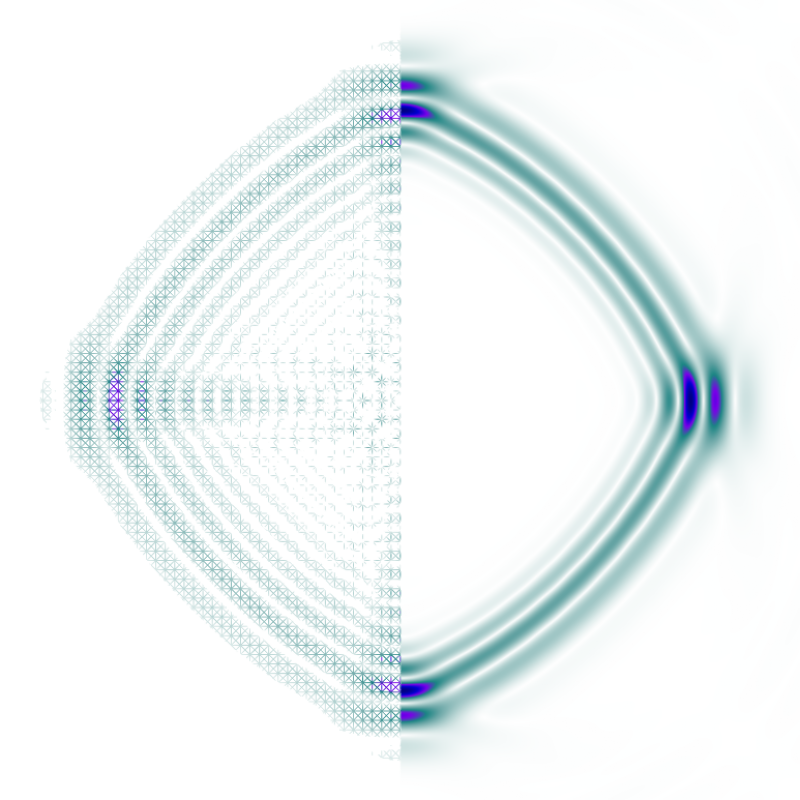}
\caption{$\rmD_4$ unit cell: shear waves generated using a central frequency $0.36\,\mathrm{MHz}$.  Wavefronts (displacement magnitude) at $t =0.20\,\mu\mathrm{s}$, in the leading-order homogenized Cauchy medium (top left), the microstructure (top middle) and the \SGvt medium (top right). Comparisons between the microstructured and homogenized solutions: LOE (bottom left) and \SGvt (bottom right).} 
\label{fig:D4X:propagation}
\end{figure}


Figure~\ref{fig:D4X:propagation} first illustrates the response of the three models to the body force defined in Equation~\eqref{eq:source_comp}. The benchmark solution, corresponding to the microstructured continuum, exhibits both anisotropic and strongly dispersive wave propagation. As expected, the anisotropy reflects the symmetry of the architecture, characterized by two orthogonal mirror planes. In addition, backscattering-like effects are observed, which is consistent with the presence of a nearly flat shear-wave branch at slightly higher frequencies. As predicted by the dispersion diagrams of Figure~\ref{fig:D4X:GXM}, this feature is captured by neither the leading-order equivalent (LOE) model nor the strain-gradient (SG) model. Nevertheless, the higher-order corrections incorporated in the SG model provide a more accurate prediction of the position and shape of the main wavefront. In contrast, the wavefront predicted by the LOE model propagates noticeably faster than that of the microstructured medium, consistently with the overestimation of the shear-wave velocity observed on the dispersion diagrams.


\subsection{$\rmD_6$ periodicity cell (hexagon)}
\label{sec:D6}

This second example concerns a hexagonal periodicity cell, also centrosymmetric and with homogeneous mass density, so that the comments made in Section \ref{sec:D4X} about the various tensors also apply here. This time a larger value $\vartheta = 0.35$ is chosen for the model to be well-posed.

Again we plot the FB dispersion diagram, and dispersion branches for the homogenized models, see Figure \ref{fig:D6:GMK}. The same observations as for the $\rmD_4$ cell hold: the \SGz model provides a good low-wavenumber agreement but fails to describe the entire Brillouin-zone for the $\Gamma$-M branch (propagation at $30^\circ$). The well-posed \SGvt model presents a less dispersive behavior.\enlargethispage*{5ex}

Figure~\ref{fig:D6:disp:errors} presents the relative error on dispersion for both branches, averaged on the propagation direction $\theta \in [0,\pi/6]$. Again second-order and fourth-order errors are observed for the LOE and SG models, respectively. 

\begin{figure}[b]
 \centering
 \includegraphics[width=\textwidth]{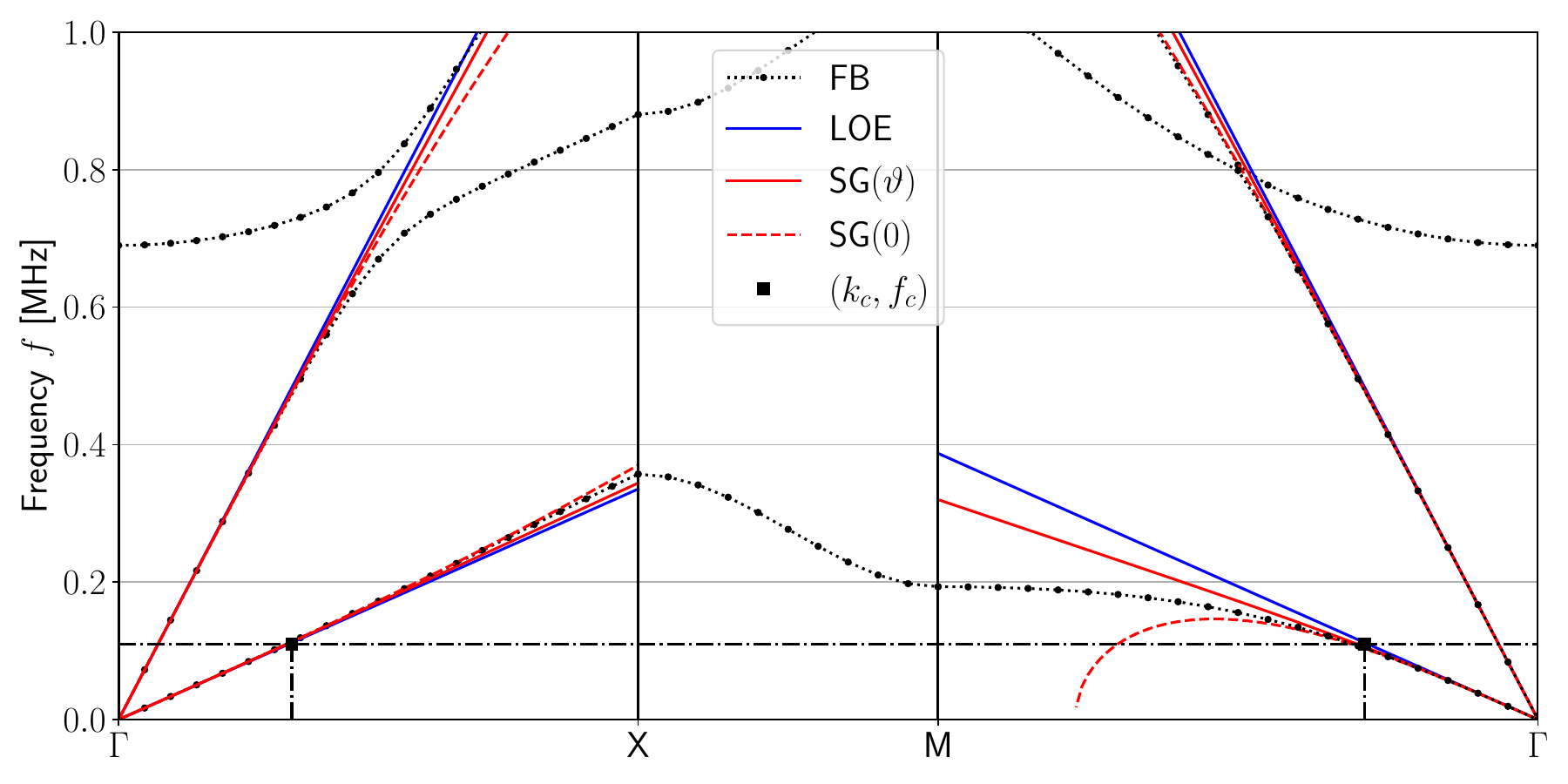}
 \caption{$\rmD_6$ unit cell: dispersion diagram and dispersion branches of the LOE and SG models.}
 \label{fig:D6:GMK}
\end{figure}

\begin{figure}[b]
 \centering
 \includegraphics[width=\textwidth]{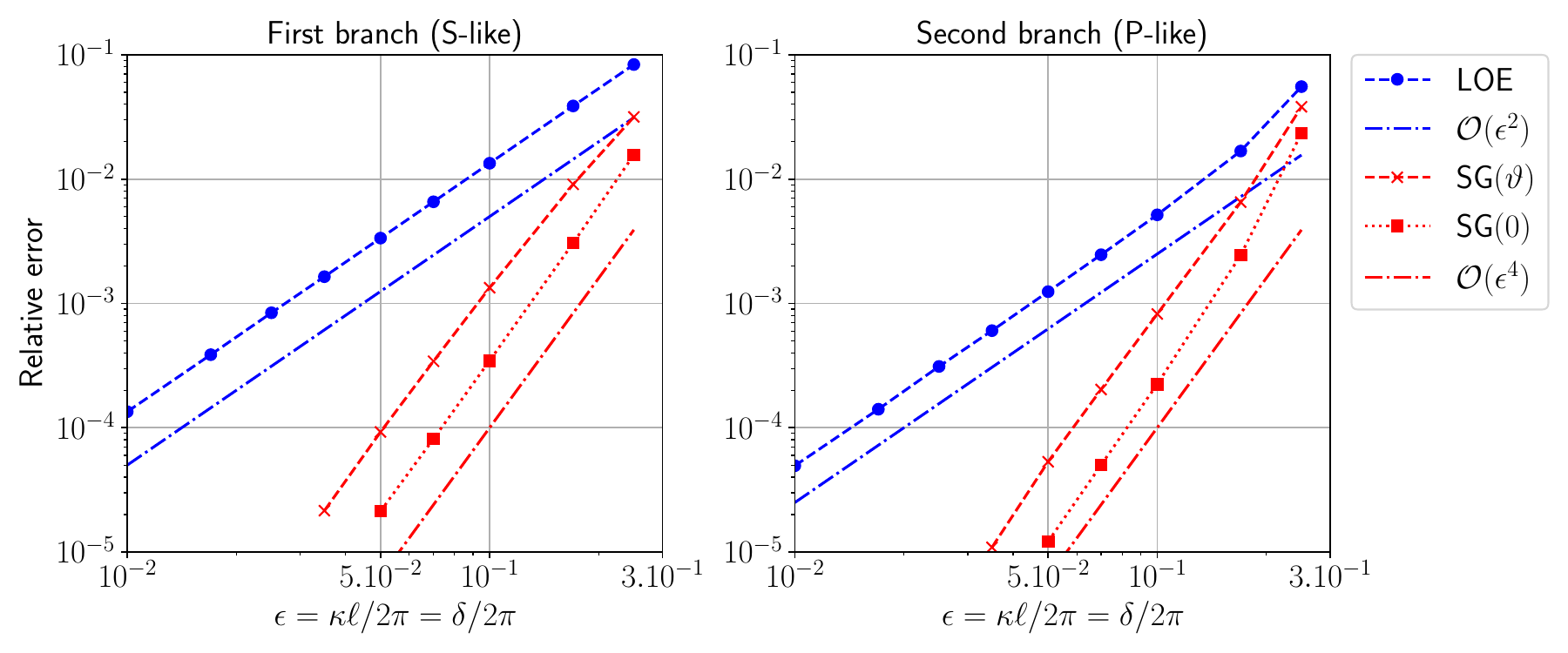}
 \caption{$\rmD_6$ unit cell: errors on the dispersion relations for all models, averaged on
the propagation direction $\theta$ (here errors for branches at $\theta = 0,5,\dots,30\circ$ were computed), and up to $\upkappa=\pi/2\ell$, \ie the middle of the Brillouin zone in the horizontal direction.}
 \label{fig:D6:disp:errors}
\end{figure}

Figure~\ref{fig:D6:propagation} shows the responses of the three models to the body force defined in Equation~\eqref{eq:source_comp}. The LOE model is isotropic as expected, and therefore does not capture the pronounced anisotropic behavior of the benchmark solution, arising from the hexagonal symmetry of the underlying microstructure. By contrast, the SG model successfully reproduces this anisotropic wave propagation, in agreement with the anisotropy predicted by the corresponding dispersion branches, as highlighted by the provided side-by-side comparison. These observations are consistent with the findings reported in \cite{rosi:19}, where the SG model was obtained using an identification method rather than homogenization.\enlargethispage*{1ex}

\begin{figure}[t]
 \centering
 \includegraphics[width=0.4\textwidth]{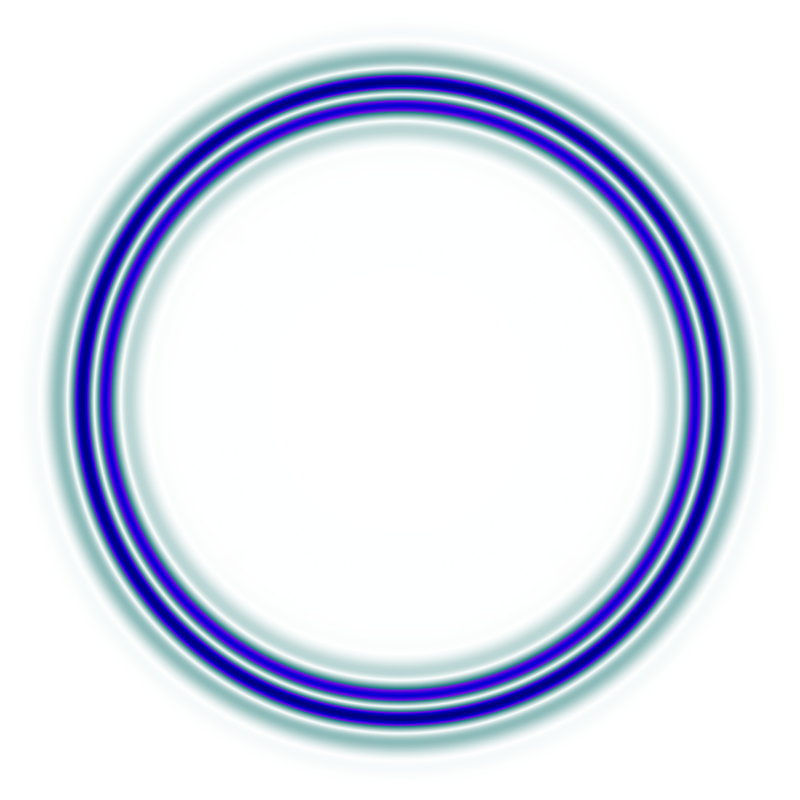}
  \includegraphics[width=0.4\textwidth]{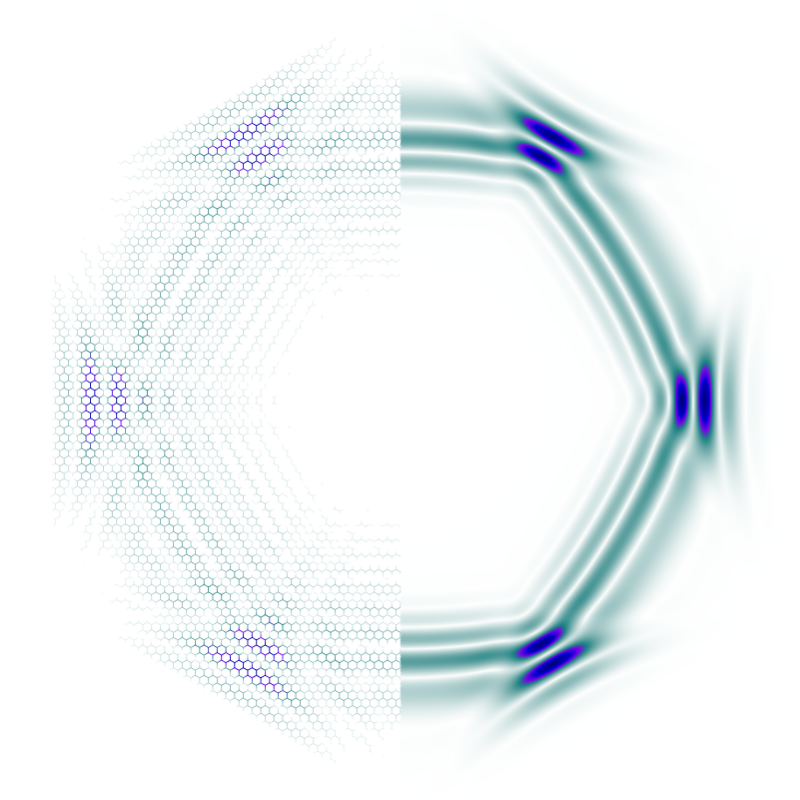}
 \caption{$\rmD_6$ unit cell: shear waves generated using a central frequency $0.11\,\mathrm{MHz}$.  Wavefronts (displacement intensity) at $t =0.67\,\mu\mathrm{s}$ in the leading-order homogenized Cauchy medium (left), and in the microstructure and the \SGvt medium, plotted side-by-side for comparison (right).} 
\label{fig:D6:propagation}
\end{figure}


\subsection{$\mathrm{Z}_{3}$ periodicity cell with two materials}

Our final example concerns a $\mathrm{Z}_{3}$ periodicity cell made of two materials, which is neither centrosymmetric nor homogeneous: it is an example of the general case, in which odd-order tensors $\brKsh$ and
$\brMsh$ as well as the inertial tensor $\brcii$ \emph{a priori} enter the SG models. However, in this case, $\brK^\sharp$ vanishes as a consequence of the non-trivial rotational invariance of the material. Indeed, it can be shown that $\brK^\sharp$ behaves as a first-order tensor with respect to orthogonal transformations. It follows that any invariance with respect to a non-trivial rotation implies the vanishing of $\brK^\sharp$. As a result, the only symmetry class in $\Rbb^2$
 for which $\brK^\sharp$ has a non-zero contribution reduces to complete anisotropy\footnote{The situation is obviously completely different in $\Rbb^3$.}.

\begin{figure}[b]
 \centering
 \includegraphics[width=\textwidth]{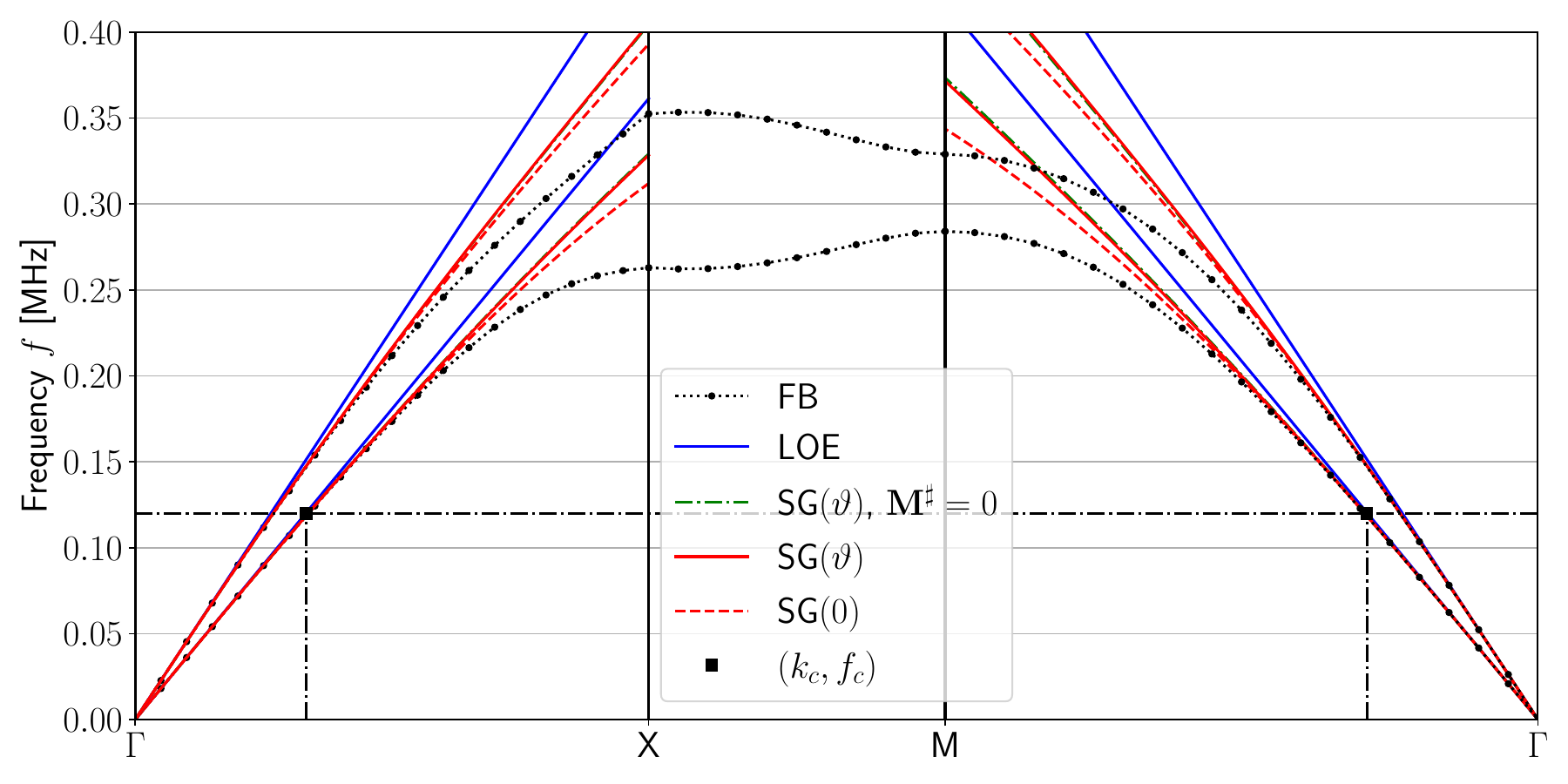}
 \caption{$\rmZ_3$ unit cell: dispersion diagram and dispersion branches of the LOE and SG models.}
 \label{fig:Z32:GMK}
\end{figure}

The dispersion diagram is represented in Figure \ref{fig:Z32:GMK}, while the corresponding errors are represented in Figure \ref{fig:Z32:disp:errors}. The correction amplitude is weaker than in the previous cases ($\vartheta = 0.05$), and the branches of the \SGvt model are consequently closer to those of the \SGz model. Contrarily to the previous examples, both SG models provide dispersion curves spanning the entire Brillouin zone: \SGz does not appear to fail on this representation. 

Errors are plotted in Figure \ref{fig:Z32:disp:errors}, again averaged for $\theta \in [0,\pi/6]$. First, as predicted by Proposition \ref{prop:disp:expansion}, there is no first-order contribution in the dispersion expansion (because the P- and S-wave velocity are distinct), and the LOE error remains of second-order even if there are first-order terms in the SG model. Both SG errors are of fourth-order as usual.\enlargethispage*{5ex}

To stress the importance of the odd-order tensor $\brMsh$, we also plotted the dispersion branches and errors for an incomplete SG model, without this tensor. In this case, the thresholds are slightly modified: one obtains $\vartheta_e = 0.0468$ (the chosen weight $\vartheta = 0.05$ still guaranties a well-posed model). The difference is not really noticeable on the dispersion diagram, but is more apparent on the error plot. On the the S-like branch, the error stays locked to second-order as for LOE, even if adding the second-order SG tensors $(\brJ,\brA)$ to the LOE provides a one-order-of-magnitude improvement. In contrast, the P-like branch seems unaffected. 

Although the odd-order tensor $\brMsh$ significantly improves the asymptotic accuracy of the dispersion relation, its influence on the transient wavefield is less apparent for the present excitation. This suggests that its contribution primarily manifests through higher-order corrections to wave dispersion rather than through readily observable changes in the overall wavefront morphology. A more detailed analysis of the influence of $\brMsh$ (and $\brKsh$, when nonzero) on the Christoffel equation is left for future work.
 \enlargethispage*{1ex}


\begin{figure}[t]
 \centering
 \includegraphics[width=\textwidth]{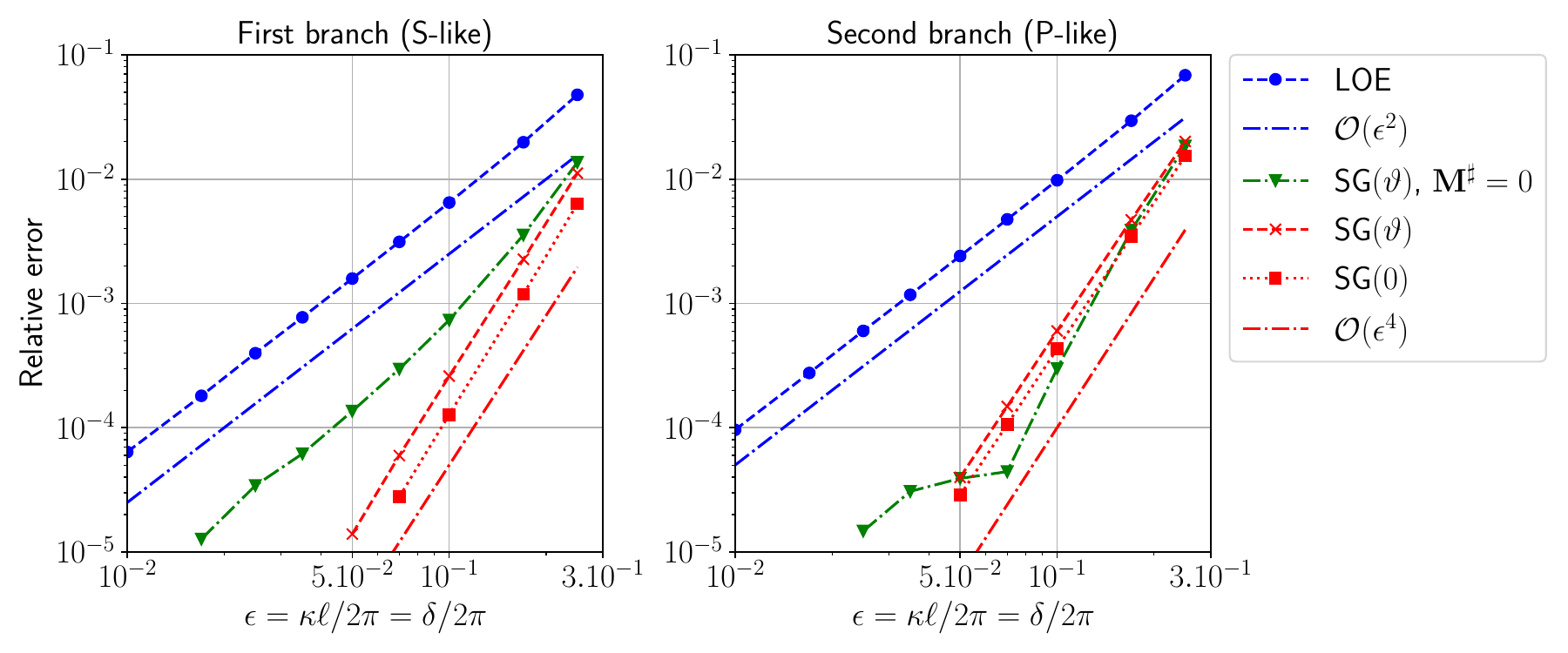}
 \caption{$\rmZ_3$ unit cell: errors on the dispersion relations for all models, averaged on the propagation direction $\theta$ (here errors for branches at $\theta = 0,5,\dots,30^\circ$ were computed), and up to $\upkappa=\pi/2\ell$, \ie the middle of the Brillouin zone in the horizontal direction.}
 \label{fig:Z32:disp:errors}
\end{figure}

\begin{figure}[t]
 \centering
 \includegraphics[width=0.3\textwidth]{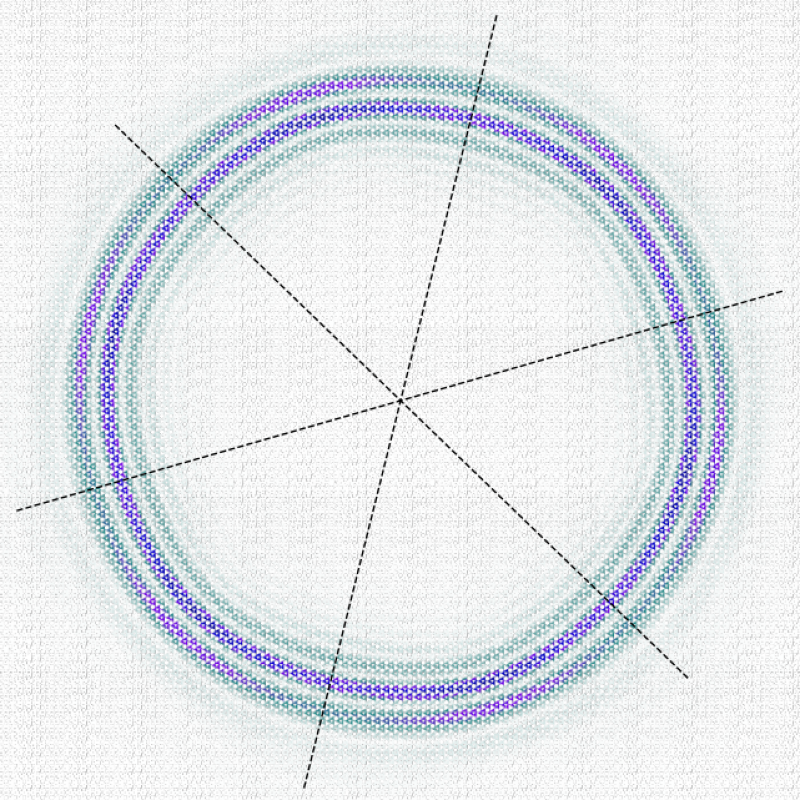}
 \includegraphics[width=0.3\textwidth]{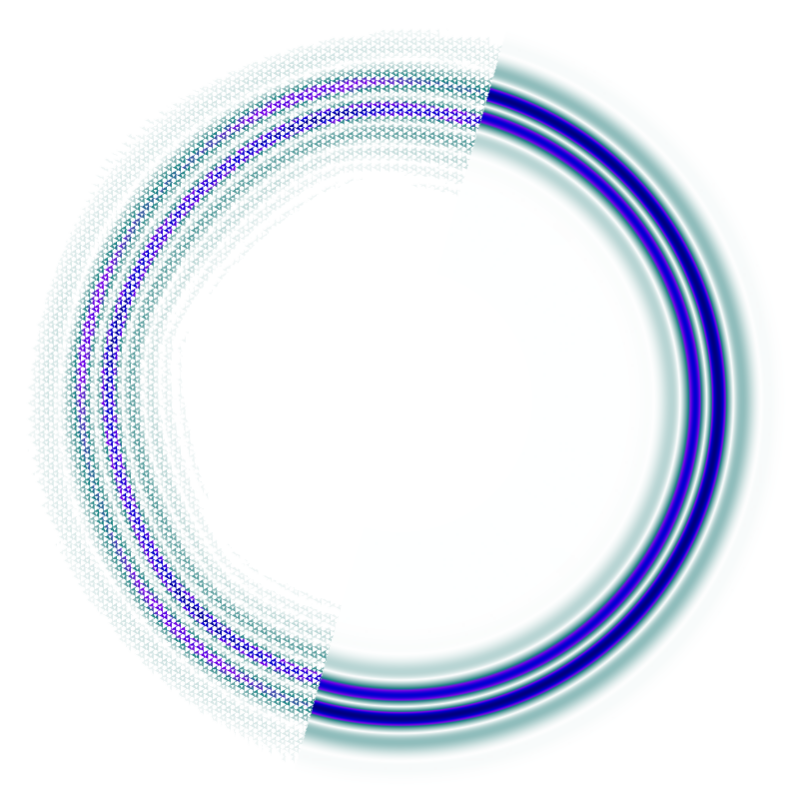}
 \includegraphics[width=0.3\textwidth]{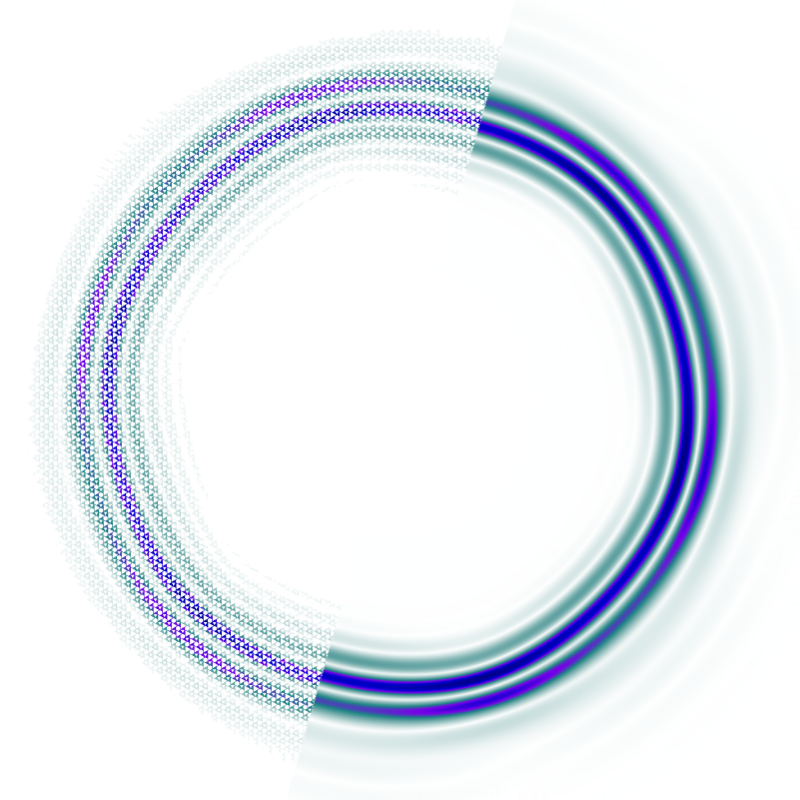}
  \caption{$\rmZ_3$ unit cell: shear waves generated using a central frequency $0.12\,\mathrm{MHz}$. Wavefronts (displacement magnitude) at $t =0.62\,\mu\mathrm{s}$ in the microstructured medium (left) and comparisons with the leading-order homogenized Cauchy medium (middle), and in the \SGvt medium (right).} 
\label{fig:Z3:propagation}
\end{figure}


As for the other choices of cell, Figure~\ref{fig:Z3:propagation} finally illustrate the transient wave propagation. Considering first the benchmark solution, the propagation is clearly anisotropic, with six "valleys" of lower displacement magnitude being identifiable and highlighted by dashed lines. As the cell does not have any mirror symmetry lines, these "propagation symmetry lines" are not directly correlated with the chosen geometry (in particular they do not align with the symmetry axis of the hexagon as in the previous case). On the displayed comparison, one can observe the expected isotropic behavior of the LOE model and the anisotropy of the \SGvt model, which reproduces that of the benchmark solution. In addition, the amplitudes of the first wavefronts are more accurately predicted by the \SGvt model.

Overall, the transient simulations confirm the conclusions drawn from the dispersion analysis: compared with the leading-order homogenized model, the proposed well-posed SGE model provides a significantly improved description of wavefront propagation while remaining computationally much less expensive than direct simulations of the microstructured media.

\section{Proofs}
\label{sec:proofs}

\subsection{Proof of Proposition~\ref{prop:recip:zf}}
\label{sec:proof:recip:zf}

We prove the expressions given for all three effective elastostatic tensors, even though the first two results are not new; in addition to completeness, this allows to emphasize the benefit and insight provided by systematic use of reciprocity identities satisfied by various pairs of cell solutions.

\proofstep{Proof of identity~\eqref{zf:reciprocity 0}} With $\bfP^{ij}$ as defined in Proposition~\ref{prop:recip:zf}, we find by simple calculations and using component notation that
\begin{equation}
  C_{ijk\ell} = P^{ij}_{p,q}C_{pqrs}P^{k\ell}_{r,s}, \qquad
  C_{ijrs}\chi^{1k\ell}_{r,s} = P^{ij}_{p,q}C_{pqrs}\chi^{1k\ell}_{r,s},
\end{equation}
so that, recalling the definition~\eqref{F:generic} of $A(\dotp,\dotp)$, the defining formula~\eqref{C0:comp} for $\brCz$ can be recast as
\begin{equation}
  |\Ycal|\CZ_{ijk\ell} = A\lpar \bfP^{ij},\,\bfP^{k\ell}\shp\bfchi^{1k\ell} \rpar.
\end{equation}
Moreover, using $\bfchi^{1k\ell}$ as test function in the variational formulation~\eqref{cell:generic:weak}, \eqref{weak form chi1 index} of $\bfchi^{1ij}$, we obtain
\begin{equation}
  A\lpar \bfchi^{1ij},\,\bfchi^{1k\ell} \rpar + A\lpar \bfP^{ij},\,\bfchi^{1k\ell} \rpar = 0,
\end{equation}
and summing the last two equalities yield the desired alternative expression~\eqref{zf:reciprocity 0} of $\brCz$.

\proofstep{Proof of identity~\eqref{zf:reciprocity 1}} Invoking the reciprocity identity~\eqref{cell:generic:recip} with $\bfw_1=\bfchi^{1ij}$ and $\bfw_2=\bfchi^{2k\ell m}$, we have
\begin{equation}
  F^{ij}\lpar \bfchi^{2k\ell m} \rpar - F^{k\ell m}\lpar \bfchi^{1ij} \rpar = 0,
\end{equation}
where the linear functionals $F^{ij}$ and $F^{k\ell m}$ are respectively defined by~\eqref{weak form chi1 index} and~\eqref{zf:weak form chi2 index}. The above equality thus yields
\begin{align}
 0 &= \iY C_{rsij} \chi^{2k\ell m}_{r,s} \dy + \iY \Scal_{mp}^{0k\ell} \chi^{1ij}_{p} \dy
  - \iY \chi^{1k\ell}_s C_{smpq} \chi^{1ij}_{p,q} \dy \\
  &= |\Ycal|\,\Ci_{ijk\ell m}
  + \iY \Scal_{mp}^{0k\ell} \chi^{1ij}_{p}\dy - \iY \Scal_{mp}^{0ij} \chi^{1k\ell}_{p}\dy. \label{aux07b}
\end{align}
the second equality provides the sought expression~\eqref{zf:reciprocity 1}; it results from using the expression~\eqref{zf:C1:comp} of the static effective tensor $\brCi$ in the first integral, then using definition~\eqref{S0:def} of $\bSz^{(ij)}$ and rearranging terms.

\proofstep{Proof of identity~\eqref{zf:reciprocity 2}} We now invoke the reciprocity identity~\eqref{cell:generic:recip} with $\bfw_1=\bfchi^{3k\ell mn}$ and $\bfw_2=\bfchi^{1ij}$, which yields
\begin{equation}
  F^{k\ell mn}\lpar \bfchi^{1ij} \rpar = F^{ij}\lpar \bfchi^{3k\ell mn} \rpar,
\end{equation}
where the linear functionals $F^{ij}$ and $F^{k\ell mn}$ are respectively given by~\eqref{weak form chi1 index} and~\eqref{zf:weak form chi3 index}. Another identity results from setting $\bfv=\bfchi^{2k\ell m}$ in the weak formulation~\eqref{zf:weak form chi2 index} satisfied by $\bfchi^{2ijn}$, to obtain
\begin{equation}
  F^{ijn}(\bfchi^{2k\ell m}) = A(\bfchi^{2ijn},\bfchi^{2k\ell m}).
\end{equation}
Adding the two previous equalities, we then have
\begin{equation}
  \underbrace{F^{ijn}(\bfchi^{2k\ell m})}_{(A)}
  + \underbrace{F^{k\ell mn}\lpar \bfchi^{1ij} \rpar}_{(C)}
  = \underbrace{A(\bfchi^{2ijn},\bfchi^{2k\ell m})}_{(B)}
   + \underbrace{F^{ij}\lpar \bfchi^{3k\ell mn} \rpar}_{(D)}, \label{zf:aux08}
\end{equation}
whose explicit form will provide the claimed expression. To this aim, we note that~\eqref{F:generic} and~\eqref{zf:weak form chi2 index} provide
\begin{align}
  (A)
 &= \iY \Lcb \lpar \Scal_{nr}^{0ij} - \CZ_{ijnr} \rpar \chi^{2k\ell m}_r - \chi^{1ij}_p C_{pnrs}  \chi^{2k\ell m}_{r,s} \Rcb \dy, \\
  (B)
 &= \iY \chi^{2ijn}_{p,q}  C_{pqab} \chi^{2k\ell m}_{a,b} \dy.
\end{align}
Then, expressing $F^{k\ell mn}(\bfchi^{1ij})$ using the right-hand side of~\eqref{zf:weak form chi3 index} produces
\begin{equation}
  (C)
 = -\iY \chi^{2k\ell m}_s C_{snpq}  \chi^{1ij}_{p,q}\dy
    + \iY \Scal_{pn}^{1k\ell m} \chi^{1ij}_p \dy
 = \iY \chi^{2k\ell m}_s \lpar C_{snij} - \Scal_{sn}^{0ij} \rpar \dy
    + \iY \Scal_{pn}^{1k\ell m} \chi^{1ij}_p\dy, \label{zf:varf chi31}
\end{equation}
having used the expression~\eqref{S0:def} of $\bSz$ in the first integral to obtain the second equality, while the expression~\eqref{zf:C2:comp} of $\brCii$ finally provides
\begin{equation}
  (D)
 = \iY C_{ijpn} \chi_{p}^{2 k\ell m} \dy - |\Ycal|\Cii_{ijnk\ell m}. \label{zf:aux09}
\end{equation}
We then substitute the above expressions for (A), (B), (C), (D) into~\eqref{zf:aux08} to obtain the equality
\begin{align}
\MoveEqLeft[3]
  \iY \Lpar \Scal_{na}^{0ij} \chi^{2k\ell m}_a - \chi^{1ij}_p \Scal_{pn}^{1k\ell m} \Rpar \dy
    + \iY \chi^{1ij}_p C_{pnam} \chi^{1k\ell}_a \dy
    + \iY \chi^{2k\ell m}_s \lpar C_{snij} - \Scal_{sn}^{0ij} \rpar \dy
	+ \iY \Scal_{pn}^{1k\ell m} \chi^{1ij}_p\dy \\
 &= \iY \chi^{2ijn}_{p,q}  C_{pqab} \chi^{2k\ell m}_{a,b} \dy
    + \iY C_{ijpn} \chi_{p}^{2k\ell m} \dy - |\Ycal|\Cii_{ijnk\ell m}.
\end{align}
Canceling identical integrals and rearranging remaining terms finally yields the sought expression~\eqref{zf:reciprocity 2} of $\brCii$.\enlargethispage*{1ex}

\subsection{Proof of Proposition~\ref{zf:coercivity}}
\label{sec:proof:coercivity:zf:fourier}

\proofstep{(a) Positivity of the potential energy density.}
We start by using $\brAZii(\vartheta) = \vartheta\brAth - \brCii$ (with $\brAth$ as in~\eqref{zf:A:def}), in the potential energy density given by~\eqref{eZ:def}, to obtain
\begin{equation}
  2\eZ_{\vartheta}(\bfeps,\bfeta)
 = \labs \brCz\dip\bfeps + \tdemi\eps\brCi\therefore\bfeta \rabs^2_{\brCz}
 + \eps^2\bfetaB\therefore\lpar \vartheta\brAth - \brCii - \tquart\brCi{}\Tsup\dip[\brCz]^{-1}\dip\brCi \rpar\therefore\bfeta
\end{equation}
(with the weighted norm $|\Cdot|_{\brCz}$ defined by $|\bfsig|^2_{\brCz}:=\bfsig\Hsup[\brCz]^{-1}\bfsig$). Then, let $\thZ_e:=\muZ\Max$, where $\muZ\Max$ is the largest eigenvalue of the symmetric generalized eigenvalue problem
\begin{equation}
  \lpar \brCii + \tquart\brCi{}\Tsup\dip[\brCz]^{-1}\dip\brCi \rpar\therefore\bfeta - \mu\brAth\therefore\bfeta = \bfze, \qquad
  \bfeta\in S^{2}(\Cbb^d)\tens\Cbb^d \label{e:eigenvalue:problem:zf}
\end{equation}
(which can be recast as a symmetric generalized eigenvalue problem on $d^2(d\shp1)/2$-vectors). We therefore have
$\bfetaB\therefore\lpar \thZ_e\brAth - \brCii - \tquart\brCi{}\Tsup\dip[\brCz]^{-1}\dip\brCi \rpar\therefore\bfeta\geq0$ for any $\bfeta$, and consequently
\begin{equation}
  2\eZ_{\vartheta}(\bfeps,\bfeta) \geq \labs \brCz\dip\bfeps + \tdemi\eps\brCi\therefore\bfeta \rabs^2_{\brCz} + (\vartheta\shm\thZ_e) \eps^2\bfetaB\therefore\brAth\therefore\bfeta 
\end{equation}
for any $(\bfeps,\bfeta)\not=\bfze$. For any $\vartheta>\thZ_e$, the quadratic form $(\bfeps,\bfeta)\mapsto \eZ_{\vartheta}(\bfeps,\bfeta)$ is therefore positive definite on the finite-dimensional space $S^{2}(\Cbb^d)\times \lpar S^{2}(\Cbb^d)\tens\Cbb^d\rpar$, which implies that there exists $C=C(\vartheta)>0$ such that
\begin{equation}
  \eZ_{\vartheta}(\bfeps,\bfeta) \geq C(\vartheta)\lpar |\bfeps|^2 + |\bfeta|^2 \rpar \label{eZ:elliptic}
\end{equation}
The threshold $\thZ_e$ clearly does not depend on $\eps$. The condition $\vartheta>\thZ_e$ is also necessary, since if $\vartheta\shleq\thZ_e$ we have $\eZ_{\vartheta}(\bfeps,\bfeta)\shleq0$ for $\bfeta$ an eigentensor for the eigenvalue $\mu=\muZ\Max$ and $\bfeps=-\tdemi[\brCz]^{-1}\dip(\brCi\therefore\bfeta)$.

\proofstep{(b) Positivity and $\brH^2(\Rbb^d)$-coercivity of $\rmQZth$.} 
Let $\bru\in\brH^2(\Rbb^d)$. Recalling definitions \eqref{PQ:CGE:bil:def} and \eqref{ke:SGE:def}, we have
 \begin{equation}
   \rmQZth(\bru,\bru) = \iRd 2\eZ_{\vartheta}(\bfeps[\bru],\bfeta[\bru]) \dx > 0
 \end{equation}
 and the integral is finite. Letting now $\bruH$ be the Fourier transform of $\bru$, we have
\begin{equation}
  \rmQZth(\bru,\bru)
 = \iRd 2\eZ_{\vartheta}(\widehat{\bfeps[\bru]},\widehat{\bfeta[\bru]}) \dxi
 \geq 2C(\vartheta) \iRd \lpar |\widehat{\bfeps[\bru]}|^2 + |\widehat{\bfeta[\bru]}|^2 \rpar \dxi \label{aux39}
\end{equation}
by Parseval's theorem and the ellipticity property~\eqref{eZ:elliptic}. Since $\widehat{\bfeps[\bru]}\sheq\bruH\stens(\rmi\bfxi)$ and $\widehat{\bfeta[\bru]}\sheq(\bruH\stens\rmi\bfxi)\tens(\rmi\bfxi)$, we find
\begin{equation}
\begin{gathered}
  |\widehat{\bfeps[\bru]}|^2 = |\bruH\stens\rmi\bfxi|^2 = \tdemi\lpar |\bruH|^2|\bfxi|^2 + |\bruh\sip\bfxi|^2 \rpar
  \geq \tdemi |\bruH|^2|\bfxi|^2 \\
  |\widehat{\bfeta[\bru]}|^2 = |(\bruH\stens\rmi\bfxi)\tens\bfxi|^2
  = \tdemi\lpar |\bruH|^2|\bfxi|^2 + |\bruh\sip\bfxi|^2 \rpar|\bfxi|^2 \geq  \tdemi |\bruH|^2|\bfxi|^4.
\end{gathered}
\label{aux41}
\end{equation}
We then use the above inequalities in~\eqref{aux39}, and invoke again Parseval's theorem, to obtain
\begin{align}
  \rmQZth(\bru,\bru)
 &\geq C(\vartheta) \iRd \lpar |\bruH|^2|\bfxi|^2 + |\bruH|^2|\bfxi|^4 \rpar \dxi \suite[-1ex]\quad
 = C(\vartheta) \lpar \|\bfna\bru\|^2_{\brL^2(\Rbb^d)} + \|\bfna^2\bru\|^2_{\brL^2(\Rbb^d)} \rpar
 = C(\vartheta) \lpar \|\bru\|^2_{\brH^2(\Rbb^d)} - \|\bru\|^2_{\brL^2(\Rbb^d)} \rpar,
\end{align}
which establishes the $\brH^2(\Rbb^d)$-coercivity of the bilinear form $\rmQZth$.

\subsection{Proof of Proposition~\ref{prop:recip}}
\label{sec:proof:recip}

\proofstep{Proof of identity~\eqref{bsh:expr}} Writing the reciprocity identity~\eqref{cell:generic:recip} for $\bfw_1=\bfchi^{1i\ell}$ and $\bfw_2=\bfzet^{2k}$, then recalling the right-hand sides~\eqref{weak form chi1 index} and~\eqref{weak form chi2 index} of the relevant variational cell problems and the definition~\eqref{phi1:def} of $\bfphc^1$, we have
\begin{equation}
  0 = - F^{i\ell}(\bfzet^{2k}) +  \FHat^k(\bfchi^{1i\ell})
  = \iY  C_{i\ell rs} \zeta^{2k}_{r,s} \dy + \iY (\beta\shm1) \chi_k^{1i\ell} \dy
  = \iY \lsqb \phi^1_{ik\ell} + \beta\chi^{1i\ell}_k - \beta\chi^{1k\ell}_i \rsqb \dy \label{aux22}
\end{equation}
which, in view of the definition~\eqref{bfbsh:def} of $\bfbsh$, establishes the claimed expression~\eqref{bsh:expr}.

\proofstep{Proof of identity~\eqref{bii:expr}} Still following the same general approach, we exploit the equality
\begin{equation}
  0 = -\FHat^{km}(\bfchi^{1i\ell}) + F^{i\ell}(\bfzet^{3km}) - F^{i\ell m}(\bfzet^{2k}) + \FHat^k(\bfchi^{2i\ell m}) \label{aux20}
\end{equation}
obtained by summing the reciprocity identities~\eqref{cell:generic:recip} written for $(\bfw_1,\bfw_2)=\lpar\bfchi^{1i\ell},\bfzet^{3km}\rpar$ and $(\bfw_1,\bfw_2)=\lpar\bfzet^{2k},\bfchi^{2i\ell m}\rpar$. This yields
\begin{align}
  0
 &= \iY (\phi^1_{rkm} - \bsh_{rkm} ) \chi^{1i\ell}_r \dy
  - \iY \zeta^{2k}_a C_{amrs} \chi^{1i\ell}_{r,s} \dy
  + \iY C_{i\ell ab} \zeta^{3km}_{a,b} \dy
  + \iY C_{i\ell mr} \zeta^{2k}_r \dy \suite
  + \iY \lpar C_{mrab}\chi^{1i\ell}_{a,b} - \CZ_{i\ell mr} \rpar \zeta^{2k}_r \dy
  - \iY \chi^{1i\ell}_q C_{qmrs} \zeta^{2k}_{r,s} \dy
    + \iY (1\shm\beta) \chi^{2i\ell m}_k \dy \\
 &= \iY \Lcb \phi^2_{ik\ell m} - \CZ_{i\ell mr} \zeta^{2k}_r - \beta\chi^{2mk\ell}_i + (1\shm\beta) \chi^{2i\ell m}_k
   - \bsh_{rkm} \chi^{1i\ell}_r + \beta \chi^{1i\ell}_r \chi^{1km}_r \Rcb \dy, \label{aux21}
\end{align}
with the second equality in~\eqref{aux21} obtained using~\eqref{phi2:def} to express $C_{i\ell ab} \zeta^{3km}_{a,b}+C_{i\ell mr} \zeta^{2k}_r$ in terms of $\phi^2_{ik\ell m}$, replacing $\phi^1_{rkm}$ with its expression~\eqref{phi1:def}, and canceling identical integrals. Then, since in addition $\zeta^{2k}_r$, $\chi^{2i\ell m}_k$ and $\chi^{1i\ell}_r$ have a zero mean, \eqref{aux21} provides
\begin{equation}
  \iY \phi^2_{ik\ell m} \dy = \iY \beta \lsqb \chi^{2mk\ell}_i + \chi^{2i\ell m}_k - \chi^{1i\ell}_r \chi^{1km}_r \rsqb \dy
\end{equation}
from which  the claimed expression~\eqref{bii:expr} of $\bfBii$ readily follows.

\proofstep{Proof of identity~\eqref{c2:recipr}} This equality directly results from using $\bfv=\bfzet^{2k}$ in the variational equation~\eqref{weak form chi2 index} satisfied by $\bfzet^{2i}$ and recalling the definition~\eqref{bfBii:def} of the tensor $\brcii$. Then, for any $\brV\shin\Cbb^d$, we have $\overline{\brV}{}\Tsup\sip\brcii\brV=A\lpar \bfzet^{2i}\overline{\rmV}_i,\, \bfzet^{2k}\rmV_k \rpar>0$ as the strain energy of the nonvanishing periodic cell function created by the zero-mean forcing $\bfg=(\beta\shm1)\brV$ with $\beta\neq0$, proving the positive definiteness of $\brcii$ unless the mass density is homogeneous.

\subsection{Proof of Proposition~\ref{prop:coercivity}}\mbox{}
\label{sec:proof:coercivity:fourier}

\proofstep{(a) Positivity of the strain energy density.} The proof given in Sec.~\ref{sec:proof:coercivity:zf:fourier} for $\eZ_{\vartheta}$ works for any tensor $\brCii\in S^2\lpar S^2(\Rbb^d)\tens\Rbb^d \rpar$, and hence applies to this case by replacing $\brCii$ with $-\brAii(0) = \brCii - \bfcii$, see~\eqref{A2:def}. Under this replacement, the eigenvalue problem~\eqref{e:eigenvalue:problem:zf} becomes
\begin{equation}
  \lpar \brCii - \bfcii + \tquart\brCi{}\Tsup\dip[\brCz]^{-1}\dip\brCi \rpar\therefore\bfeta - \mu\brAth\therefore\bfeta = \bfze, \qquad
  \bfeta\in S^{2}(\Cbb^d)\tens\Cbb^d \label{e:eigenvalue:problem}
\end{equation}
and setting $\vartheta_e=\mu\Max$, with $\mu\Max$ the largest eigenvalue of problem~\eqref{e:eigenvalue:problem}, yields the sought threshold $\vartheta_e$ that ensures
\begin{equation}
  e_{\vartheta}(\bfeps,\bfeta) \geq C(\vartheta)\lpar |\bfeps|^2 + |\bfeta|^2 \rpar > 0 \quad (\bfeps,\bfeta)\not=(\bfze,\bfze), \ \vartheta>\vartheta_e. \label{e:elliptic}
\end{equation}
Moreover, inequality~\eqref{eZ:e:compar} yields $e_{\vartheta} \geq \eZ_{\vartheta}$, so that the verification of the elastostatic inequality~\eqref{eZ:elliptic} implies that of the above inequality~\eqref{e:elliptic} for any $\vartheta\shgeq\thZ_e$. We must therefore have $\thZ_e \geq \vartheta_e$.

\proofstep{(b) Positivity of the kinetic energy density.} We reason as in item (a) of Section~\ref{sec:proof:coercivity:zf:fourier}. Writing $\brJii(\vartheta)= \vartheta\brI^2 - \brB^2 + \brb^2$ (with reference to~\eqref{A2:def}) in the kinetic energy density given by~\eqref{ek:def}, we obtain
\begin{align}
  2k_{\vartheta}(\bru,\bfga)
 &:= \rhoz \lcb |\bru|^2 + \eps\bru\sip\bfbsh\dip\bfga + \eps^2 \bfga\dip (\vartheta\brI^2 - \brB^2 + \brb^2 )\dip\bfga \rcb \\
 &= \rhoz \lcb \labs \bru + \tdemi\eps\bfbsh\dip\bfga \rabs^2 + \eps^2 \bfga\dip\lpar \vartheta\brI^2 - \brB^2 + \brb^2 - \tquart\bfbsh{}\Tsup\sip\bfbsh \rpar\dip\bfga \rcb.
\end{align}
Then, let $\vartheta_k:=\mu\Max$, where $\mu\Max$ is the largest eigenvalue of the symmetric eigenvalue problem
\begin{equation}
  \lsqb \brB^2 - \brb^2 + \tquart\bfbsh{}\Tsup\sip\bfbsh \rsqb\dip\bfga - \mu\bfga = \brze, \qquad
  \bfga\in\Rbb^{d\times d} \label{k:eigenvalue:problem}
\end{equation}
(which can be rewritten as a symmetric eigenvalue problem on $d^2$-vectors); we thus have
$\bfga\dip\lpar \vartheta_k\brI^2 - \brB^2 + \brb^2 - \tquart\bfbsh{}\Tsup\sip\bfbsh \rpar\dip\bfga\geq0$ for any $\bfga$, and consequently
\begin{equation}
  2k_{\vartheta}(\bru,\bfga)
 \geq \rhoz \lcb \labs \bru + \tdemi\eps\bfbsh\dip\bfga \rabs^2 + \eps^2(\vartheta\shm\vartheta_k)|\bfga|^2 \rcb 
\end{equation}
for any $(\bru,\bfga)\not=\bfze$. For any $\vartheta>\vartheta_k$, the quadratic form $(\bru,\bfga)\mapsto k_{\vartheta}(\bru,\bfga)$ is therefore positive definite on the finite-dimensional space $\Cbb^d\times(\Cbb^d\tens\Cbb^d)$, which implies that there exists $C=C(\vartheta)>0$ such that
\begin{equation}
  k_{\vartheta}(\bru,\bfga) \geq C(\vartheta)\lpar |\bru|^2 + |\bfga|^2 \rpar \label{k:elliptic}
\end{equation}
The threshold $\vartheta_k$ again does not depend on $\eps$, and as before the condition $\vartheta>\vartheta_k$ is also necessary. Definitions~\eqref{PQ:CGE:bil:def} with the replacements~\eqref{corr:tens} then imply that both bilinear forms $\rmQth$ and $\rmRth$ are positive.

\proofstep{(c) Positive definiteness of $\sfQth(\rmi\bfxi)$ and $\sfRth(\rmi\bfxi)$} Let $\brz\shin\Cbb^3$ and $\bfxi\shin\Rbb^3$. We have
\begin{align}
  \brzB\sip\sfQth(\rmi\bfxi)\brz
 &= -\brzB\sip\brCz\therefore(\rmi\bfxi\tens\brz\tens\rmi\bfxi)
 - \eps\brzB\sip\brCi\qip\lsqb\rmi\bfxi\tens\brz\tens(\rmi\bfxi)^{\otimes 2}\rsqb
 + \eps^2 \brzB\sip\brAZii(\vartheta)\dotp\lsqb (\rmi\bfxi)^{\otimes 2}\tens\brz\tens(\rmi\bfxi)^{\otimes 2} \rsqb \\
 &= (\overline{\brz\tens\rmi\bfxi})\dip\brCz\dip(\brz\tens\rmi\bfxi)
 + \tdemi\eps(\overline{\brz\tens\rmi\bfxi})\dip\brCi\therefore(\brz\tens\rmi\bfxi\tens\rmi\bfxi)
 + \tdemi\eps(\brz\tens\rmi\bfxi)\dip\brCi\therefore(\overline{\brz\tens\rmi\bfxi\tens\rmi\bfxi}) \suite
 + \eps^2 (\overline{\brz\tens\rmi\bfxi\tens\rmi\bfxi})\therefore\brAZii(\vartheta)\therefore(\brz\tens\rmi\bfxi\tens\rmi\bfxi) \\
 &= \overline{\brzB\sip\sfQth(\rmi\bfxi)\brz},
\label{aux28} \intertext{and}
\brzB\sip\sfRth(\rmi\bfxi)\brz
 &= \rhoz \lcb |\brz|^2 + \eps\brzB\sip\bfbsh\dip(\brz\tens\rmi\bfxi) - \eps^2 \brzB\sip\brJii(\vartheta)\therefore\lsqb\brz\tens(\rmi\bfxi)^{\otimes 2}\rsqb \rcb \\
 &= \rhoz \lcb |\brz|^2 + \tdemi\eps\brzB\sip\bfbsh\dip(\brz\tens\rmi\bfxi)
 + \tdemi\eps\brz\sip\bfbsh\dip(\overline{\brz\tens\rmi\bfxi})
 + \eps^2 (\overline{\brz\tens\rmi\bfxi})\dip\brJii(\vartheta)\dip(\brz\tens\rmi\bfxi) \rcb \\
 &= \overline{\brzB\sip\sfRth(\rmi\bfxi)\brz}, 
\label{aux43}
\end{align}
which shows that the matrices $\sfQth(\rmi\bfxi)\in\Cbb^{d\times d}$ and $\sfRth(\rmi\bfxi)\in\Cbb^{d\times d}$ are Hermitian. The second equality in~\eqref{aux28} (resp.~\eqref{aux43}) uses 
the skew-symmetry property~\eqref{K:skew} of $\brCi$ (resp. that~\eqref{bsh:sym} of $\bfbsh$).
Moreover, $\sfQth(\rmi\bfxi)$ and $\sfRth(\rmi\bfxi)$ are both positive definite since, by comparing~\eqref{aux28} to~\eqref{eZ:def} and~\eqref{aux43} to~\eqref{ek:def}, we find
\begin{equation}
  \brzB\sip\sfQth(\rmi\bfxi)\brz = 2\eZ_{\vartheta}\lpar \brz\tens\rmi\bfxi,\brz\tens(\rmi\bfxi)^{\otimes 2} \rpar > 0, \qquad
  \brzB\sip\sfRth(\rmi\bfxi)\brz = 2k_{\vartheta}(\brz,\brz\tens\rmi\bfxi) > 0.
\end{equation}

\proofstep{(d) 
$\brH^2(\Rbb^d)$-coercivity of $\rmQth$.} The proof is the same as that of Section~\ref{sec:proof:coercivity:zf:fourier} with $\brAZii(\vartheta)$ replaced by $\brAii(\vartheta)$.


\proofstep{(e) $\brH^1(\Rbb^d)$-ellipticity of $\rmRth$.} Let $\bru\in\brH^1(\Rbb^d)$.
Using again~\eqref{PQ:CGE:bil:def} with~\eqref{corr:tens}, expressing the result with the Fourier transform $\bruH$ of $\bru$ (invoking Parseval's theorem) and recalling inequality~\eqref{k:elliptic}, we have
\begin{equation}
  \rmRth(\bru,\bru)
 = \iRd 2k_{\vartheta}(\bruH,\bruH\tens\rmi\bfxi) \dxi
 \geq 2C(\vartheta) \iRd \lpar |\bruH|^2(1+|\bfxi|^2) \rpar \dxi = 2C(\vartheta)\|\bru\|^2_{\brH^1(\Rbb^d)} \geq 0,
\end{equation}
which establishes the $\brH^1(\Rbb^d)$-ellipticity of the bilinear form $\rmRth$.

\proofstep{(f) Threshold $\vartheta_0$.} Setting $\vartheta_0:=\max(\vartheta_e,\vartheta_k)$, all the properties claimed in Proposition~\ref{prop:coercivity} hold for any $\vartheta>\vartheta_0$.

\subsection{Proof of Proposition~\ref{prop:evol} (well-posedness of transient effective wave equation)}
\label{sec:transient_equation:proof}


In this section, we establish the well-posedness results claimed in Proposition~\ref{prop:evol} for the evolution problem~\eqref{IVP} in $\Rbb^d$.
%
We first consider problem~\eqref{IVP} for the homogeneous effective PDE, i.e. with $\brF\sheq\bfze$.
The spaces $\Hcal$ and $\Dcal(\Acal)$ are respectively equipped with the scalar products
\begin{equation}
\begin{aligned}
  \lpar \bsfU_1,\bsfU_2 \rpar_{\Hcal}
 &:= \rmQth(\brU_1,\brU_2) + \rmRth(\brU_1,\brU_2) + \rmRth(\brV_1,\brV_2), \\
  \lpar \bsfU_1,\bsfU_2 \rpar_{\Dcal(\Acal)}
 &:= \lpar \bsfU_1,\bsfU_2 \rpar_{\Hcal} + \lpar \Acal\bsfU_1,\Acal\bsfU_2 \rpar_{\Hcal}.
\end{aligned}
\label{H:scalarp}
\end{equation}
By~\eqref{AB:cont} and Proposition~\ref{prop:coercivity}, the resulting norm $\|\dotp\|_{\Hcal}$ is equivalent to the standard $H^2\shtimes H^1$ norm.

\subsubsection*{Application of the Hille-Yosida theorem.} The solvability of the initial-value problem~\eqref{IVP:order1} can be decided for any initial datum $\bsfU_0$ having appropriate regularity (to be specified later) by checking that it satisfies the conditions of the Hille-Yosida theorem~\cite{goldstein,brezis:11}. In the present context, we need to verify that there exists $\lambda\in\Rbb^d$ such that $\Acal_{\lambda}:=\Acal+\lambda\Ical: D(\Acal)\to\Hcal$ is \emph{maximal monotone}, i.e. satisfies
\begin{equation}
\begin{aligned}
  \lpar \Acal_{\lambda}\bsfU,\bsfU \rpar_{\Hcal} &\geq 0 \text{ \ for any }\bsfU\in D(\Acal) &\qquad& \text{($\Acal_{\lambda}$ monotone),} \\
  \text{for any }\bsfF\in\Hcal, &\text{ there exists }\bsfU\in D(\Acal) \text{ such that } (\Acal_{\lambda}+\Ical)\bsfU = \bsfF && \text{($\Acal_{\lambda}+\Ical$ surjective).}
\end{aligned} \label{maxmon}
\end{equation}
We now show that conditions~\eqref{maxmon} are indeed satisfied:

\proofstep{1. Monotonicity.} Let $\lambda>0$ and $\bsfU\in D(\Acal)$. Recalling definitions~\eqref{IVP:order1} and~\eqref{H:scalarp}, we have
\begin{align}
  \lpar \Acal\bsfU,\bsfU \rpar_{\Hcal}
 &= -\rmQth(\brV,\brU) -\rmRth(\brV,\brU) + \rmRth(\sfRth^{-1}\sfQth\brU,\brV) \suite
  \stackrel{(a)}{=} -\rmQth(\brV,\brU) -\rmRth(\brV,\brU) + (\sfQth\brU,\brV)
  \stackrel{(b)}{=} -\rmRth(\brV,\brU), \label{aux11a}
\end{align}
where step (a) uses the definition~\eqref{Binv:def} of $\sfRth^{-1}$ and step (b) the Green identity~(\ref{green:AB}a), and then
\begin{align}
  \lpar \Acal\bsfU,\bsfU \rpar_{\Hcal} + \lambda\lpar \bsfU,\bsfU \rpar_{\Hcal}
  &= -\rmRth(\brV,\brU)
 + \lambda \rmQth(\brU,\brU) + \lambda \rmRth(\brU,\brU) + \lambda \rmRth(\brV,\brV) \\
  &= \lambda \rmQth(\brU,\brU)
  + \Lpar \lambda-\frac{1}{2} \Rpar \big[ \rmRth(\brU,\brU) + \rmRth(\brV,\brV) \big]
  + \frac{1}{2} \rmRth(\brU-\brV,\brU-\brV). \label{aux11}
\end{align}
Consequently, for any $\lambda>1/2$, the operator $\Acal_{\lambda}:=\Acal+\lambda\Ical$ satisfies the monotonicity requirement in~\eqref{maxmon}.

\proofstep{2. Surjectivity.} We now consider the solvability for $\bsfU\in D(\Acal)$ of
\begin{equation}
  \lpar \mu\Ical + \Acal \rpar\bsfU = \bsfF, \quad\text{i.e.}\quad\bigg\{\
  \begin{aligned}
    \text{(a) \ } & \mu\brU - \brV &&= \textbf{f} \\ \text{(b) \ } & \mu\sfRth\brV + \sfQth\brU &&= \sfRth\brg \label{maxmon:eqs}
  \end{aligned}
\end{equation}
for given $\bsfF=(\textbf{f},\brg)^{\text{T}}\in\Hcal$ and for some $\mu\in\Rbb^d,\,\mu\not=0$. Using (a) to eliminate $\brV$ in (b), we are left to determine the solvability of
\begin{equation}
  \mu^2\sfRth\brU + \sfQth\brU = \sfRth\lpar \brg + \mu\textbf{f} \rpar.
\end{equation}
The above problem has the variational formulation
\[
  \mu^2 \rmRth(\brU,\brW) + \rmQth(\brU,\brW) = \rmRth\lpar \brg + \mu\textbf{f},\brW \rpar \qquad
  \text{for all } \brW\in \brH^2(\Rbb^d).
\]
As the bilinear form $(\brU,\brW)\mapsto \mu^2 \rmRth(\brU,\brW) + \rmQth(\brU,\brW)$ is, by Proposition~\ref{prop:coercivity}, $\brH^2(\Rbb^d)$-elliptic, the above variational problem has a unique solution $\brU\in \brH^2(\Rbb^d)$ by the Lax-Milgram lemma.
Equation (a) in~\eqref{maxmon:eqs} then gives $\brV\in \brH^2(\Rbb^d)$. Finally, equation (b) in~\eqref{maxmon:eqs} yields $\sfQth\brU = \sfRth(\brg-\mu\brV)\in \brH^{-1}(\Rbb^d)$, so that we have $(\brU,\brV)\in D(\Acal)$. Consequently, for any $\mu\not=0$, the equation $(\Acal+\mu\Ical)\bsfU=\bsfF$ is solvable for $\bsfU\in D(\Acal)$ for any given $\bsfF\in\Hcal$.

\proofstep{3. Proof conclusion.} Choosing $\mu=\lambda+1$, items 1 and 2 above show that $\Acal_{\lambda}:=\Acal+\lambda\Ical:D(\Acal)\to\Hcal$ is maximal monotone. The Hille-Yosida theorem hence applies to the initial-value problem~\eqref{IVP:order1} where $\brF\sheq\bfze$.

\subsubsection*{Forced response evolution problem}

Consider now the generic forced-response evolution problem
\begin{equation}
  \bsfU'(t) + \Acal \bsfU(t) = \bsfF , \qquad \bsfU(0) = \bfze \label{FRP:order1}
\end{equation}
with initial rest.
Then, the maximal monotone character of $\Acal$ allows to show, via an adaptation of the variation-of-constant method, that problem~\eqref{FRP:order1} has a unique solution $\bsfU\in C^1([0,T];\Hcal) \cap C^0([0,T];D(\Acal))$ if either $\bsfF\in C^1([0,T];\Hcal)$ or $\bsfF\in C^0([0,T];D(\Acal))$. Applied to the right-hand side of the evolution problem~\eqref{IVP}, i.e. to $\bsfF$ given by~\eqref{IVP:order1}, the latter condition is met for any $\brF$ satisfying the requirement stated in the proposition, by virtue of the variational definition~\eqref{Binv:def} of $\sfRth^{-1}$.


\subsubsection*{Energy identities}



Let $\bsfU(t) \in C^1([0,T];\Hcal)$ solve the evolution problem~\eqref{IVP} with data $\bsfU_0,\bsfF$ satisfying the assumptions of Proposition~\ref{prop:evol}.
Using the differential equation~\eqref{FRP:order1} with~\eqref{aux11a} at $\tau\in[0,T]$, we have
\begin{equation}
  \lpar \bsfU'(\tau),\bsfU(\tau) \rpar_{\Hcal} + \lpar \Acal \bsfU(t) ,\, \bsfU(t) \rpar_{\Hcal}
 = \lpar \bsfU'(\tau),\bsfU(\tau) \rpar_{\Hcal} - \rmRth\lpar \brU'(\tau),\brU(\tau) \rpar
 = \lpar \bsfF(\tau),\, \bsfU(\tau) \rpar_{\Hcal}. \label{aux10}
\end{equation}

For problem~\eqref{IVP} with $\brF\sheq\bfze$, integrating~\eqref{aux10} over $\tau\in[0,t]$ for some $t\in[0,T]$ and recalling the initial-rest conditions yields the claimed energy conservation for this case, since we obtain
\begin{equation}
  E(t) = E(0), \qquad  E(t):= \frac{1}{2} \lpar \|\bsfU\|_{\Hcal}^2(t) - \rmRth(\brU(\tau),\brU(t)) \rpar
 = \tdemi\rmQth(\brU(t),\brU(t)) + \tdemi\rmRth(\brU'(t),\brU'(t)). \label{eq:energy}
\end{equation}

For the non-homogeneous effective PDE, applying Young's inequality to the second term in the right-hand side of~\eqref{aux10} provides the inequality
\[
  2\lpar \bsfU'(\tau),\bsfU(\tau) \rpar_{\Hcal}
 \leq \| \bsfF(\tau) \|^2_{\Hcal} + \| \bsfU(\tau) \|^2_{\Hcal} + \rmRth\lpar \brU(\tau),\brU(\tau) \rpar
 + \rmRth\lpar \brV(\tau),\brV(\tau) \rpar
 \leq \| \bsfF(\tau) \|^2_{\Hcal} + 2\| \bsfU(\tau) \|^2_{\Hcal}.
\]
Then, integrating over $\tau\in[0,t]$ for some $t\in[0,T]$ and recalling the initial-rest condition yields the inequality
\begin{equation}
  \|\bsfU\|_{\Hcal}^2(t)
 \leq 2 \int_{0}^{t} \| \bsfU(\tau) \|^2_{\Hcal} \dtau + \|\bsfU_0\|_{\Hcal}^2 + \| \bsfF \|^2_{L^2([0,T];\Hcal)},
\end{equation}
from which we obtain the claimed estimate
\begin{equation}
  \|\bsfU\|_{\Hcal}^2(t) \leq C(T) \lpar \|\bsfU_0\|_{\Hcal}^2 + \| \bsfF \|^2_{L^2([0,T];\Hcal)} \rpar \quad 
  \text{with}\quad C(T)=\exp(2T), \qquad t\in[0,T],
\label{energy:forced}
\end{equation}
by invoking Gronwall's lemma~\cite[Chap. 18]{dautray:lions:5}, recalled for convenience:
\begin{lemma}[Gronwall]\label{gronwall}

Let $\Phi\in L^{\infty}([0,T])$ verify $\Phi(t)\geq0$ a.e. in $[0,T]$, and assume that $\Phi(t) \leq C_1\int_0^t \Phi(\tau) \dtau + C_2$ holds a.e. in $[0,T]$ for some constants $C_1,C_2\geq0$. Then: $\Phi(t) \leq C_2\exp(C_1 t), \ t\shin[0,T]$.
\end{lemma}

\subsection{Proof of Proposition~\ref{prop:disp:expansion}}

Seeking a $\Ocal(\delta^2)$ expansion of the eigenvalues $\mu=\mu(\delta)$ solving problem~\eqref{planewave:eig} with~\eqref{planewave:expand}, we insert the expansions
\begin{equation}
  \brV(\delta) = \brV_0 + \delta\brV_1 + \delta^2\brV_2 + o(\delta^2), \qquad
  \mu(\delta) = \mu_0 + \delta\mu_1 + \delta^2\mu_2 + o(\delta^2)
\end{equation}
in~\eqref{planewave:eig}; this produces the following $\Ocal(\delta^0)$,  $\Ocal(\delta^1)$,  $\Ocal(\delta^2)$ equalities:
\begin{equation}
\begin{aligned}
  \text{(a) \ }\brze
 &= \lpar \sfQ^0 \shm \mu_0\sfR^0 \rpar\sip\brV_0, \\[-0.5ex]
  \text{(b) \ }\brze
 &=  \lpar \sfQ^0 \shm \mu_0\sfR^0 \rpar\sip\brV_1 + \rmi \lpar \sfQ^1 \shm \mu_0\sfR^1 \rpar\sip\brV_0 - \mu_1\sfR^0\sip\brV_0,  \\[-0.5ex]
  \text{(c) \ }\brze
 &=  \lpar \sfQ^0 \shm \mu_0\sfR^0 \rpar\sip\brV_2 + \rmi \lpar \sfQ^1 - \mu_0\sfR^1 \rpar\sip\brV_1
  + \lpar \sfQ^2_{\vartheta} \shm \mu_0\sfR^2_{\vartheta} \rpar\sip\brV_0 - \mu_1 \lpar \sfR^0\sip\brV_1 + \rmi\sfR^1\sip\brV_0 \rpar \shm \mu_2\sfR^0\sip\brV_0
  \label{planewave:expand:series}
\end{aligned}
\end{equation}
Equality (a) yields the eigenpairs $(\brV_0,\mu_0)$ solving the standard Christoffel equation, the eigenvalue $\mu_0\shg0$ being either single or double.

\subsubsection*{Single eigenvalue $\mu_0$} In this case, $\brV_0$ can be assumed to be real and to verify $\brV_0\sip\sfR^0\sip\brV_0=1$. Then, left-contracting (b) by $\brV_0$ and recalling the symmetric eigenvalue equation (a) and the skew-symmetry of $\sfQ^1,\,\sfR^1$, we obtain
\begin{equation}
  \mu_1 = \rmi \brV_0 \sip \lpar \sfQ^1 - \mu_0\sfR^1 \rpar\sip\brV_0 = 0,
\end{equation}
Finally, on writing the identity $\brV_0\dotp\text{(c)}+\brV_1\dotp\text{(b)}$ and recalling $\mu_1=0$ and (again) the skew-symmetry of $\sfQ^1,\,\sfR^1$, 
The next eigenvalue correction $\mu_2$ is found to be real and given by
\begin{equation}
  \mu_2 = \brV_0\sip\lpar \sfQ^2_{\vartheta} - \mu_0\sfR^2_{\vartheta} \rpar\sip\brV_0
  + \brV_1\sip\lpar \sfQ^0 - \mu_0\sfR^0 \rpar\sip\brV_1;
\end{equation}
it is completely determined once equality (b) is solved for $\brV_1$ constrained by $\brV_1\sip\sfR^0\sip\brV_0=0$. Now, recalling the expressions~\eqref{planewave:matrices} of $\sfQ^2_{\vartheta}$ and $\sfR^2_{\vartheta}$, we find that
\begin{equation}
  \brV_0\sip\lpar \sfQ^2_{\vartheta} - \mu_0\sfR^2_{\vartheta} \rpar\sip\brV_0
 = - (\mu_0)^2\brV_0\sip\brCii\sip\brV_0 - \brV_0\sip\lpar \brC^2 - \mu_0\brB^2 \rpar\sip\brV_0,
\end{equation}
making $\mu_2$ independent on $\vartheta$. As expected, the weighting parameter $\vartheta$ affects only the $o(\delta^2)$ remainder of $\mu(\delta)$.

\subsubsection*{Double eigenvalue $\mu_0$} The eigenspace of $\mu_0$ is in this case generated by a pair of unit real orthonormal vectors $\brV_{0a},\brV_{0b}$. We may seek $\brV_0$ as $\brV_0=\alpha\brV_{0a}\shp\beta\brV_{0b}$ in (b), whereupon left-contracting that equation by $\brV\Tsup_{0a,b}$ shows that $(\alpha,\beta)$ and $\mu_1$ obey the 2-dimensional eigenvalue problem
\begin{equation}
  \left\{ \begin{aligned} \rmi p\beta-\mu_1\alpha &= 0 \\[-0.5ex] -\rmi p\alpha-\mu_1\beta &= 0 \end{aligned} \right., \qquad 
  p:=\brV\Tsup_{0a} \sip\lpar \sfQ^1 - \mu_0\sfR^1 \rpar\sip\brV_{0b}\in\Rbb,
\end{equation}
whose eigenpairs are $(\alpha,\beta)_{\pm}\sheq(\rmi,\pm1)/\sqrt{2},\,\mu_{1\pm}\sheq \pm p$.
Each pair $(\alpha,\beta)_{\pm}$ defines the leading-order term $\brV_{0\pm}$ of the eigenvector paths $\brV_{\pm}(\delta)$ intersecting at $\delta\sheq0$. Proceeding as in the single-multiplicity case, we left-contract (c) by $\overline{\brV}{}^0$ (with $\brV_0=\brV_{0\pm}$) and (b) by $\overline{\brV}{}^1=\overline{\brV}{}^1_{\pm}$, to obtain
\begin{align}
  \mu_2 &= \rmi \overline{\brV}{}^0\sip\lpar \sfQ^1 - \mu_0\sfR^1 \rpar\sip\brV_1
  - \mu_1\overline{\brV}{}^0\sip \lpar \sfR^0\sip\brV_1 + \rmi\sfR^1\sip\brV_0 \rpar
  + \overline{\brV}{}^0\sip\lpar \sfQ^2_{\vartheta} - \mu_0\sfR^2_{\vartheta} \rpar\sip\brV_0, \\
  0 &= \overline{\brV}{}^1\sip\lpar \sfQ^0 - \mu_0\sfR^0 \rpar\sip\brV_1
  + \rmi \overline{\brV}{}^1\sip\lpar \sfQ^1 - \mu_0\sfR^1 \rpar\sip\brV_0
   - \mu_1 \overline{\brV}{}^1\sip\sfR^0\sip\brV_0.
\end{align}
Summing the latter then yields
\begin{equation}
  \mu_2 = 2\Re\Lcb \overline{\brV}{}^0\sip\lpar \rmi \sfQ^1 - \mu_0\rmi \sfR^1 - \mu_1\sfR^0 \rpar\sip\brV_1 \Rcb
  + \overline{\brV}{}^0\sip\lpar \sfQ^2_{\vartheta} - \mu_0\sfR^2_{\vartheta} - \mu_1\rmi\sfR^1 \rpar\sip\brV_0
  + \overline{\brV}{}^1\sip\lpar \sfQ^0 - \mu_0\sfR^0 \rpar\sip\brV_1
\end{equation}
which shows that, as before, the eigenvalue corrections $\mu_{2\pm}$ are real and independent on $\vartheta$.

\section{Conclusion and outlook}
\label{sec:conclusion_perspectives}

Concluding, in this work the ``classical'' two-scale homogenization framework (that uses the field equations rather than
variational approaches) is revisited. Thanks to reciprocity identities, we show that one can build a
strain-gradient-like model, with appropriate symmetries for the constitutive tensors, by solving only (i) two
families of ``static''  cell problems depending on the elasticity of the cell and (ii) an additional family of ``inertial'' cell problems depending on both elasticity and mass density used in the inertial part of the effective PDE. Sufficient conditions for this SGE model to be well-posed are then established, involving the positivity of generalized strain and kinetic energies. An original ``Boussinesq trick'' is proposed to satisfy these conditions while keeping a second-order asymptotic
accuracy.

Preliminary numerical investigations confirm (i) the ill-posedness of some of the non-corrected homogenized models,
and therefore the need for a correction, (ii) the second-order asymptotic accuracy of both non-corrected and corrected models on dispersion curves (although the correction implies a loss of accuracy), and finally (iii) the ability to the SGE model to reproduce important propagation features such as 6-fold
anisotropic propagation. 

An appealing perspective of this work is to compare the proposed ``trick'' to existing alternatives, and to enrich our methodology to build other families of well-posed models that embed as much dispersive behavior as possible. Among the many directions to be explored, let us quote (i) the energy-based, Cholesky-inspired proposition of \cite{thbaut:25} to obtain a model associated with a positive energy in elastostatics, (ii) the spectral decomposition of the sixth-order tensor $\brA$ to keep only the eigenspaces associated
with positive eigenvalues in \cite{durand:22}, still for elastostatics, (iii) the second-order-in-time "filtered" models of \cite{allaire:24} and (iv) the use of a fourth-order in time wave equation, as the one derived in Section \ref{sec:fourth:order:in:time}: if this solution is proved to produce well-posed hyperbolic systems as in 1D \cite{corn:lomb:23}, it could be further explored by deriving a family of such models thanks to alternative "tricks".

Another perspective is to pursue the numerical exploration and exploitation of these models. A quantitative study on the convergence of our asymptotic model in terms of predicted fields (and not only dispersion relations), completing the present examples, will require a well-controlled numerical set-up, careful handling of initial conditions or source terms, and a way to add oscillating correctors to the strain-gradient solution for pointwise comparison with the microstructured fields. As another numerically demanding perspective, we also plan to investigate the effective behavior of 3D cells in order to continue the exploration of wave propagation in non-centrosymmetric materials \cite{Rosi2020, rosi:24b}, and strengthen our understanding on the influence of the odd-order tensors established in the course of this study. 

To go beyond known microstructures such as gyroids in 3D, \emph{topological optimization} algorithms could be used to optimize dispersive properties described by the homogenized tensors, even without correction: we
would not aim at using the model but only use it as a descriptor of the media, as in \eg \cite{allaire:18,corn:bell:20}. To explore the multiple configurations required by these algorithms, efficient
FFT-based solvers are available to compute cell solutions \cite{tran:12,corn:bell:20} without recourse to finite elements.

Finally, a challenging but necessary development for homogenized strain-gradient models to be used in finite domains is to derive adapted \emph{boundary and interface conditions}. These conditions should (i) account for well-known boundary or interface layers \cite{dumontet:86,fergoug:22a}, (ii) include \emph{additional conditions} to complete the fourth-order-in-space strain-gradient equation, compared to a microstructured elastic problem featuring standard conditions on the displacement or normal stress, and (iii) ensure that the asymptotic order of approximation achieved by the volumic model (second-order for SGE) is preserved. So far this problem of finding such adapted, asymptotically-preserving conditions has been addressed in 1D in statics \cite{thbaut:26} and for waves \cite{corn:lomb:23} (which uses a modified equation deleting the strain-gradient term), but to our best knowledge remains open in higher dimensions, although results exist for scalar (e.g. acoustic) waves \cite[Chap. 4]{FlissHDR2019}.


\bibliographystyle{amsplain}
\bibliography{marcbibs,files/otherbibs}

\end{document}